\documentclass[aap,reqno]{imsart}
\pdfoutput=1
\setattribute{journal}{name}{}
\usepackage{a4wide}
\usepackage{amssymb}
\usepackage{mathrsfs}
\usepackage{amsmath,amsfonts}
\usepackage{amsthm}
\usepackage[usenames,dvipsnames]{color}
\usepackage{thmtools,thm-restate}
\usepackage{appendix}
\usepackage{wrapfig}

\usepackage{blkarray}

\usepackage{float}
\usepackage{bbm}
\usepackage{bm}
\usepackage{enumerate}
\usepackage{dsfont} 
\usepackage{graphicx}
\usepackage[utf8]{inputenc}
\usepackage[square,numbers]{natbib}
\usepackage{url}
\usepackage{comment}
\usepackage{multicol}
\usepackage{algorithm}
\usepackage{algorithmic}

\usepackage{hyperref}
\usepackage{scalerel}
\usepackage[shortlabels]{enumitem}
\usepackage{verbatim}

\usepackage{caption} 
\usepackage{subcaption}

\usepackage{calc}
\usepackage{tikz}
\usetikzlibrary{arrows, automata}
\usetikzlibrary{positioning}
\usepackage{tikz-3dplot}
\usetikzlibrary{calc,tikzmark}
\usetikzlibrary{decorations.markings}
\tikzstyle{vertex}=[circle, draw, inner sep=0pt, minimum size=6pt]

\numberwithin{equation}{section}

\theoremstyle{plain} 
\newtheorem{theorem}{Theorem}[section]
\newtheorem{lemma}[theorem]{Lemma}
\newtheorem{proposition}[theorem]{Proposition}
\newtheorem{corollary}[theorem]{Corollary}
\newtheorem{definition}[theorem]{Definition}
\newtheorem{example}[theorem]{Example}
\newtheorem{remark}[theorem]{Remark}

\newcommand{\bthe}{\begin{theorem}}
\newcommand{\ethe}{\end{theorem}}

\newcommand{\ben}{\begin{enumerate}}
\newcommand{\een}{\end{enumerate}}

\newcommand{\bit}{\begin{itemize}}
\newcommand{\eit}{\end{itemize}}

\newcommand{\beq}{\begin{equation}}
\newcommand{\eeq}{\end{equation}}

\newcommand{\ble}{\begin{lemma}}
\newcommand{\ele}{\end{lemma}}

\newcommand{\bde}{\begin{definition}}
\newcommand{\ede}{\end{definition}}

\newcommand{\bco}{\begin{corollary}}
\newcommand{\eco}{\end{corollary}}

\newcommand{\bpr}{\begin{proposition}}
\newcommand{\epr}{\end{proposition}}

\newcommand{\brem}{\begin{remark}\rm}
\newcommand{\erem}{\end{remark}}

\newcommand{\bproof}{\begin{proof}}
\newcommand{\eproof}{\end{proof}}

\newcommand{\bexam}{\begin{example}\rm}
\newcommand{\eexam}{\end{example}}

\newcommand{\beao}{\begin{eqnarray*}}
\newcommand{\eeao}{\end{eqnarray*}\noindent}

\newcommand{\balo}{\begin{align*}}
\newcommand{\ealo}{\end{align*}}

\newcommand{\balm}{\begin{align}}
\newcommand{\ealm}{\end{align}\noindent}

\newcommand{\beam}{\begin{eqnarray}}
\newcommand{\eeam}{\end{eqnarray}\noindent}

\newcommand{\barr}{\begin{array}}
\newcommand{\earr}{\end{array}}

\newcommand{\N}{\mathbb{N}}
\renewcommand\P{\mathbb{P}}
\newcommand{\R}{\mathbb{R}}

\def\bC{\mathbb{C}}

\def\bE{\mathbf E}

\newcommand{\D}{\mathcal{D}}

\def\cald{\mathcal{D}}

\def\supp{{\rm supp}}
\newcommand{\sgn}{\text{sgn} }
\DeclareMathOperator*{\argmax}{arg\,max}
\DeclareMathOperator*{\argmin}{arg\,min}

\def\bone{\boldsymbol 1}
\def\bA{\boldsymbol A}
\def\bB{\boldsymbol B}
\def\bC{\boldsymbol C}

\def\bE{\boldsymbol E}
\def\bF{\boldsymbol F}
\def\bG{\boldsymbol G}
\def\bI{\boldsymbol I}

\def\bU{\boldsymbol U}

\def\bX{\boldsymbol X}

\def\bZ{\boldsymbol Z}

\newcommand{\nto}{{n\to\infty}}

\newcommand{\al}{{\alpha}}
\newcommand{\la}{{\lambda}}

\newcommand{\eps}{\varepsilon}

\newcommand{\ov}{\overline}

\newcommand{\wt}{\widetilde}

\newcommand{\pa}{{\rm pa} } 
\newcommand{\de}{{\rm de} } 
\newcommand{\an}{{\rm an} } 
\newcommand{\An}{{\rm An} }

\newcommand{\halmos}{\quad\hfill\mbox{$\Box$}} 

\allowdisplaybreaks 
\definecolor{darkred}{RGB}{139,0,0}
\definecolor{darkgreen}{RGB}{0,139,0}

\endlocaldefs

\begin{document}
\begin{frontmatter}

\title{Recursive max-linear models with propagating noise}
\runtitle{Recursive max-linear models with propagating noise}

\author{\fnms{Johannes} \snm{Buck} 
\ead[label=e1]{j.buck@tum.de}}
\address{Center for Mathematical Sciences\\Technical University of Munich\\ Boltzmanstrasse 3\\ 85748 Garching, Germany\\  \printead{e1}}
\and
  \author{\fnms{Claudia} \snm{Kl\"uppelberg} \ead[label=e2]{cklu@ma.tum.de}}
  \address{Center for Mathematical Sciences\\Technical University of Munich\\ Boltzmanstrasse 3\\ 85748 Garching, Germany\\  \printead{e2}}
 \affiliation{Technical University of Munich}
 
\runauthor{J. Buck and C. Kl\"uppelberg}

\begin{abstract}
Recursive max-linear vectors model causal dependence between node variables by a structural equation model, expressing each node variable as a max-linear function of its parental nodes in a directed acyclic graph (DAG) and some exogenous innovation. 
For such a model, there exists a unique minimum DAG,
represented by the Kleene star matrix of its edge weight matrix, 
which identifies the model and can be estimated. 
For a more realistic statistical modeling we introduce some random observational noise. 
A probabilistic analysis of this new noisy model reveals that the unique minimum DAG representing the distribution of the non-noisy model remains unchanged and identifiable. Moreover, the distribution of the minimum ratio estimators of the model parameters at their left limits are completely determined by the distribution of the noise variables up to a positive constant. 
Under a regular variation condition on the noise variables we prove that the estimated Kleene star matrix converges to a matrix of independent Weibull entries after proper centering and scaling.  
\end{abstract}



\end{frontmatter}

\section{Introduction}\label{s1}

Graphical modeling has shown to be a powerful tool for understanding causal dependencies in a multivariate random vector. However, most models have been linear and limited to discrete or Gaussian distributions (see e.g. \cite{KF} and \cite{lau}).
Such models lead to severe underestimation of large risks and, therefore, are not suitable in the context of extreme risk assessment.
First examples combining extreme value methods with graphical models include  flooding in river networks (\cite{engelke:hitz:18}), financial risk (\cite{einmahl2016}, \cite{krali}), and nutrients (\cite{krali}).

We consider the class of recursive max-linear (ML) models, which has been defined in \cite{nadine1}. A recursive ML model is defined by a structural equation model (SEM) of the form
\begin{align}\label{nonnoise}
	X_i= \bigvee_{j \in \pa(i)}c_{ji} X_j \lor Z_i, \quad i=1,\ldots,d
\end{align}
where the dependence structure between random variables is represented by a DAG $\D:=(V,E)$ with node set $V:=\{1,\ldots,d\}$ and edge set $E=E(\D)\subseteq V\times V$, and each variable $X_i$ for $i\in V$ has a representation in terms of ML functions of its parental nodes $\pa(i)=\{j\in V: (j,i)\in E\}$ and an independent innovation $Z_i$. 

Both, SEMs (e.g. \cite{Bollen}, \cite{pearl}) and directed graphical models (e.g. \cite{KF}, \cite{lau}, \cite{spirt}) are well-established and widely used to understand causality.

ML models similar to \eqref{nonnoise} have been proposed and studied in a time series context (e.g. \cite{davis:resnick}), in terms of moving maxima processes (e.g. \cite{hall:2002}), or as tropical models in algebra (e.g. \cite{joswig:20}, \cite{MS})  with applications to various optimization problems (e.g. \cite{BCOQ}, \cite{butkovic}, \cite{ngoc_auction}).

As shown in \cite{KL2017} recursive ML models respect the basic Markov properties associated with DAGs (e.g.  \cite{lauritzen:01},\cite{Lauritzen1990}). 
Moreover, the equation system \eqref{nonnoise} has the solution
\begin{align*}
	X_i= \bigvee_{j \in \pa(i)}b_{ji} Z_j , \quad i=1,\ldots,d,
\end{align*}
with ML coefficient matrix (in tropical algebra called the Kleene star matrix) $\bB:=(b_{ij})_{d \times d}$, see \cite{butkovic}, Corollary 1.6.16. 
Unlike the edge weight matrix $\bC=(c_{ij})_{d \times d}$, $\bB$ is identifiable and completely determines the distribution of $\bX:=(X_1,\ldots, X_d)$ (see \cite{nadine2}, Theorem 1). Also, $\bB$ is idempotent with respect to the tropical matrix multiplication defined in \eqref{ch2:odot} below, and defines a graphical model on a DAG. Furthermore, \cite{nadine2} proposes a minimum ratio estimator for $\bB$, which itself is idempotent, and is a generalized maximum likelihood estimator in the sense of \cite{gmle}.

We extend the original model \eqref{nonnoise} by allowing for multiplicative 
observation errors and define 
\begin{align}\label{ext}
	U_i= \Big(\bigvee_{j \in \pa(i)}c_{ji} U_j \lor Z_i\Big)\varepsilon_i, \quad i=1,\ldots,d,
\end{align}
with $\varepsilon_i \geq 1$ and iid for $i=1,\ldots,d$. 
By taking advantage of tropical algebra, we present in Theorem~\ref{solution} a solution of \eqref{ext} which represents each node variable $U_i$ in terms of a ML function of its ancestral nodes and an independent innovation $Z_i$ given by
\begin{align*}
    U_i= \bigvee_{j \in \an(i)\cup\{i\}}\bar b_{ji} Z_j,\quad i=1,\ldots,d,
\end{align*}
where $\an(i)$ denotes the ancestors of $i$ and $\bar b_{ji}$ are random variables involving the edge weights and the noise variables. 

It comes as no suprise that the true DAG and edge weights for a recursive ML model with propagating noise inherit the non-identifiability property from the non-noisy model. 
However, as we will prove in section~\ref{s4}, the ML coefficient matrix $\bB=(b_{ij})_{d\times d}$
remains identifiable in spite of the observational noise and even if we do not know the underlying DAG. 

To link up our new model \eqref{ext} with existing literature, observe that a log-transformation of \eqref{ext} yields 
\begin{align}\label{logtrans}
	\tilde U_i= \bigvee_{j \in \pa(i)}( \tilde c_{ji}+ \tilde U_j) \lor  \tilde Z_i + \tilde \varepsilon_i, \quad i=1,\ldots,d
\end{align}
with $\tilde \varepsilon_i \geq 0$. Thus, for every $j \in \pa(i)$, the difference $\tilde U_i-\tilde U_j$ is lower-bounded by $\tilde c_{ji}$ and 
$$\mathbb P(\tilde U_i - \tilde U_j \leq \tilde c_{ji}+x  \mid \tilde U_i = \tilde c_{ji}+ \tilde U_j + \tilde \varepsilon_i )=\mathbb P( \tilde \varepsilon_i \leq x ).$$

The estimation of (linear) functions with one-sided errors has been considered in the literature before. For instance, in \cite{hvk:09} and \cite{jmr:14} observations are given by $Y_j=f(X_j)+\eps_j$ for $j=1,\dots,n$ with observation errors $\eps_j>0$, with density given conditionally or unconditionally on $X_j=x$, and $f$ describes some frontier or boundary curve, which has to be estimated. 
To present an archetypical example, consider the linear regression problem stated in \cite{smith:1985} and \cite{smith:1994} as $Y_i=\beta + \eps_i$ for $i=1,\dots,n$ and observation errors, which have density $g(x)\sim \alpha c x^{\alpha-1}$ as $x\downarrow 0$ for $\alpha,c>0$. 
In these papers, the focus is on the non-regular case, when $\al<2$. Then $\beta$ can be estimated by the sample minimum $Y_{1,n}$ which has a Weibull limit law:
\begin{align} \label{Smith:1985}
    \lim_{\nto} \mathbb P\big((nc)^{-1/\alpha}(Y_{1,n}-\beta_0 )\leq x\big) = 1-\exp(-x^{-\alpha}), \quad 0<x<\infty.
\end{align}
The work in \cite{smith:1985} has been used in \cite{davis:mccormick:89} to estimate the coefficient $\phi$ of a first order autoregressive time series with positive innovations. They propose the minimum ratio estimator $\hat\phi=\bigwedge_{j=1}^n X_j/X_{j-1}$ and show in their Corollary~2.4 that it also has a Weibull limit law similar to \eqref{Smith:1985}.

In our model \eqref{ext} we find two interpretations for the noise variables. 
Firstly, in the log-transformed version \eqref{logtrans} we consider a ML model as baseline model, which is observed with some additive noise.
A second representation is given in Corollary~\ref{solutioncor} below, where the edge and path weights become noisy by the noise variables. This gives rise to the interpretation that we observe the model parameters with noise similarly as in the regression examples above. 
As a consequence, a path from $j$ to $i$ realising the ML coefficient $b_{ji}$ is no longer deterministic but depends on the individual realizations of the noise variables. 
However, in Theorem~\ref{epscount} we show that at the left limit of support
the distribution of the ratio of two model components is determined by all noise variables along the path between the two nodes. 
Assuming noise variables with regularly varying distribution in their left limit of support, we propose a minimum ratio estimator and show in Theorem~\ref{3.11} that the estimated ML coefficient matrix converges to a matrix of independent Weibull entries after proper centering and rescaling.

The paper is organized as follows. In section \ref{s2}, we summarize the properties of recursive ML models as defined in \eqref{nonnoise} and state the most important results relevant for our paper. In section~\ref{s3} we consider the extension of the recursive ML model given in \eqref{ext}, which we coin the \emph{max-linear model with propagating noise} and present its solution and the main properties of this new model.
In section~\ref{s4} we address the identifiability of the ML model with propagating noise. 
Similarly as in \eqref{Smith:1985} we suggest minimum ratio estimators for the model parameters $\bB$.
In section~\ref{s5} we assume regular variation of the noise variables. 
Under this assumption, we show that the minimum ratios are asymptotically independent and Weibull distributed. 
Finally, in section~\ref{s6}, we provide a data example and apply the theory that we have derived in the previous sections. All proofs are postponed to an Appendix.

Throughout we use the following notation. 
$\mathbb R_+=(0,\infty)$ and $\overline \R_+=[0,\infty)$,  $x\wedge y=\min\{x,y\}$ and $x\vee y=\max\{x,y\}$ with $\bigwedge_{i\in\emptyset} x_i =\infty$ and $\bigvee_{i\in\emptyset} x_i =0$ for $x_i\in \mathbb R_+$. 
Bold letters denote vectors and matrices, e.g. $\bI_d$ denotes the $d \times d$ identity matrix. Moreover, all vectors are row vectors unless stated otherwise.
For two functions $f,g$ we write $f(x) \sim g(x)$ as $x \downarrow c$ if $\lim_{x \downarrow c}f(x)/g(x)=1$ and $\bone$ denotes the indicator function. Moreover, $\an(i)$, $\pa(i)$ and $\de(i)$ denote the ancestors, the parents, and the descendants of node $i$, respectively, and $\An(i):=\an(i) \cup \{i\}$. Every edge $(j,i) \in E$ is a directed edge $j \to i$.  Finally, for a path $p=[k_0\rightarrow \ldots \rightarrow k_n]$ we define the {\em node set on the path (excluding the initial node)} by $S_{p}:=\{k_1,\ldots,k_n\}$ and its {\em path length} by $|S_p|$. For a random variable $Y$ with distribution function $F_Y$, the symbol $F^\leftarrow_Y$ denotes its {\em quantile function.}

\section{Preliminaries — Recursive max-linear models}\label{s2}

We first formally introduce the class of recursive ML models and state their most important results for this paper. Let $\mathcal D = (V, E)$ be a DAG with nodes $V= \{1,\ldots, d\}$ and edges $E$.
Then a random vector $\bX:=\left(X_1,\ldots, X_d\right)$ is a \emph{recursive max-linear} \emph{vector} or follows a {\em max-linear Bayesian network} on $\mathcal D$ if
\begin{align}\label{olddef}
    X_i:= \bigvee_{k \in \pa(i)}c_{ki}X_k\lor Z_i,\quad i \in 1,\ldots,d,
\end{align}
with positive edge weights $c_{ki}$ for $i \in V$ and $k \in \pa(i)$, and independent positive random variables $Z_1,\ldots,Z_d$ with support
$\mathbb R^+_0:=[0, \infty)$ and atom-free distributions. We shall refer to $\bZ:=(Z_1,\ldots,Z_d)$ as the \emph{vector of innovations}. 

For a path $p=[j=k_0 \rightarrow k_1 \rightarrow \ldots \rightarrow k_n=i]$ from $j$ to $i$ we define the {\em path weight} 
\begin{align}\label{path_weight}
    d_{ji}(p):=\prod_{l=0}^{n-1}c_{k_lk_{l+1}}.
\end{align}
Denoting the set of all paths from $j$ to $i$ by $P_{ji}$, we define the ML coefficient matrix $\bB=(b_{ij})_{d \times d}$ of $\bX$ with entries 
\begin{align*}
    b_{ij}:=\bigvee\limits_{p \in P_{ij}} d_{ij}(p) \quad \text{for } i \in \an(j), \quad b_{ii}=1, \quad \text{and} \quad b_{ij}=0 \quad \text{for } i \in V\setminus \An(j).
\end{align*}
The components of $\bX$ can also be expressed as ML functions of their ancestral innovations and an independent one; the corresponding ML coefficients are the entries of $\bB$:
\begin{align}\label{olddef:2nddef}
    X_i=\bigvee_{k \in \An(i)}b_{ki}Z_k, \quad i \in 1,\ldots, d,
\end{align}
which can be shown by a path analysis as in Theorem~2.2 in \cite{nadine1} or by tropical algebra as in \eqref{eq:dotrepr} below, and as we explain now.

For two non-negative matrices $\bF$ and $\bG$, where the number of columns in $\bF$ is equal to the number of rows in $\bG$, we define the matrix product $\odot: \overline\R_+^{m\times n}\times \overline\R_+^{n\times p} \rightarrow \overline\R_+^{m\times p}$
by
\begin{align}\label{ch2:odot} 
(\bF=(f_{ij})_{m\times n},\bG=(g_{ij})_{n\times p}) \mapsto \bF\odot \bG :=\Big(\bigvee\limits_{k=1}^n f_{ik}g_{kj}\Big)_{m\times p}.
\end{align}
The triple  $(\overline \R_+,\vee,\cdot)$, 
is an idempotent semiring with $0$ as 0-element and  $1$ as 1-element and 
the operation $\odot$ is therefore a matrix product over this semiring; see for example \cite{butkovic}. 
Denoting by $\mathcal M$ all $d\times d$ matrices with non-negative entries and by $\vee$ the componentwise maximum between two matrices, $(\mathcal M,\vee,\odot)$ is also a semiring with the null matrix as 0-element and the $d\times d$ identity matrix  $\bI_d$ as 1-element.  

The matrix product $\odot$ allows us to represent the ML coefficient matrix $\bB$ of $\bX$ in terms of the edge weight matrix $\bC:=(c_{ij} \bone_{\pa(j)}(i))_{d\times d}$ of $\cald$, since \eqref{olddef} can be rewritten as 
\begin{align}\label{eq:XtroprepX}
   \bX=\bX\odot \bC\vee\bZ 
\end{align}
with unique solution (equivalent to \eqref{olddef:2nddef}) given by
 \begin{align}\label{eq:dotrepr}
 \bB= (\bI_d\vee \bC)^{\odot (d-1)}= \bigvee_{k=0}^{d-1}  \bC^{\odot k}, \qquad \bX= \bZ\odot \bB,
   \end{align}
where $\bB$ is the {\em Kleene star matrix} and we have let $\bA^{\odot 0}=\bI_d$ and $\bA^{\odot k}= \bA^{\odot(k-1)}\odot \bA$ for $\bA\in \overline\R_+^{d\times d}$ and $k\in\mathbb{N}$; see Proposition~1.6.15 of \cite{butkovic} as well as Theorem~2.4 and Corollary 2.5 of \cite{nadine1}. For more information on the max-times (tropical) algebra in ML models see section~2.2 in \cite{Ngoc_etal}.

We have seen that a recursive ML vector $\bX$  has two representations, one in terms of parental nodes $X_j$ and edge weights $c_{ji}$ and another in terms of innovations $Z_j$ and ML coefficients $b_{ji}$. However, while the ML coefficient matrix $\bB$ of $\bX$ is identifiable from the distribution of $\bX$, the edge weight matrix $\bC$ is generally not, see Theorem~5.4(b) in \cite{nadine1}.  Theorem~5.3 in that paper and Theorem~2 in \cite{nadine2} show that an edge with edge weight $c_{ji}$ is identifiable from $\bB$ if and only if it is the unique path from $j$ to $i$ with $d_{ji}(p)=b_{ji}$. 

For a recursive ML vector $\bX$ on a DAG $\mathcal D=(V,E)$ and ML
coefficient matrix $\bB$ this result leads to the following definition.

\begin{definition}\label{minD:original}
    Let $\bX\in\R_+^d$ be a recursive ML vector on the DAG $ \mathcal D=(V, E)$ with ML coefficient matrix $\bB$. We define the {\em minimum ML DAG of $\bX$} as 
    \begin{align*}
        \mathcal D^B=(V,E^B):=\Big(V,\Big\{(j,i) \in E: b_{ji}>\bigvee\limits_{k \in \de(j) \cap \pa(i)}\frac{b_{jk}b_{ki}}{b_{kk}} \Big\}\Big).
    \end{align*}
\end{definition} 

Moreover, it has be shown that for a recursive ML vector $\bX$ for the support it holds that
\begin{align*}
    \supp(X_i/X_j) & =
    \begin{cases}
        [b_{ji},\infty) & \text{for } j \in \an(i), \\
        [0,1/b_{ij}]    & \text{for } i \in \an(j), \\
        \{1\}    & \text{for } i =j, \\
        \mathbb R_+     & \text{otherwise, }
    \end{cases}
\end{align*}
with $\mathbb P(X_i/X_j=b_{ji})>0$ for all $j \in \an(i)$; see Lemma~1 of \cite{nadine2}. 
Hence, for a given iid sample $\bX^1,\ldots,\bX^n$ from $\bX$ 
define a minimum ratio estimator $\hat\bB$ of $\bB$ by
$\hat{b}_{ij}:=\bigwedge_{k=1}^n (X_i^k/X_j^k)$ for $i,j\in V$. 
Moreover, when the DAG $\D$ is known, we define $\bB_0$ by
\begin{align*}
\bB_0 = (B_0(i,j))_{d \times d} :=\Big(\bigwedge\limits_{k=1}^n \frac{X_j^k}{X_i^k} \bone_{\pa(j) }(i)\Big)_{d \times d}\quad\mbox{and set}\quad
    \hat{\bB}=(\bI_d \lor \bB_0)^{\odot(d-1)}.
\end{align*}
Theorem~4 of \cite{nadine2} ensures that $\hat{\bB}$ is a generalized maximum likelihood estimate (GMLE) in the sense of \cite{gmle}.

\section{Recursive ML model with propagating noise} \label{s3}

In this section we present structural results for the recursive ML model with propagating noise, in particular, we investigate which properties of the non-noisy model prevail.

\begin{definition}\label{def:mlprop}
A vector  $\bU \in \mathbb R^d_+$ is a {\em recursive ML vector with propagating noise}  on a DAG $\mathcal D=(V,E)$, if 
\begin{align} \label{1stdef_1}
    U_i := \Big(\bigvee_{k \in \pa(i)}c_{ki}U_k\lor Z_i\Big)\eps_i,\quad i \in 1,\ldots,d,
\end{align}
with edge weight matrix $\bC:=(c_{ij} \bone_{\pa(j)}(i))_{d\times d}$.
The noise variables $\varepsilon_1,\ldots,\varepsilon_d$ are iid and atom-free random variables with $\varepsilon_i\ge 1$ and unbounded above for all $i \in V$, and independent of the innovations vector $\bZ:=(Z_1,\ldots, Z_d)$. For simplicity, we denote by $\varepsilon$ a generic noise variable and by $Z$ a generic innovation. 
\end{definition}

Although the noise variables act on the observations, formally we can view them as random scalings of edge weights.
More precisely, for a path $p=[j=k_0 \rightarrow k_1 \rightarrow \ldots \rightarrow k_n=i]$ from $j$ to $i$ we define the {\em random path weight} $\bar{d}_{ji}$ similarly to the definition of $d_{ji}$ in \eqref{path_weight} as
\begin{align}\label{randompathweight}
    \bar{d}_{ji}(p):=\varepsilon_j \prod_{l=0}^{n-1}c_{k_lk_{l+1}}\varepsilon_{k_{l+1}} = d_{ji}(p)\eps_j\prod_{l=0}^{n-1}\eps_{k_{l+1}}.
\end{align}
If we define the {\em random edge weight matrix} 
    \begin{align}\label{randomweights}
        \bar{\bC}=(\bar c_{ij})_{d\times d} :=(c_{ij} \varepsilon_j \bone_{\pa(j)}(i))_{d\times d}
    \end{align}
 we can rewrite \eqref{randompathweight} as 
    \begin{align*}
        \bar{d}_{ji}(p):=\varepsilon_{j} \prod_{l=0}^{n-1}\bar{c}_{k_lk_{l+1}}
    \end{align*}
   for every path $p=[j=k_0 \rightarrow k_1 \rightarrow \ldots \rightarrow k_n=i]$ from $j$ to $i$. Hence, we can view the noise variables as random scalings for the edge weights $c_{ji}$.
    Since $\varepsilon \geq 1$, the edge weights $c_{ji}$ of the non-noisy model are lower bounds for the random edge-weights $\bar{c}_{ji}$ of the propagating noise model.

Again denoting the set of all paths from $j$ to $i$ by $P_{ji}$, we define the \emph{random ML coefficient matrix} $\bar{\bB}=(\bar{b}_{ij})_{d \times d}$ of $\bU$ with entries
\begin{align}\label{randomb}
    \bar{b}_{ji}:=\bigvee\limits_{p \in P_{ji}} \bar{d}_{ji}(p) \quad \text{for } j \in \an(i), \quad \bar{b}_{jj}=\varepsilon_{j}, \quad \text{and} \quad \bar{b}_{ji}=0 \quad \text{for } j \in V\setminus \An(i).
\end{align}

We next show that there exists a solution of \eqref{1stdef_1} in terms of the ancestral innovations $\bZ$ and $\bar{\bB}$. All proofs of this section are postponed to Appendix~\ref{A}.

    \begin{theorem}\label{solution}
    Let $\bU \in \mathbb R^d_+$ be a recursive ML vector with propagating noise on a DAG $\mathcal D$ as in \eqref{1stdef_1}. Define  $(\bE_d)_{d \times d}$ as the diagonal matrix given by
    \begin{align*}
        E_d(i,i)= \varepsilon_i\quad\mbox{for } i \in V\quad\mbox{and}\quad E_d(i,j)=0 \quad\mbox{for } i,j \in V \ \mbox{and }   i \neq j.
    \end{align*}
    We rewrite \eqref{1stdef_1}
  in matrix form by means of the matrix multiplication \eqref{ch2:odot}  as
$$\bU = \big(\bU\odot \bC \vee \bZ\big)\odot \bE_d.$$
Then $\bU$ has a unique solution in terms of the tropical matrix multiplication with random matrix $\bar{\bB}$ given by
    \begin{align}\label{Bstern}
        \bar{\bB}=\bE_d \odot (\bI_d\vee \bar{\bC})^{\odot (d-1)}, \quad \bU=\bZ \odot \bar{\bB},
    \end{align} 
with $\bar{\bC}$ as defined in \eqref{randomweights}.  
    \end{theorem}
    
Since $\bar{b}_{ji}=0$ whenever $j \not \in \An(i)$, the representation \eqref{Bstern} can be rewritten as follows.

\begin{corollary}\label{solutioncor}
    Let $\bU$ be as in Theorem~\ref{solution} and $\bar{b}_{ji}$ be the random ML coefficients defined in $\eqref{randomb}$. Then \eqref{noisezdef} is equivalent to 
    \begin{align}\label{noisezdef}
        U_i= \bigvee_{j \in \An(i)}\bar{b}_{ji}Z_j, \quad i \in 1,\ldots,d.
    \end{align}
    \end{corollary}

Note that the definition in~\eqref{1stdef_1} is equivalent to 
\begin{align} \label{1stdef_2}
   U_i= \tilde U_i \, \varepsilon_i\quad\mbox{with}\quad \tilde U_i:= \bigvee_{k \in \pa(i)}c_{ki}U_k\lor Z_i,\quad i \in 1,\ldots,d.
\end{align}
From this result we can compute the following representation.

\begin{corollary}\label{cor:3rddef}
    Let $\bU$ and $\bar{b}_{ji}$ be as in Corollary~\ref{noisezdef}.
 
    Then \eqref{noisezdef} is equivalent to
    \begin{align}\label{noise3rddef}
        U_i= \bigvee_{k \in \An(i)}\bar{b}_{ki}\tilde U_k, \quad i = 1,\ldots,d,
    \end{align}
    with $\tilde U_k$ as in \eqref{1stdef_2}.
\end{corollary}
{We next define critical and generic paths which play an essential role for the understanding of our model.}

\begin{definition}\label{def:critical}
Let $\D$ be a DAG with edge weight matrix $\bC$ and let $\bB$ be the corresponding ML coefficient matrix (i.e. the Kleene star of $\bC$). Let $p$ be a path from $j$ to $i$ with node set $S_p$.
\begin{itemize}
    \item[(a)]
$p$ is called a \emph{(non-random) critical path} if $d_{ji}(p)=b_{ji}$.
\item[(b)]
$p$ is called a \emph{generic path} if it is the only path satisfying $d_{ji}(p)=b_{ji}$. 
\item[(c)]
We call $\bC$ {\em generic}, if all paths in $\D$ are generic.
    \item[(d)]
$p$ is called a \emph{random critical path} if $\bar d_{ji}(p)=\bar b_{ji}$.
\item[(e)]
$p$ is called a {\em possible critical path realization}, if $U_i=U_j d_{ji}(p)\prod_{k \in S_{p}}\varepsilon_k= \tilde U_j \bar d_{ji}(p)$ happens with positive probability.
\end{itemize}
\end{definition}

\begin{remark}
We have defined a non-random critical path and a random critical path.
We want to emphasize, however, that while the first path property is simply inherited from $\bC$ via $\bB$, the second one is inherited from $\bC$ and the noise variables. 
We also note that by continuity of the innovations and the noise variables, any random critical path between a pair of nodes must be unique, although it may vary with the realizations of the noise variables.
\end{remark}

We explain the model and the notions of Definition~\ref{def:critical} in an example.

\begin{example} \label{dag1}
Consider the DAG:

   \begin{center}
   \begin{tikzpicture}
    [->,every node/.style={circle,draw},line width=0.8pt, node distance=1.6cm,minimum size=0.8cm,outer sep=1mm]
    \node (n1) {${1}$};
    \node (n2) [right of=n1] {${2}$};
    \node (n3) [right of=n2] {${3}$};
    \foreach \from/\to in {n1/n2,n2/n3}
    \draw[->, line width=0.8pt] (\from) -- (\to);
    \draw[->, thick] (n1) to[out=45, in=135] (n3);
\end{tikzpicture} 
   \end{center}
Then, $\bC$ is generic if and only if $c_{13}\neq c_{12}c_{23}$. Moreover, we have
\begin{align*}
    U_3=(\bar{c}_{13} \lor \bar{c}_{12} \bar{c}_{23})\varepsilon_1 Z_1 \lor \bar{c}_{23}\varepsilon_2 Z_2 \lor \varepsilon_3 Z_3,
\end{align*}
with $\bar{c}_{ji}=c_{ji}\varepsilon_i$ as defined in \eqref{randomweights}. 
Now assume that $c_{13}>c_{12}c_{23}$. In that case, $[1 \to 3]$ is the critical path, while the path $[1 \to 2 \to 3]$ is not critical. However,  $\mathbb P(\bar{c}_{13}<\bar{c}_{12}\bar{c}_{23})= \mathbb P(\varepsilon_2>c_{13}/(c_{12}c_{23}))>0$. For this reason, both paths can be random critical. Finally, all paths in $\mathcal D$ can be possible critical path realizations. To stress the difference between a random critical path and a possible critical path realization, observe that e.g. $[1\to 3]$ and $[2\to 3]$ can be random critical for the same realized noise and innovation variables, however, the two paths cannot be possible critical path realizations for the same noise and innovation variables up to a null set. 

In contrast, if $c_{13}<c_{12}c_{23}$ we have $\mathbb P(\bar{c}_{13}>\bar{c}_{12}\bar{c}_{23})= \mathbb P(\varepsilon_2<c_{13}/(c_{12}c_{23}))=0$. In this case, the path $[1 \to 3]$ can not be random critical and particularly not a possible critical path realization.

This illustrates that a path $p$ from $j$ to $i$  with path weight $d_{ji}(p)<b_{ji}$ may as well contribute to the distribution of $U_i$. 
However, an edge $p=[j \to i]$ with $d_{ji}(p)<b_{ji}$ is still not identifiable and does not change the distribution of $\bU$. 
\end{example}

While we still want to estimate the (non-random) ML coefficient matrix $\bB$, we first present some useful properties of  $\bB$  and $\bar{\bB}$ and a link between the noisy and non-noisy model as defined in \eqref{olddef} and \eqref{1stdef_1}, respectively.

\begin{lemma}\label{lemma1}
    Let $\bU \in \mathbb R^d_+$  be a recursive ML vector with propagating noise on a DAG $\mathcal D$ as defined in \eqref{1stdef_1} with $\bB$ and $\bar{\bB}$ defined in \eqref{eq:dotrepr} and \eqref{Bstern}, respectively. 
    Then the following assertions hold:
    \begin{enumerate}
        \item
    $\begin{aligned}[t]
                      \bar{b}_{ji} = \bigvee\limits_{k \in V}\frac{\bar{b}_{jk}\bar{b}_{ki}}{\bar{b}_{kk}} \ge  \bigvee\limits_{k \in \de(j) \cap \an(i)}\frac{\bar{b}_{jk}\bar{b}_{ki}}{\bar{b}_{kk}},
                  \end{aligned}$
                  where the inequality is strict, whenever the random critical path from $j$ to $i$ is the edge $j \to i$, or  $j=i$.
                \item
There exists some path $p:=[j \to \ldots \to k \to \ldots \to i]$ from $j$ to $i$ that passes through $k$ such that 
$$\bar d_{ji}(p)=\bar b_{ji}\quad\text{ if and only if }\quad\bar b_{ji}=\frac{\bar b_{jk}\bar b_{ki}}{\bar b_{kk}}.$$
        \item
              $\begin{aligned}[t]
                      \frac{U_i}{U_j}\geq \frac{\bar{b}_{ji}}{\bar{b}_{jj}} \geq b_{ji} \mbox{\quad with $b_{ji}=0$ for $j\notin\An(i)$} 
                  \end{aligned}$ 
        \item
              $\begin{aligned}[t] \label{supportnoise1}
                      \normalfont{\supp}(U_i/U_j) & =
                      \begin{cases}
                          [b_{ji},\infty) & \text{for } j \in \an(i),\\
                          [0,1/b_{ij}]    & \text{for } i \in \an(j), \\
                          \{1\}               & \text{for } i =j,         \\
                          \mathbb R_+     & \text{otherwise. }
                      \end{cases}
                  \end{aligned}$ \\
                  Moreover, for $j \neq i$, neither the distribution of $U_i/U_j$ nor the distribution of $U_j/U_i$ have any atoms.
        \item 
        If $b_{ji} = \bigvee\limits_{k \in \de(j) \cap \an(i)}\frac{b_{jk}b_{ki}}{b_{kk}}$, then  $\bar{b}_{ji} = \bigvee\limits_{k \in \de(j) \cap \an(i)}\frac{\bar{b}_{jk}\bar{b}_{ki}}{\bar{b}_{kk}}$.  \\
        \item
    If 
$b_{ji} > \bigvee\limits_{k \in \de(j) \cap \an(i)}\frac{b_{jk}b_{ki}}{b_{kk}} \text{ and } \de(j) \cap \an(i) \neq\emptyset$, then 
$$\P\Big(\bar{b}_{ji} > \bigvee\limits_{k \in \de(j) \cap \an(i)}\frac{\bar{b}_{jk}\bar{b}_{ki}}{\bar{b}_{kk}}\Big)>0\quad\mbox{ and}\quad
\P\Big(\bar{b}_{ji} = \bigvee\limits_{k \in \de(j) \cap \an(i)}\frac{\bar{b}_{jk}\bar{b}_{ki}}{\bar{b}_{kk}}\Big)>0.$$ 

    \end{enumerate}
\end{lemma}

Lemma~\ref{lemma1} b) and f) motivate the following definition.

\begin{definition}\label{def:RanMinDag}
    Let $\bU \in \mathbb R^d_+$  be a recursive ML vector with propagating noise on the DAG $\mathcal D=(V,E)$ as defined in \eqref{1stdef_1}. Then we define the \emph{minimum ML DAG $\mathcal D^{*B}$ of $\bU$} as
    \begin{align*}
        \mathcal D^{*B}=(V,E^{*B}):=\left(V,\bigg\{(j,i) \in E: \mathbb P \Big(\bar{b}_{ji}>\bigvee\limits_{k \in \de(j) \cap \pa(i)}\frac{\bar{b}_{jk}\bar{b}_{ki}}{\bar{b}_{kk}}\Big)>0  \bigg\}\right).
    \end{align*}
\end{definition} 

{In addition, applying  first Lemma~\ref{lemma1} e) and f), an in the second part Lemma~\ref{lemma1} b) yields the following result.}

\begin{corollary}\label{minD}
    Let $\bX \in \mathbb R^d_+$ be a recursive ML vector on a DAG $\mathcal D=(V,E)$ as defined in \eqref{olddef} and $\bU \in \mathbb R^d_+$ be a recursive ML vector with propagating noise as defined in \eqref{1stdef_1} on the same DAG $\mathcal D$ with the same edge weight matrix $\bC$. Then
    \begin{align*}
        \mathcal D^B = \mathcal D^{*B}.
    \end{align*}
    {Moreover, the minimum ML DAG $\mathcal D^B$ is the smallest DAG that preserves the distribution of $\bX$ and of $\bU$.}
\end{corollary} 

Therefore, we will henceforth only use the term $\mathcal D^B$.

\begin{lemma}\label{lemma:crit_rand_path}
Let $\bU \in \mathbb R^d_+$  be a recursive ML vector with propagating noise on a DAG $\mathcal D$ as defined in \eqref{1stdef_1}.
Then the following assertions hold:
\begin{enumerate}
    \item 
    A path $p=[j=k_0 \to \ldots \to k_n=i]$ in $\D$ is a  possible critical path realization from $j$ to $i$
    if and only if all edges of $p$ belong to the minimum ML DAG $\D^B$. 
   \item 
   Let $p_1$ and $p_2$ be two possible critical path realizations from $j$ to $i$ and from $l$ to $m$, respectively.  Then 
        \begin{align}\label{jointprobability}
       \Big\{U_i=  U_j  d_{ji}(p_1) \prod_{k \in S_{p_1}} \varepsilon_k, U_m= U_l d_{lm}(p_2) \prod_{k \in S_{p_2}}\varepsilon_k \Big\}
        \end{align}
        has positive probability if and only if $S_{p_1} \cap S_{p_2} = \emptyset$, or for every $r \in S_{p_1} \cap S_{p_2}$ the sub-path of $p_1$ from $j$ to $r$  is a sub-path of $p_2$ or the  sub-path  of $p_2$ from $l$ to $r$ is a sub-path of $p_1$.
\end{enumerate}
\end{lemma}
We illustrate part b) with Figure~\ref{fig:posscrit1} and Figure~\ref{fig:posscrit2}.

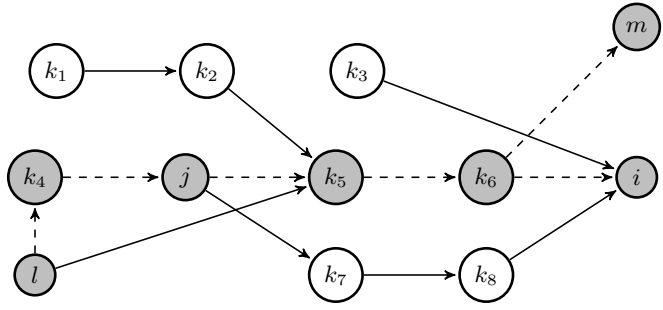
\begin{figure}[ht]
\begin{center}
\begin{tikzpicture}[->,>=stealth',shorten >=1pt,auto,node distance=2cm,semithick]
  \tikzstyle{every node}=[circle,line width =1pt,font=\footnotesize]
  \node (a2) [draw,fill=gray!50] {$j$};  
  \node (a6) [right of = a2,draw,fill=gray!50] {$k_5$};
  \node (a7) [right of = a6,draw,fill=gray!50] {$k_6$};
  \node (a3) [right of=a7,draw,fill=gray!50] {$i$};
  \node (a9) [left of=a2,draw,,fill=gray!50] {$k_4$};
  \node (a11) [below of=a9, yshift=0.7cm, draw,fill=gray!50] {$l$};
  \node (a13) [above of=a3,draw,fill=gray!50] {$m$};
 
  \node (a1) [below of=a6, yshift=0.7cm, draw] {$k_7$};
  \node (a4) [above left of=a2,xshift=1.7cm,draw] {$k_2$};
  \node (a5) [right of=a4,draw] {$k_3$};
  \node (a8) [below of=a7,yshift=0.7cm, draw] {$k_8$};
  \node (a10) [left of=a4,draw] {$k_1$};
    
  \path (a2) edge [dashed](a6);
  \path (a6) edge [dashed](a7);
  \path (a7) edge [dashed](a3);
  \path (a11) edge [dashed](a9);
  \path (a9) edge [dashed](a2);
  \path (a7) edge [dashed](a13);
 
  \path (a4) edge [color=black](a6);
  \path (a5) edge [color=black](a3);
  \path (a1) edge [color=black](a8);
  \path (a2) edge [color=black](a1);
  \path (a8) edge [color=black](a3);
  \path (a10) edge [color=black](a4);
  \path (a11) edge [color=black](a6);

\end{tikzpicture}
\end{center}
\caption{
Both dashed paths $p_1:=[j \to k_5 \to k_6 \to i]$ and $p_2:=[l \to k_4 \to j \to  k_5 \to k_6 \to m]$ can be possible critical path realizations from the same realized noise variables along the nodes.
}\label{fig:posscrit1}
\end{figure}

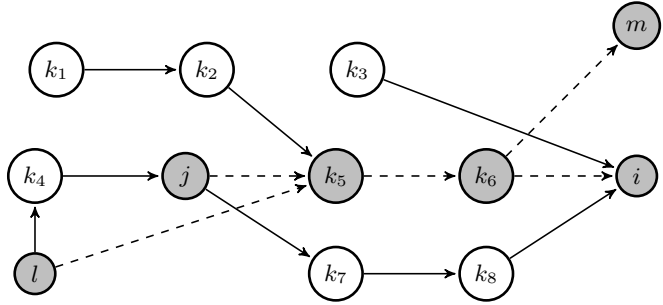
\begin{figure}[ht]
\begin{center}
\begin{tikzpicture}[->,>=stealth',shorten >=1pt,auto,node distance=2cm,semithick]
  \tikzstyle{every node}=[circle,line width =1pt,font=\footnotesize]
  \node (a2) [draw,fill=gray!50] {$j$};  
  \node (a6) [right of = a2,draw,fill=gray!50] {$k_5$};
  \node (a7) [right of = a6,draw,fill=gray!50] {$k_6$};
  \node (a3) [right of=a7,draw,fill=gray!50] {$i$};
  \node (a9) [left of=a2,draw] {$k_4$};
  \node (a11) [below of=a9, yshift=0.7cm, draw,fill=gray!50] {$l$};
  \node (a13) [above of=a3,draw,fill=gray!50] {$m$};
 
  \node (a1) [below of=a6, yshift=0.7cm, draw] {$k_7$};
  \node (a4) [above left of=a2,xshift=1.7cm,draw] {$k_2$};
  \node (a5) [right of=a4,draw] {$k_3$};
  \node (a8) [below of=a7,yshift=0.7cm, draw] {$k_8$};
  \node (a10) [left of=a4,draw] {$k_1$};

  \path (a6) edge [dashed](a7);
  \path (a7) edge [dashed](a3);
  \path (a2) edge [dashed](a6);
  \path (a9) edge [color=black](a2);
  \path (a7) edge [dashed](a13);
  \path (a11) edge [dashed](a6);
  \path (a11) edge [color=black](a9);
 
  \path (a4) edge [color=black](a6);
  \path (a5) edge [color=black](a3);
  \path (a1) edge [color=black](a8);
  \path (a2) edge [color=black](a1);
  \path (a8) edge [color=black](a3);
  \path (a10) edge [color=black](a4);

\end{tikzpicture}
\end{center}
\caption{ 
Both dashed paths $p_1:=[j \to k_5 \to k_6 \to i]$ and $p_2:=[l \to  k_5 \to k_6 \to m]$ can only  on a null-set be possible critical path realizations from the same realized noise variables along the nodes.
}\label{fig:posscrit2}
\end{figure}

We conclude this section with an important result that not only helps us to understand the model better, but is also an important step for learning the model.

\begin{theorem} \label{epscount} 
    Let $\bU\in\R^d_+$ be a recursive ML vector with propagating noise on a DAG $\mathcal D$ as defined in \eqref{1stdef_1}. 
    Suppose that $p_{\max}:=[j=k_0 \to \dots \to k_n=i]$ is generic.
    Let $S_{p_{\max}}=\{k_1,\dots,k_n\}$ be the set of nodes on $p_{\max}$. 
    Then
    \begin{align*}
       \mathbb P \bigg(\frac{U_i}{U_j} \leq  b_{ji}x\bigg)
       &\sim \mathbb P\Big(\prod\limits_{k\in S_{p_{\max}}} \varepsilon_{k} \leq x,\frac{U_i}{U_j} = b_{ji}\prod\limits_{k\in S_{p_{\max}}} \varepsilon_{k}\Big) \, \sim   \, c \, \mathbb P\bigg(\prod\limits_{k\in S_{p_{\max}}} \varepsilon_{k} \leq x \bigg),\quad x\downarrow 1,
    \end{align*}
    for some constant $c \in (0,1)$.
\end{theorem}

\begin{remark}
If the distributions of the noise variables and the innovations as well as the path weights of the underlying DAG $\mathcal D$ are given, the constant $c$ in Theorem~\ref{epscount} can be calculated explicitly.
\end{remark}

Theorem~\ref{epscount} also shows that, while any path $p$ from $j$ to $i$ with $d_{ji}(p)<b_{ji}$ contributes to the distribution of $\bU$ (as we have seen in Example~\ref{dag1}), they influence the distribution of $U_i/U_j$ at their left limit of support only by a constant. 

We now extend the result to situations with several critical paths.

\begin{corollary}\label{epscount_prop}
    Let $\bU$ be as in Theorem~\ref{epscount}. Suppose that there are several paths $p_1, \dots, p_n$ from $j$ to $i$ that are critical; i.e., $d_{ji}(p_1)=\ldots=d_{ji}(p_n)=b_{ji}$.  Then
     \begin{align*}
       \mathbb P\Big(\frac{U_i}{U_j} \leq  b_{ji}x\Big)
       \sim c \, \mathbb P\bigg(\bigcap\limits_{p \in \{p_1,\ldots, p_n\}}\Big\{\prod\limits_{k \in  S_{p} } \varepsilon_{k} \leq x \Big\} \bigg),\quad x\downarrow 1,
    \end{align*}
for some constant $c \in (0,1)$.
\end{corollary} 

For simplicity, we assume from now on that $\bC$ is generic in the sense of Definition~\ref{def:critical}.
However, we want to remark that all such results can be extended to the case of several non-random critical paths between two nodes. The proofs of such results work similarly as the proof of Corollary~\ref{epscount_prop}.

We continue with another consequence of Theorem~\ref{epscount}.

\begin{corollary} \label{epscount_cor}
    Let $\bU$ be as in Theorem~\ref{epscount} and suppose that $p:=[j=k_0 \to \dots \to k_n=i]$ is generic. Let $\bU^1,\ldots, \bU^n$ for $n\in\N$ be an iid sample from $\bU$. Then, for the same constant $c\in(0,1)$ as in Theorem~\ref{epscount}, we have
    \begin{align*}
       \mathbb P\Big(\bigwedge\limits_{k=0}^n \frac{U^k_i}{U^k_j} \leq  b_{ji}x\Big)
       \sim c \, n \, \mathbb P\bigg(\prod\limits_{i=1}^{n} \varepsilon_{k_i} \leq x \bigg),\quad x\downarrow 1.
    \end{align*}
\end{corollary}

We conclude this section by extending Theorem~\ref{epscount} to multivariate distributions. We only formulate and prove the bivariate case, the general case is then obvious.
Recall that in Lemma~\ref{lemma:crit_rand_path} we gave a necessary and sufficient condition for \eqref{biv:posscrit} below.
\begin{theorem} \label{epscountmult}
    Let $\bU\in\R^d_+$ be a recursive ML vector with propagating noise on a DAG $\mathcal D$ as defined in \eqref{1stdef_1}.  Suppose generic paths $p_1$ from $j$ to $i$ and $p_2$  from $l$ to $m$.
    Assume that 
    \begin{align} \label{biv:posscrit}
            \mathbb P \Big(U_i= U_j b_{ji}\prod_{k \in S_{p_1}}\varepsilon_{k}, U_m=U_l b_{lm}\prod_{k \in S_{p_2}}\varepsilon_{k}\Big)>0.
    \end{align}
    Then 
    \begin{align*}
       &\mathbb P\left(\frac{U_i}{U_j} \leq  b_{ji}x_1, \frac{U_m}{U_l} \leq  b_{lm}x_2\right) \sim   \,c \, \mathbb P\Big(\prod_{k \in S_{p_1}} \varepsilon_{k} \leq x_1,\prod_{k \in S_{p_2}} \varepsilon_{k} \leq x_2\Big)  \\ \sim \, &\mathbb P\Big(\prod_{k \in S_{p_1}} \varepsilon_{k} \leq x_1,\prod_{k \in S_{p_2}} \varepsilon_{k} \leq x_2,\frac{U_i}{U_j}= b_{ji}\prod_{k \in S_{p_1}} \varepsilon_{k}, \frac{U_m}{U_l}= b_{lm}\prod_{k \in S_{p_2}} \varepsilon_{k}\Big) ,\quad x_1,x_2\downarrow 1.
    \end{align*}
    for some constant $c \in (0,1)$.
\end{theorem}

\section{Identification and estimation}\label{s4}

We first address the question of identifiability of $\bB$ from the distribution of $\bU$. In particular, we are going to show that even though innovations and noise variables are generally not identifiable, $\bB$ remains identifiable also in the propagating noise model. 

We discuss three settings \ref{setting1}-\ref{setting3} below.
For each setting, we propose an appropriate minimum ratio estimator for $\bB$. Afterwards, we will show the almost sure convergence of each of the estimators.

\subsection{Identifiability of the model}

Most results concerning the identifiability are based on results from section~\ref{s3}. As we have already seen in Example~\ref{dag1}, the edge weight matrix $\bC$ is generally not identifiable from the distribution of $\bU$. However, an immediate consequence of Lemma~\ref{lemma1}~d) is the following.

\begin{corollary}
Let $\bU\in\R^d_+$ be a recursive ML model with propagating noise on a DAG $\mathcal D$ as defined in \eqref{1stdef_1}. Then the ML coefficient matrix  $\bB$ is identifiable from the distribution of $\bU$. 
\end{corollary}

Since we can identify $\bB$ from the distribution of $\bU$, we can also identify the minimum ML DAG $\D^B$ from Definition~\ref{minD:original} (which by Definition~\ref{def:RanMinDag} and Corollary~\ref{minD} is the minimum DAG preserving the distribution of $\bU$). 
Therefore, since $\varepsilon\geq 1$, Theorem~2 of \cite{nadine2} also holds for the propagating noise model as defined in \eqref{1stdef_1}. 
Therefore, as exemplified in Example~\ref{dag1}, we can identify the class of all DAGs and edge weights that could have generated $\bU$.

However, unlike for the non-noisy model, we can generally not identify innovations or noise variables. To see this assume a source node $U_i$ in a DAG $\mathcal D$ such that $\an(i)=\emptyset$. If $\bU$ follows a recursive ML model with propagating noise, then $U_i:=Z_i \eps_i.$ In particular, we can not identify $Z_i$ or $\eps_i$. 

When estimating a recursive ML model, we distinguish between three settings:

\begin{enumerate}[label={(\arabic*)}] 
    \item \label{setting1} 
    All ancestral relations are known; i.e., we know the set of edges $E$, hence the DAG.
    This might be the case when modeling networks that contain natural information about edges. The problem then reduces to finding appropriate estimates $\hat{b}_{ji}$ for $j \in \an(i)$.
    \item 
    The ancestral relations are unknown; however, we know a topological order of the nodes. Then, in contrast to setting 1, we need to decide if a path from $j$ to $i$ with $j<i$ exists.
    \item \label{setting3} 
    Neither the underlying DAG nor a topological order of the nodes is known. Then we need to find a topological order of the nodes and proceed then as in setting 2.
\end{enumerate}

We next want to estimate $\bB$ for each of the three settings \ref{setting1}-\ref{setting3}.

\subsection{Known DAG structure with unknown edge weights}\label{s51}

Given an iid sample $\bU^1,\ldots,\bU^n$ from a recursive ML model with propagating noise on a known DAG $\mathcal D$ as defined in \eqref{1stdef_1} and knowing all ancestral relations of $\mathcal D$, we could choose the simple estimate 
\begin{align}\label{estimate2a}
    \check{\bB}:=(\check{b}_{ij})_{d \times d}=\Big(\bigwedge\limits_{k=1}^n \frac{ U_j^k}{ U_i^k} \bone_{\An(j)}(i) \Big)_{d \times d}.
\end{align}
However, as in the non-noisy model, the estimate \eqref{estimate2a} may not define any recursive ML model on the given DAG $\D$, cf. Example~3 of \cite{nadine2}. 

We use instead
\begin{align}\label{estimate1}
\bB_0 = (B_0(i,j))_{d \times d} :=\Big(\bigwedge\limits_{k=1}^n \frac{U_j^k}{U_i^k} \bone_{\pa(j) }(i)\Big)_{d \times d}\quad\mbox{and set}\quad
    \hat{\bB}=(\bI_d \lor \bB_0)^{\odot(d-1)}.
\end{align}
Applying Lemma~2 in \cite{nadine2} to $\bB_0$,
the estimator $\hat \bB$ yields a valid estimate of the given DAG in the sense that $\hat{\bB}$ defines a recursive ML model and for any pair $(j,i) \not \in E(\mathcal D)$ we have $\hat{b}_{ji}=\bigvee_{k \in \{1, \ldots, d\} \setminus \{j,i\}}\hat{b}_{jk}\hat{b}_{ki}$. Moreover, by the idempotency of $\hat{\bB}$ and Lemma~\ref{lemma1} c), similarly to the non-noisy model, it also holds that
\begin{align}\label{ineq:estimator}
    b_{ji}\leq \hat{b}_{ji} \leq \check{b}_{ji},\quad j\in\an(i). 
\end{align}

\subsection{Known topological order} \label{s52}

Given an iid sample $\bU^1,\ldots,\bU^n \in \mathbb R_+^d$ from a recursive ML model with propagating noise without knowing $\D$, but knowing the topological order of nodes, we adapt the estimator \eqref{estimate2a} to this situation and define
\begin{align}\label{estimate2}
    \hat{\bB}:=(\hat{b}_{ij})_{d \times d}=\Big(\bigwedge\limits_{k=1}^n \frac{U_j^k}{U_i^k} \bone_{(i<j)} \Big)_{d \times d}.
\end{align}

\subsection{Unknown DAG and unknown topological order}\label{s53}

Given an iid sample $\bU^1,\ldots,\bU^n \in \mathbb R_+^d$ from a recursive ML model with propagating noise without knowing $\mathcal D$ or the topological order, we will recover a topological order first and then proceed as in section~\ref{s52}.

{Estimating the topological order of an underlying DAG is often done by learning algorithms that successively identify source nodes and succeeding generations. For additive models, usually regression techniques are applied (see e.g. \cite{drtoncausal} or \cite{JMLR:v15:peters14a}). In the recursive ML model, the noise is not additive and the model is highly non-linear.
Hence, such regression methods cannot be applied. However, under the condition of multivariate regular variation, the paper \cite{krali} suggests a learning algorithm for the model without noise as given in \eqref{nonnoise}. 
We propose a different approach, which to the best of our knowledge has not been considered in the literature before. It applies to the propagating noise model without any distributional assumptions on the innovations and noise variables and learns the DAG by using minimum ratios.}
We first consider the matrix of all minimum ratios given by
\begin{align} \label{minratios}
    \check{\bB}:=(\check{b}_{ij})_{d \times d}=\Big(\bigwedge\limits_{k=1}^n \frac{U_j^k}{U_i^k} \Big)_{d \times d}.
\end{align}
Let $\Pi$ denote the set of all topological orders of $V$. 
Furthermore, denote an equivalence class of topological orders induced by the underlying (unknown) DAG $\mathcal D=(V,E)$ by
\begin{align}\label{equivalence}
    R_{\D} :=\{\pi \in \Pi: \pi(j)<\pi(i)\quad  \mbox{ for all } (j,i) \in E\}.
\end{align}
By Lemma~\ref{lemma1}~d), 
$\check{b}_{ji}$ is lower bounded by $b_{ji}$ for $j \in \an(i)$ and  $\check{b}_{ji} \to 0$ a.s. as $\nto$ for $j \not\in \An(i)$. This is a direct result from Lemma~\ref{lemma1} c) and the fact that the minimum is non-increasing.
Hence, for any $\pi \in R_{\mathcal D}$ it holds that $\check{b}_{ji} \to 0$ a.s. as $n \to \infty$ whenever $\pi(j)>\pi(i)$. Therefore, also 
\begin{align}\label{maxel}
    \max\limits_{\substack{(j,i) \in V \times V: \\ \pi(j)> \pi(i)}}\check{b}_{ji} \to 0 \quad\mbox{ a.s. for }n \to \infty.
\end{align}
In contrast, for any $\pi \not \in R_D$, there is a pair of nodes $(j,i)$ such that $b_{ji}>0$ although $\pi(j)>\pi(i)$. For this reason, 
\begin{align}\label{maxel2}
    \max\limits_{\substack{(j,i) \in V \times V: \\ \pi(j)> \pi(i)}}\check{b}_{ji} \to c_\pi >0 \quad \mbox{ a.s. for }n \to \infty.
\end{align}

As a consequence, for a given topological order $\pi$, by \eqref{maxel} and \eqref{maxel2}, the maximum converges almost surely to zero if and only if $\pi \in R_{\mathcal D}$. 
Hence we propose a topological order that minimizes this expression, i.e.,
\begin{align}\label{mingreedy}
    \argmin\limits_{\pi \in \Pi}\max\limits_{\substack{(j,i) \in V \times V: \\ \pi(j)> \pi(i)}}\check{b}_{ji}.
\end{align}

A topological order found by \eqref{mingreedy} generally is not unique. Algorithm~\ref{alg1} returns a unique topological order for any fixed estimated matrix $\check{\bB}$.

\begin{algorithm}
\caption{Estimating a topological order }\label{alg1}
\begin{algorithmic}[1]
\REQUIRE A matrix of minimum ratios $\check{\bB}$ as in \eqref{minratios} 
\ENSURE An estimated topological order $\hat \pi$
\STATE  Set $ \check  \D=(V,E)$ with $V=\{1,\ldots,d\}$ and $E=\emptyset$.
\STATE  Set $S:=\{(j,i) \in V \times V: j \neq i\}$ and sort the elements $(j,i)$ of $S$ by the size of $\check{b}_{ji}$ from big to small.
\FOR{$(j,i)$ in S}
    \IF{$i \not\in\an(j)$ in $ \check \D$}
        \STATE $E=E\cup (j,i)$
    \ENDIF
\ENDFOR
\STATE \textbf{return} the topological order $\hat{\pi}$ of the DAG $ \check \D$
\end{algorithmic}
\end{algorithm}
The DAG $\check\D$ constructed in Algorithm~\ref{alg1} works as an auxiliary instrument to infer a topological order. Since $\check\D$ is complete with edges between every node pair in $V$, it returns a unique topological order. Moreover, since we sort the weights by size, the algorithm solves \eqref{mingreedy} in an optimal way for given $\check B$. 
At first sight the algorithm bears some similarity to Kruskal's classical algorithm for finding a minimum spanning tree; see \cite{Kruskal}.
However, Algorithm~\ref{alg1} works with directed edges and, of course, the optimization problem itself is very different.

Adding an edge and checking the presence of a path between any pair of nodes both can be implemented in $O(d)$ amortized complexity (see \cite{it1}). Hence, since $S$ as computed in line~2 of Algorithm~\ref{alg1} contains $d(d-1)$ pairs of nodes, we have an overall amortized complexity of $O(d^3)$.
After Algorithm~\ref{alg1} we can again use the minimum ratio estimator
\begin{align} \label{estimate3}
    \hat{\bB}:=(\hat{b}_{ij})_{d \times d}=\Big(\bigwedge\limits_{k=1}^n \frac{U_j^k}{U_i^k} \bone_{(\hat \pi(i)<\hat \pi(j))} \Big)_{d \times d}.
\end{align}

\subsection{Strong consistence of $\hat \bB$ and learning the minimum ML DAG $\mathcal D^B$}\label{s45} 

We first want to formally state the  a.s. convergence of the proposed estimators for the ML coefficient matrix $\bB$. Afterwards, we discuss how to learn the minimum ML DAG $\mathcal D^B$. The proofs of Proposition~\ref{a.s.convergence} and Lemma~\ref{alg2_lemma} can be found in Appendix~\ref{B}.

\begin{proposition} \label{a.s.convergence}
    Let $\bU\in\R^d_+$ be a recursive ML vector with propagating noise as defined in \eqref{1stdef_1} and let $\bU^1,\ldots,\bU^n \in \mathbb R_+^d$ be an iid sample from $\bU$. Then the estimates \eqref{estimate1}, \eqref{estimate2} and \eqref{estimate3} of $\bB$ are strongly consistent, i.e., it holds a.s. for $n \to \infty$ that
    \begin{align*}
        \hat b_{ji} \longrightarrow b_{ji} \quad \text{ for } j \in \an(i),  \quad \hat b_{ii}=1,\quad\text{and}\quad  \hat b_{ji} \longrightarrow 0 \quad \text{ for }   j\in V\setminus\An(i).
    \end{align*}
\end{proposition}
  
In sections~\ref{s51}-\ref{s53} we have been discussing how to estimate $\bB$ under the settings (1)-(3). 
However, as we know from Corollary~\ref{minD}, only critical edges of $\D$ contribute to the distribution of $\bU$. 
Asymptotically, we can almost surely identify $\mathcal D^B$ since there is an edge $j\to i$ in $\mathcal D^B$ if and only if $b_{ji}>b_{jl}b_{li}$ for all $l\in \de(j)\cap\an(i)$.

However, in real life we estimate the edges of $\mathcal D^B$ for a finite data set. Since  $\bigwedge_{k=1}^n  (U_i^k/ U_j^k)>0$ holds for all $n\in\N$ and all $i,j\in V$,
the estimators \eqref{estimate2} or \eqref{estimate3} result in a matrix representing a complete DAG.

Since small estimated values $\hat{b}_{ji}$ may well be 0 in the true model, we use a threshold $\delta_1>0$ with the aim to set an estimator $\hat{b}_{ji}<\delta_1$ equal to 0. 
However, setting single values $\hat{b}_{ji}:=0$ may destroy the idempotency of $\hat{\bB}$ since idempotency requires for any triple of nodes $(j,l,i)$,
\begin{align}\label{tripleB}
\hat{b}_{jl}\hat{b}_{li} &=\bigwedge\limits_{k=1}^n \frac{U_l^k}{ U_j^k} \bigwedge\limits_{k=1}^n  \frac{U_i^k}{ U_l^k} \leq \bigwedge\limits_{k=1}^n \frac{  U_l^k}{ U_j^n } \,\frac{ U_i^k}{  U_l^k}=\hat{b}_{ji}.
\end{align}
For the estimates however, it might be possible that $\hat{b}_{ji}<\delta_1$, while $\hat{b}_{jl}>\delta_1$ and $\hat{b}_{li}>\delta_1$. In this case, setting $\hat{b}_{ji}=0$ would result it $\hat{b}_{ji} < \hat{b}_{jl}\hat{b}_{li}$ violating \eqref{tripleB}. 
To preserve the idempotency of $\hat \bB$ while setting some small values to 0, we propose a simple adapted thresholding algorithm.

\begin{algorithm}
\caption{Thresholding while maintaining idempotency}\label{alg2}
\begin{algorithmic}[1]
\REQUIRE A topological order $\pi:1,\ldots,d$ and an idempotent estimate $\hat{\bB}$ as in \eqref{estimate2} or \eqref{estimate3} and a threshold value $\delta_1>0$
\ENSURE An idempotent estimate $\hat{ \bB}$
\STATE $E:=\{(j,i) \in V \times V:\sgn(\hat{b}_{ji})=1\text{ and } i \neq j\}$
\STATE  $\mathcal D:=(V,E)$
\STATE  $S:=\{(j,i) \in E:  0<\hat{b}_{ji}<\delta_1\}$
\STATE Sort the pairs $(j,i)$ in $S$ by the distance $i-j$ from low to high
\FOR{$(j,i)$ in S}
    \IF{$(j-i) == 1$}
        \STATE $\hat{b}_{ji}=0$
    \ENDIF
    \IF{for every $l$ with $j<l<i$: $(j,l)$ or $(l,i) \in S$} \label{If-condition}
        \STATE $\hat{b}_{ji}=0$ 
    \ELSE 
        \STATE $S=S\setminus \{(j,i)\}$
    \ENDIF
\ENDFOR
\STATE \textbf{return} $\hat{\bB}$
\end{algorithmic}
\end{algorithm}

\begin{lemma}\label{alg2_lemma}
    Algorithm~\ref{alg2} with threshold $\delta_1>0$ outputs an idempotent matrix, i.e.  $\hat{ \bB} \odot \hat{ \bB}= \hat{ \bB}$ and there is no other idempotent matrix $\bB'$ such that $b'_{ji}=\hat{b}_{ji}$ whenever $\hat{b}_{ji}>\delta_1$ that contains more zero entries than $\hat{ \bB}$.
\end{lemma}

\begin{remark}
If we choose $\delta_1\leq \min\{\check{b}_{ji}: j< i\}$ no entry is set to 0, and if $\delta_1> \max\{\check{b}_{ji}: j< i\}$ all entries are set to 0 except for the diagonal.
So in the first case, we obtain the complete DAG and in the second case the DAG consists of isolated nodes only. 
\end{remark}

In order to estimate the minimum ML DAG $\mathcal D^B$ it is not sufficient to decide if a path from $j$ to $i$ exists, i.e. if $b_{ji}>0$. 
We need in particular to decide if the edge $j\to i$ belongs to $\mathcal D^B$. By continuity of the noise variables we may observe for the estimated path weights
\begin{align*}
    \hat{b}_{ji}> \hat{b}_{jl}\hat{b}_{li}
\end{align*}
even if ${b}_{ji}={b}_{jl}{b}_{li}$. 
However, by Proposition~\ref{a.s.convergence}, in this situation the difference $ (\hat{b}_{ji}- \hat{b}_{jl}\hat{b}_{li}) \to 0$ a.s. as $n \to \infty$. 
Therefore, we introduce another threshold $\delta_2>0$ enforcing an edge in $\D^B$ if this difference is greater than $\delta_2$. 
In Theorem~\ref{epscount} we have seen that the distribution of the ratio $\mathbb P(U_i/U_j \leq b_{ji}x)$ is asymptotically determined by $\mathbb P (\prod_{k \in S_p} \varepsilon_k-1 \leq x)$ for $x \downarrow 0$. Hence, the rate of convergence of $(\hat{b}_{ji}- \hat{b}_{jl}\hat{b}_{li})$ depends crucially on the path length $m=\vert S_p \vert$. 
Ideally, we therefore choose $\delta_2=\delta_2(n,m)$ depending not only on the sample size $n$, but also on the path length $m$. 

More precisely, since $F^{\leftarrow}_{\sum_{k\in S_p}\tilde\eps_k}(1/n) \sim F^{\leftarrow}_{\prod_{k\in S_p}\eps_k-1}(1/n)$ (see Theorem~\ref{3.11} and its proof below), and assuming that $\bC$ is generic, we find that
Algorithm~\ref{alg3} asymptotically identifies $\mathcal D^B$, if 
    \begin{align*}
        F^\leftarrow_{\sum\limits_{k \in S_p}\tilde \varepsilon_k}(1/n) = o(\delta_2(n,m))\quad 
        \mbox{ for } \quad n\to\infty.
    \end{align*}
In real life we do not know the number of critical edges in either of the three settings. We distinguish between setting (1) and settings (2)-(3) and propose Algorithm~\ref{alg3} with $\delta_2(m):=\delta_2(n,m)$, i.e., for a fixed sample size $n$ we focus on the path length $m$.

\begin{algorithm}
\caption{Approximating max-weighted paths}\label{alg3}
\begin{algorithmic}
\REQUIRE Threshold sequences $\delta_2(1),\ldots,\delta_2(d)$ and \\
(a): a known underlying DAG $\mathcal D:=(V,E)$ and an estimate $\hat{\bB}$ as in \eqref{estimate1}, or \\
(b): a (known or estimated) topological order $\pi:1,\ldots,d$ and an estimate $\hat{\bB}$ as in  \eqref{estimate2} or \eqref{estimate3} 
\ENSURE An estimated minimum DAG $\mathcal D^{\hat{B}}=(V, E^{\hat B})$
\STATE $E^{\hat B}:=\emptyset$ and  $\mathcal D^{\hat B}:=(V, E^{\hat B})$
\STATE (1):\qquad $S:=\{(j,i) \in V\times V: j \in \pa(i) \}$ and infer a topological oder $\pi:1,\ldots,d$ from $\mathcal D$
\STATE  (2)-(3): $S:=\{(j,i) \in V\times V: j < i \}$
\STATE Sort pairs $(j,i)$ in $S$ by their distance $(i-j)$ according to the topological order from low to high
\FOR{$(j,i)$ in S}
    \IF{$\exists$ path $p$ from $j$ to $i$ in $\mathcal D^{\hat B}$}
        \STATE Set $m$ as the maximum path length in $\mathcal D^{\hat B}$
        \STATE Set $l:=\argmax\limits_{l \in V\setminus\{j,i\}}\left(\check{b}_{jl}\check{b}_{li}\right)$
        \IF{ ($\check{b}_{ji}- \check{b}_{jl}\check{b}_{li})>\delta_2(m)$}
            \STATE $E^{\hat B}:=E^{\hat B}\cup \{(j,i)\}$
        \ENDIF
    \ELSE 
        \IF{ $\check{b}_{ji}>0$}
            \STATE $E^{\hat B}:= E^{\hat B}\cup \{(j,i)\}$
        \ENDIF
    \ENDIF
\ENDFOR
\STATE \textbf{return} $\mathcal D^{\hat{B}}=(V, E^{\hat B})$
\end{algorithmic}
\end{algorithm}

For setting (1) we do know the underlying unweighted DAG $\mathcal D$. Therefore, we do not need to decide whether some small value $\hat b_{ji}$ corresponds to a path from $j$ to $i$. However, we do not know the minimum ML DAG $\mathcal D^B$ such that we would apply Algorithm~\ref{alg3} to estimate $\mathcal D^B$. For settings (2) and (3) we would apply first Algorithm~\ref{alg2} and afterwards Algorithm~\ref{alg3}.

In the next section we derive the asymptotic distribution of the estimators.

\section{Asymptotic distribution of the minimum ratio estimators}\label{s5}

With the goal of proving  asymptotic distributional properties of the minimum ratio estimators for the different settings (1)-(3), 
we require regular variation of the noise variable $\eps$ in its left endpoint. 
Under this condition we first prove that also the minimum ratio estimators $\bigwedge_{k=1}^n (U_i^k/U_j^k)$ are regularly varying. 
Moreover, we show that their joint limit distribution is the product of Weibull distributions. In this section we assume $\bC$ is generic in the sense of Definition~\ref{def:critical}. The results can be extended to a non-generic model by similar methods as used in Corollary~\ref{epscount_prop}.

We first recall the family of Weibull distribution functions, which will act as limit distributions for the estimators of $\bB$, whose strong consistency we have already proved in section~\ref{s45}.

\begin{definition}\label{defn:5.1}
    A positive random variable $Y$ is \emph{Weibull distributed with shape $\alpha>0$ and scale $s>0$} and we write $Y \sim$ Weibull$(\alpha,s)$ if the distribution function of $Y$ is given by \begin{align*}
        \Psi_{\alpha,s}(x)= 1-\exp\left(-\left(\frac{x}{s}\right)^{ \alpha}\right), \quad x > 0.
    \end{align*}
\end{definition}

Next we define regular variation in 0, which transforms to regular variation in 1 (or any other point) and at $\infty$ by the usual transformations (see \cite{BGT} for details).

\begin{definition}
    Let $Y$ be a positive random variable with distribution function $F$. Then we call $Y$ or $F$ {\em regularly varying at zero with exponent $\alpha>0$}, if
    \begin{align}\label{rv0} 
    \lim\limits_{t \downarrow 0}\frac{F(tx)}{F(t)}=x^{\alpha}, \quad x>0. 
\end{align}
We abbreviate this by $Y\in  RV_{\alpha}^0$ or $F\in  RV_{\alpha}^0$, respectively.
\end{definition}

In what follows we assume that the random variables $\tilde{\varepsilon}_i:=\ln(\varepsilon_i)>0$ for $i=1,\ldots,d$ are iid regularly varying at zero with exponent $\alpha>0$ and remark in passing that, by  a Taylor expansion of $\ln(\eps)$ at one, this is equivalent to $(\varepsilon-1)\in RV_{\alpha}^0$ or $\varepsilon\in RV_{\alpha}^1$.

Two families of distribution functions such that $\tilde\eps=\ln(\eps)\in RV_{\alpha}^0$ are given in the next example.

\bexam 

(a) \, [Gamma distribution] Let $\tilde\eps$ have density $g(x)=\la^\al e^{-\la x} x^{\al-1}/\Gamma(\al)$ for $x>0$ and parameters $\la>0,\al> 0$. 
Then by a l'Hospital argument, 
$$\lim_{t\downarrow 0} \frac{G(tx)}{G(t)}= \lim_{t\downarrow 0} \frac{e^{-\la tx}t^{\al-1}x^{\al}}{e^{-\la t}t^{\al-1}} = x^{\al}, \quad x>0,$$
which implies that $\tilde\varepsilon \in RV^0_{\alpha}$ and hence $\varepsilon \in RV^1_{\alpha}$. 

(b) \, [Weibull distribution] Let $\tilde\eps$ have density
$g(x)=\alpha s ^{-\alpha} x^{\alpha-1} \exp{(-(x/s)^\alpha)}$ for $x > 0$ and parameters $\alpha>0$ and $s>0$. 
Then again by a l'Hospital argument, 
$$\lim_{t\downarrow 0} \frac{G(tx)}{G(t)} = x^\alpha,\quad x \geq 0,$$
which implies that $\tilde\varepsilon \in RV^0_{\alpha}$ and hence $\varepsilon \in RV^1_{\alpha}$.
 \halmos
\eexam

We first prove that $\ln( U_i/ U_j)-\ln(b_{ji})$ is regularly varying at zero which will be a consequence of Theorem~\ref{epscount}. 
In this auxiliary result as well as in the theorems below we need that $\bC$ is generic. Further, for a path $p$ we denote by $\zeta(p)=\vert S_p\vert$ its path length.

\begin{lemma}\label{ratiorv}
    Let $\bU\in\R^d_+$ be a recursive ML vector with propagating noise on a DAG $\mathcal D$ as defined in \eqref{1stdef_1} and
    assume that the path $p:=[j\to \ldots \to i]$ from $j$ to $i$ is generic. 
 If $\ln(\varepsilon) \in RV_\alpha^0$, then
 $\ln( U_i/ U_j)-\ln(b_{ji}) \in RV_{\zeta(p)\alpha}^0$. 
\end{lemma}

The following is the main result of this section and describes the asymptotic distribution of the minimum ratio estimator $\hat \bB$ from \eqref{estimate3}.
In particular, it shows that its entries are asymptotically independent.

\begin{theorem}\label{3.11}
    Let $\bU\in\R^d_+$  be a recursive ML vector with propagating noise as defined in \eqref{1stdef_1}.  Assume that $\bC$ is generic and that  $\tilde \varepsilon=\ln(\varepsilon) \in RV_\alpha^0$. For every path $p_{ji}$ from $j$ to $i$ and node set $S_{p_{ji}}$
choose $a^{(ji)}_n\sim F_{\sum_{k \in S_{p_{ji}}}\tilde \varepsilon_k}^\leftarrow(1/n)$ as $\nto$. If $\bU^1,\ldots,\bU^n$ is an iid sample from $\bU$, then 
    \begin{align*}
        &\lim\limits_{n \to \infty}\mathbb P \left(\frac1{a^{(ji)}_n b_{ji}}\Big(
        \bigwedge\limits_{k=1}^n \frac{U_i^k}{ U_j^k}-b_{ji}\Big)\leq x_{ji}\, \forall  (j,i)\in V \times V\text{ with } b_{ji}>0 \right) \\
        = &\prod\limits_{\substack{(j,i) \in V \times V: \\ b_{ji}>0}}\Psi_{\zeta(p_{ji})\alpha,(c^{(ji)})^{1/(\zeta(p_{ji})\alpha)}}\left(x_{ji}\right), \quad x_{ji}>0,
    \end{align*}
   where $c^{(ji)}\in(0,1)$ is defined as in Theorem~\ref{epscount}.
\end{theorem}

If we know the minimum ML DAG $\mathcal D^B=(V, E(\D^B))$, it is preferable to estimate $b_{ji}$ as in \eqref{estimate1}. Then Theorem~\ref{3.11} reduces as follows.

\begin{corollary}
Let the assumptions of Theorem~\ref{3.11} hold and assume that the minimum ML DAG $\D^B(V, E(\D^B))$ is known. Then 
    \begin{align*}
        &\lim\limits_{n \to \infty}\mathbb P \left(\frac1{a^{ji}_n b_{ji}}\Big(
        \bigwedge\limits_{k=1}^n \frac{U_i^k}{ U_j^k}-b_{ji}\Big)\leq x_{ji}\, \forall  (j,i) \in  E(\mathcal D^B) \right)\\
        = &\prod\limits_{(j,i) \in E(\D^B)} \Psi_{\alpha,(c^{(ji)})^{1/\alpha}}\left(x_{ji}\right),\quad x_{ji}>0.
    \end{align*}
\end{corollary}

\section{Data analysis and simulation study}\label{s6}

We want to apply the methods that we have developed over the past sections and consider a data example. For a quality assessment we also perform a simulation study.

\subsection{Data example}

We consider dietary supplement data of $n=8327$ independent patients taken from a dietary interview from the NHANES report for the year 2015-2016, which is available at \href{https://wwwn.cdc.gov/Nchs/Nhanes/2015-2016/DR1TOT\_I.XPT}{https://wwwn.cdc.gov/Nchs/Nhanes/2015-2016/DR1TOT\_I.XPT}. 
The data contains 168 food components with the object of estimating the total intake of calories, macro and micro nutrients from foods and beverages consumed a day prior to the interview. More details can be found on the website.

In \cite{janssen2019kmeans}, the data set has been considered in terms of an adapted $k$-means clustering algorithm for extremal observations. Moreover, assuming a recursive ML model and standardising the marginal data to regular variation at $\infty$ with $\alpha=2$, \cite{krali} investigated the causal relationship between four nutrients using a different estimation method based on scalings.  

In our data example we consider the same four nutrients, namely vitamin A (DR1TVARA), $\alpha$-carotene (DR1TACAR), $\beta$-carotene (DR1TBCAR) and lutein+zeaxanthin (DR1TLZ) as in \cite{krali}. We abbreviate them by VA, AC, BC and LZ. In order to make results comparable to those of \cite{krali}, we also use the empirical integral transform to standardize the data to Fr\'echet(2) margins (see e.g. \cite{beirlant},  p. 381) by setting for $i=1,2,3,4$,
\begin{align*}
    U_{li}:=\Big(-\log\Big(\frac 1 {n+1}\sum_{j=1}^n\bone_{\{\bar U_{ji}\leq \bar U_{li}\}}\Big)\Big)^{-1/2},\quad l=1,\ldots,n=8327,
\end{align*}
where multiple ranks are uniformly randomly ordered.

We first consider the full matrix of minimum ratios $\check \bB= (\check{b}_{ij})_{d\times d}$ with $\check{b}_{ij}=\bigwedge_{t=1}^n (X^t_j/X^t_i)$ given by
  \[
\begin{blockarray}{ccccc}
{\rm VA} & AC & BC & LZ \\
\begin{block}{(cccc)c}
  1 & 0.014 & 0.011 & 0.007 & VA \\
  0.146 & 1 & 0.177 & 0.019 & AC \\
  0.321 & 0.010 & 1 & 0.025 & BC \\
  0.132 & 0.007 & 0.168 &  1 & LZ \\
\end{block}
\end{blockarray}
 \] 
We next apply Algorithm~\ref{alg1} to obtain an estimated topological order $\hat \pi:=(AC, LZ, BC, VA)$. 
First we want to assess the quality of the estimated topological order $\hat \pi$, 
which also supports or contradicts the model assumption of a Bayesian network.
Motivated by the coefficient $R^2$ of determination in regression we define the following.

\begin{definition}\label{mlcoeff}
For a given topological order $\pi$ and an estimator $\check{B}$ of the ML coefficient matrix we define the \emph{ML coefficient of determination}  
\begin{align*}
  R_{\max}( \pi)=\frac{ \sum\limits_{\substack{(j,i) \in V \times V: \\  \pi(j)<  \pi(i)}}\check{b}_{ji}}{\sum\limits_{\substack{(j,i) \in V \times V: \\ j \neq i}}\check{b}_{ji}}.
\end{align*}
\end{definition}

The coefficient $R_{\max}( \pi)$ can take any value in the interval $[0,1]$. Large $R_{\max}( \pi)$ supports the hypothesis that the underlying graph is a DAG and the estimated topological order lies in the equivalence class of topological orders defined in \eqref{equivalence}. 

In our data example, we have  $R_{\max}(\hat \pi)=0.929$, strongly supporting the hypothesis of a recursive ML model.
Now using the estimator $\eqref{estimate3}$, and applying Algorithms~\ref{alg2} and \ref{alg3} with $\delta_1=0.02$ and $\delta_2(k)=0.02$ for $k \in\{1,2,3\}$, we get the estimated minimum ML DAG $\mathcal D^{\hat B}$ and ML coefficient matrix $\hat \bB$, where we sorted the matrix according to $\hat \pi$. These are shown in Figure \ref{est_dag}.

{
\makeatletter
\def\@captype{figure}
\makeatother
\begin{wrapfigure}{l}[10pt]{7.5cm}
\centering
    \begin{tikzpicture}[->,every node/.style={circle,draw},line width=1pt, node distance=1.7cm]
        \node (1)  {AC};
        \node (2) [below right of=1] {LZ};
        \node (3) [below left of=1] {BC};
        \foreach \from/\to in {2/3,1/3}
        \draw (\from) -- (\to);   
        \node (4) [below left of=2] {VA}; 
        \foreach \from/\to in {1/4,2/4,3/4}
        \draw (\from) -- (\to);   
        \path[every node/.style={font=\small}];
    \end{tikzpicture} 
\end{wrapfigure}
\begin{align} \label{estKlenneStar}
       \hat \bB = \quad
\begin{blockarray}{ccccc}
{\rm AC} & {\rm LZ} & {\rm BC} & {\rm VA} \\
\begin{block}{(cccc)c}
  1 & 0 & 0.177 & 0.146 & {\rm AC} \\
  0 & 1 & 0.168 & 0.132 & {\rm LZ} \\
  0 & 0 & 1 & 0.321 & {\rm BC} \\
  0 & 0 & 0 &  1 & {\rm VA} \\
\end{block}
\end{blockarray}
 \end{align}
 \caption{Estimated minimum ML DAG $\mathcal D^{\hat B}$ with estimated ML coefficient matrix $\hat B$.}\label{est_dag}
}

\noindent
 Observe that, since we estimate the edge from AC to LZ to be absent, there are two possible topological orders. 

From the estimates we observe that both, $\alpha$-carotene and $\beta$-carotene lead to high amounts of vitamin A. This is in line with our expectation since $\beta$-carotene is a precursor to vitamin A and can be converted by $\beta$-carotene 15,15'-monoxygenase by many animals including humans. Similarly, also $\alpha$-carotene can be converted to vitamin A. However, it is only half as active as $\beta$-carotene which explains that the edge weight from $\alpha$-carotene to vitamin A is approximately half compared to the edge weight from $\beta$-carotene to vitamin A (0.146 compared to 0.321). 
Moreover, we can see that high amounts of lutein+zeaxanthin also lead to high amounts of $\beta$-carotene and high amounts of $\alpha$-carotene also lead to high amounts of $\beta$-carotene. However, we did not find a significant connection between $\alpha$-carotene and lutein+zeaxanthin. Observe that \cite{krali} inferred the same topological order, yet with one additional edge from $\alpha$-carotene to lutein+zeaxanthin. However, it is also the edge with the smallest estimated edge weight. Similarly as in \cite{krali}, we plot bivariate extremes in Figure~\ref{bivariateplots} to underline our finding. 
The first 5 plots in Figure~\ref{bivariateplots} look rather similar. For every large value of the substance on the vertical axis, we can see a large value of the substance on the horizontal axis. Moreover, these observations are shaped closely to a line. In contrast, a large value of the substance on the horizontal axis might as well coincide with a small value of the substance on the vertical axis. Therefore,  e.g. a high amount of $\alpha$-carotene leads to a high amount of vitamin A but a high amount of vitamin A does not necessarily lead to a high amount of $\alpha$-carotene. This also supports that the dependence is not mutual and hence we can model it by a DAG. The same can be seen for any pair given in the plots 1-5. Moreover, since parts of the observations are shaped closely along a line, which we would expect for a recursive ML model, we can conclude that the recursive ML model fits the data very well. 

The sixth plot is different from the other 5 plots, since for most large observations of $\alpha$-carotene the level of lutein+zeaxanthin is not increased as most large observations in lutein+zeaxanthin do also not result in a high level of $\alpha$-carotene. Therefore, the two substances do not seem to affect each other and we rightly concluded that there is no edge.

\begin{figure}
    \centering
    \begin{subfigure}[b]{\textwidth}
        \centering
        \includegraphics[width=0.475\linewidth]{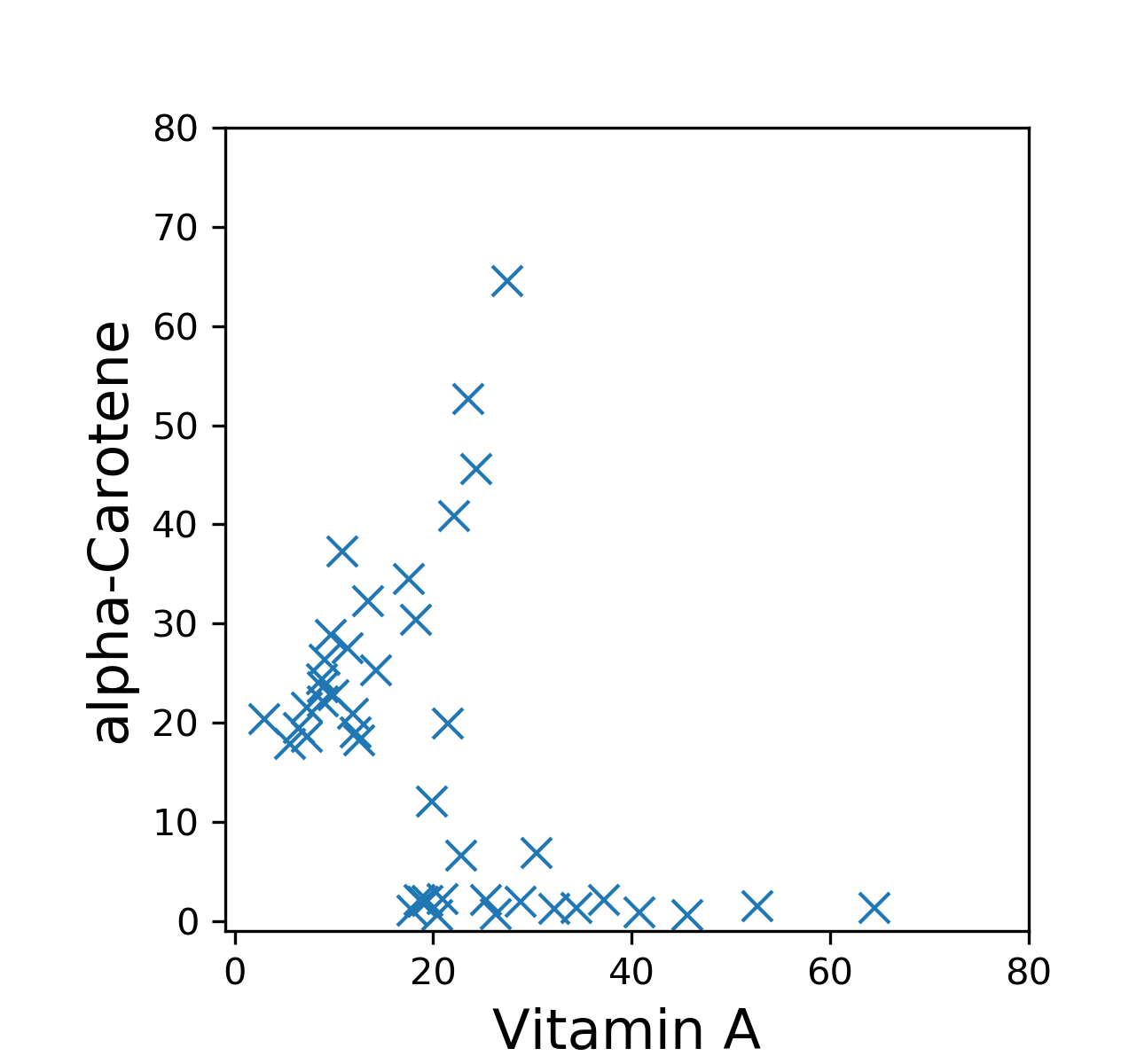}%
        \hfill
        \includegraphics[width=0.475\linewidth]{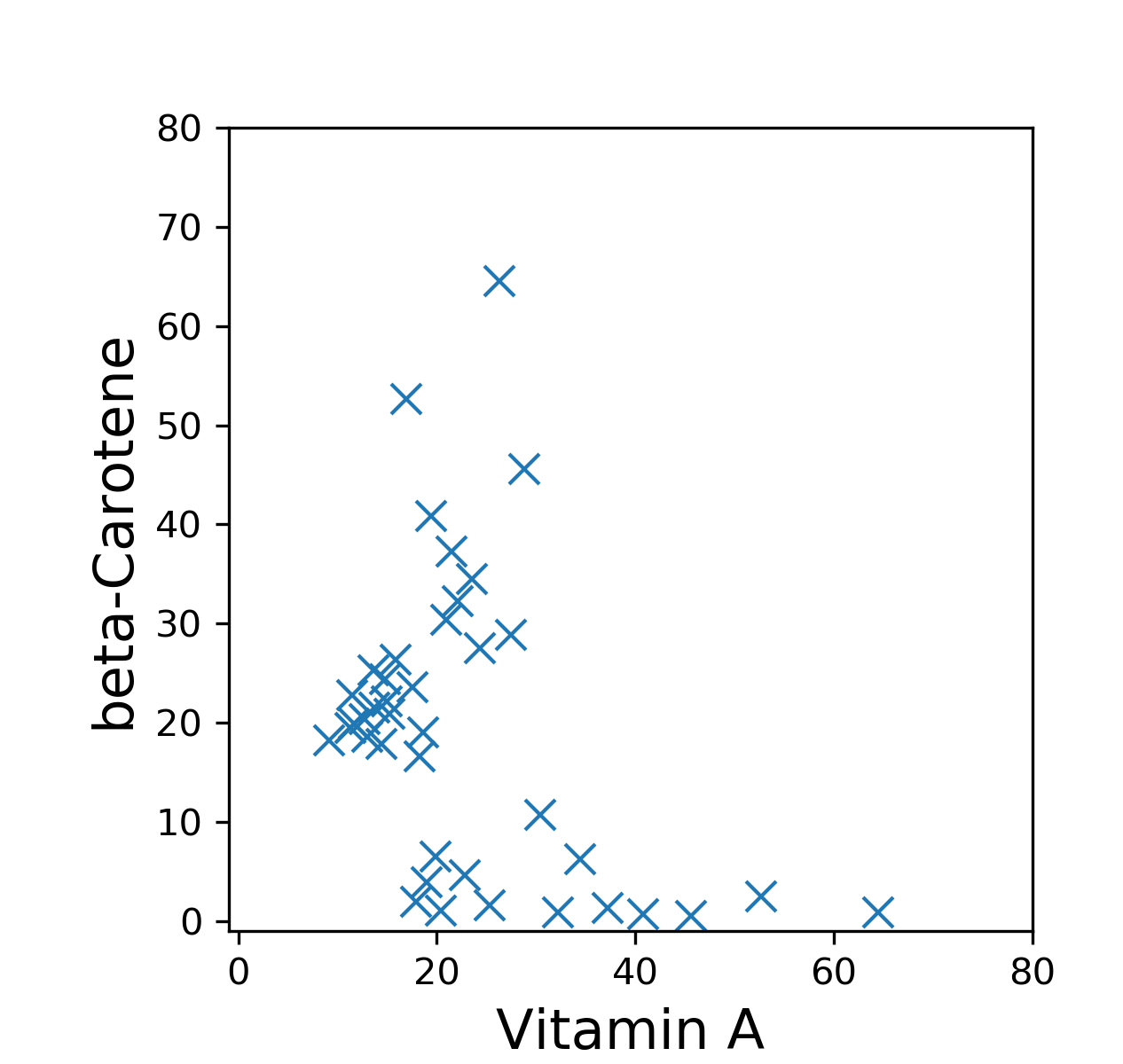}
    \end{subfigure}
    \vspace{0.2cm}
    \begin{subfigure}[b]{\textwidth}
        \centering
        \includegraphics[width=0.475\linewidth]{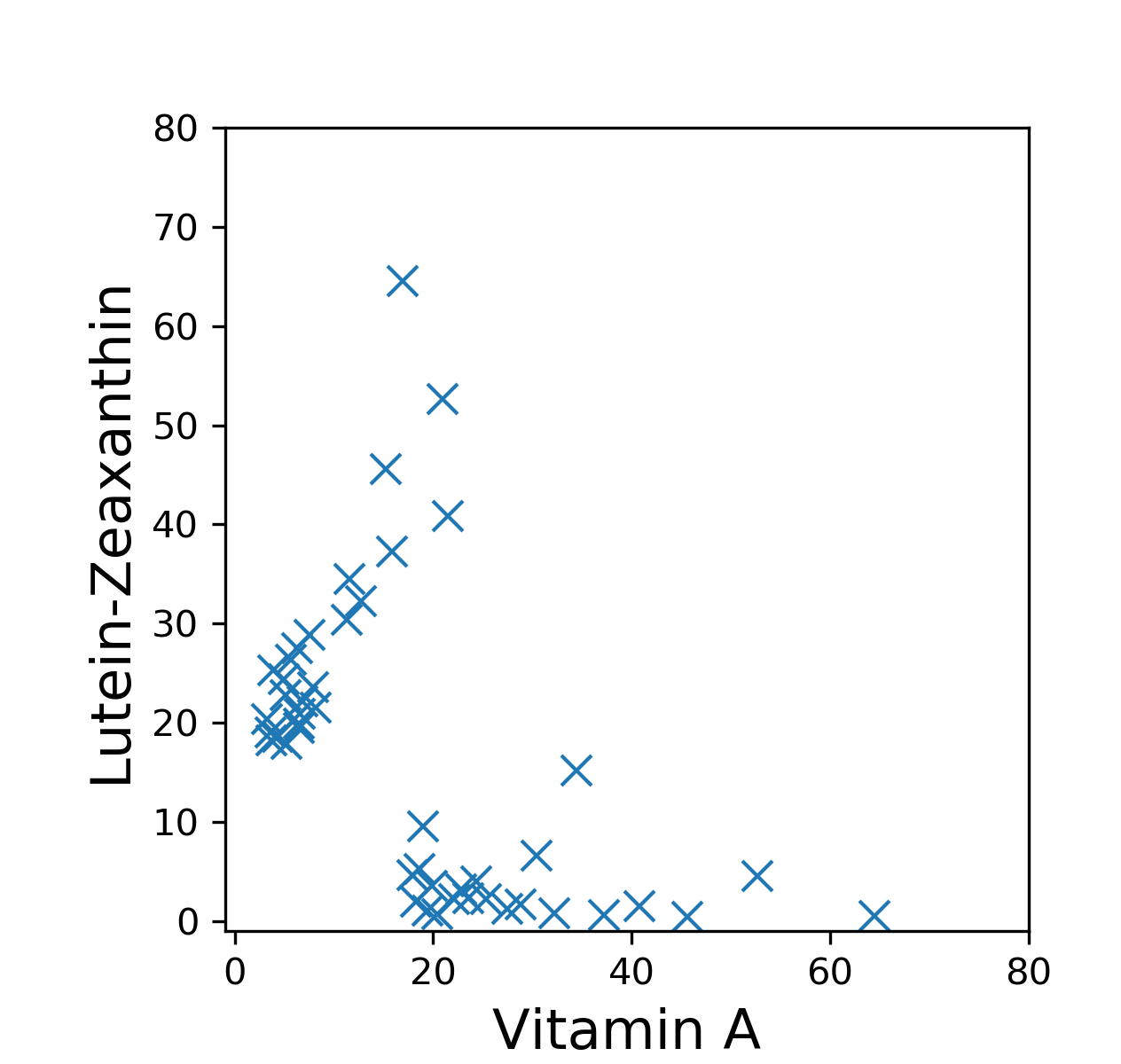}%
        \hfill
        \includegraphics[width=0.475\linewidth]{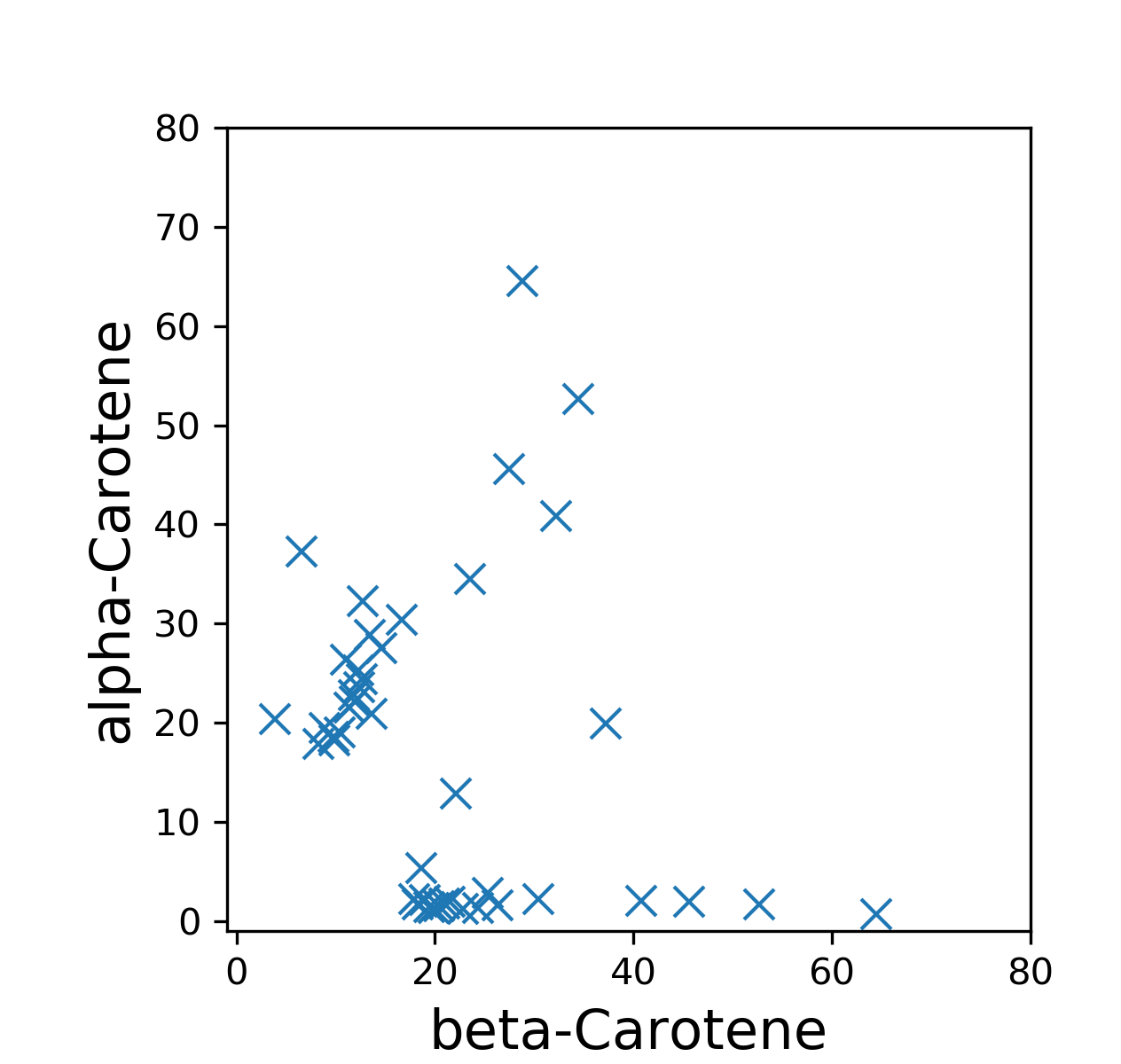}
    \end{subfigure}
    \vspace{0.2cm}
    \begin{subfigure}[b]{\textwidth}
        \centering
        \includegraphics[width=0.475\linewidth]{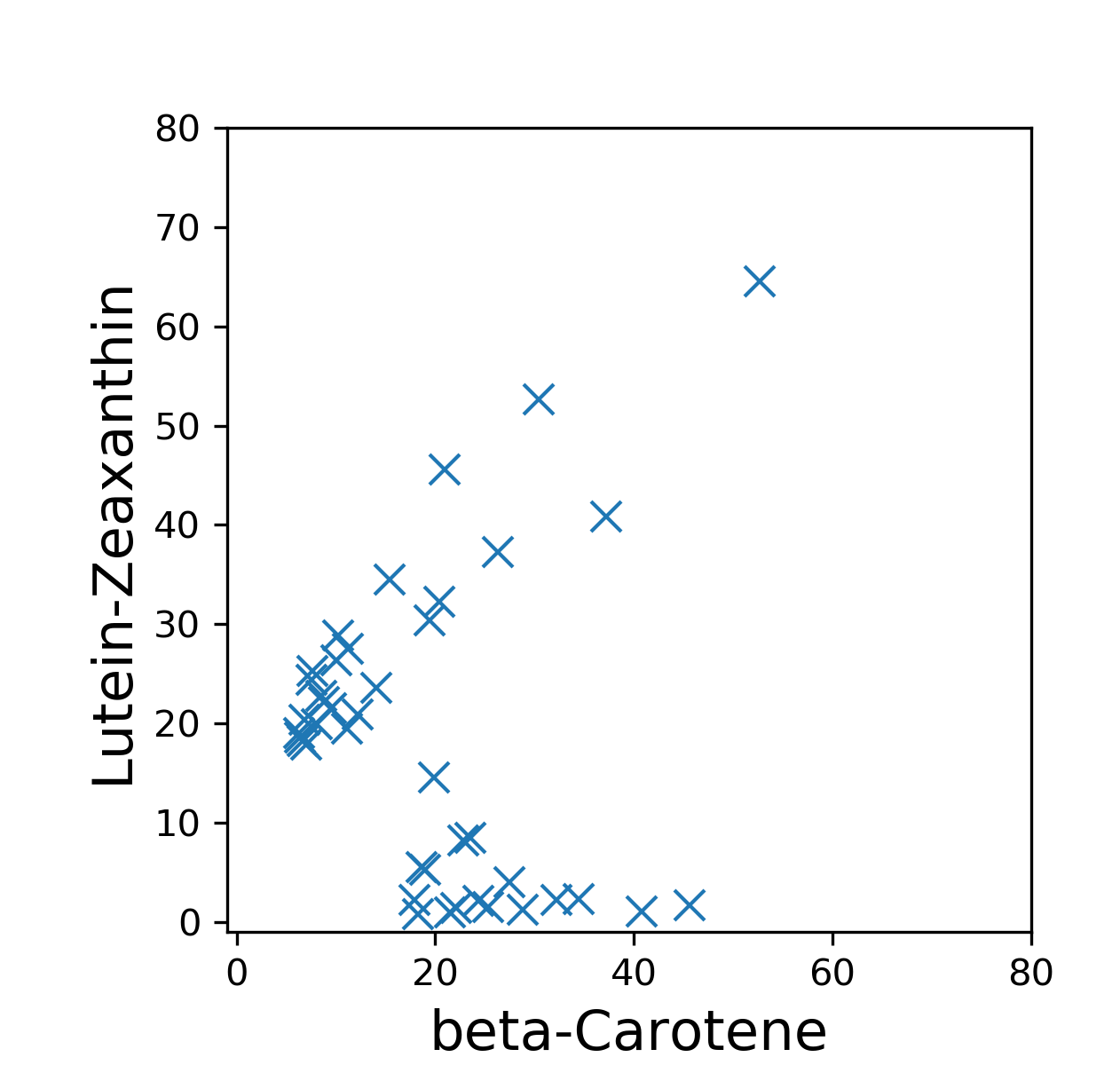}%
        \hfill
        \includegraphics[width=0.475\linewidth]{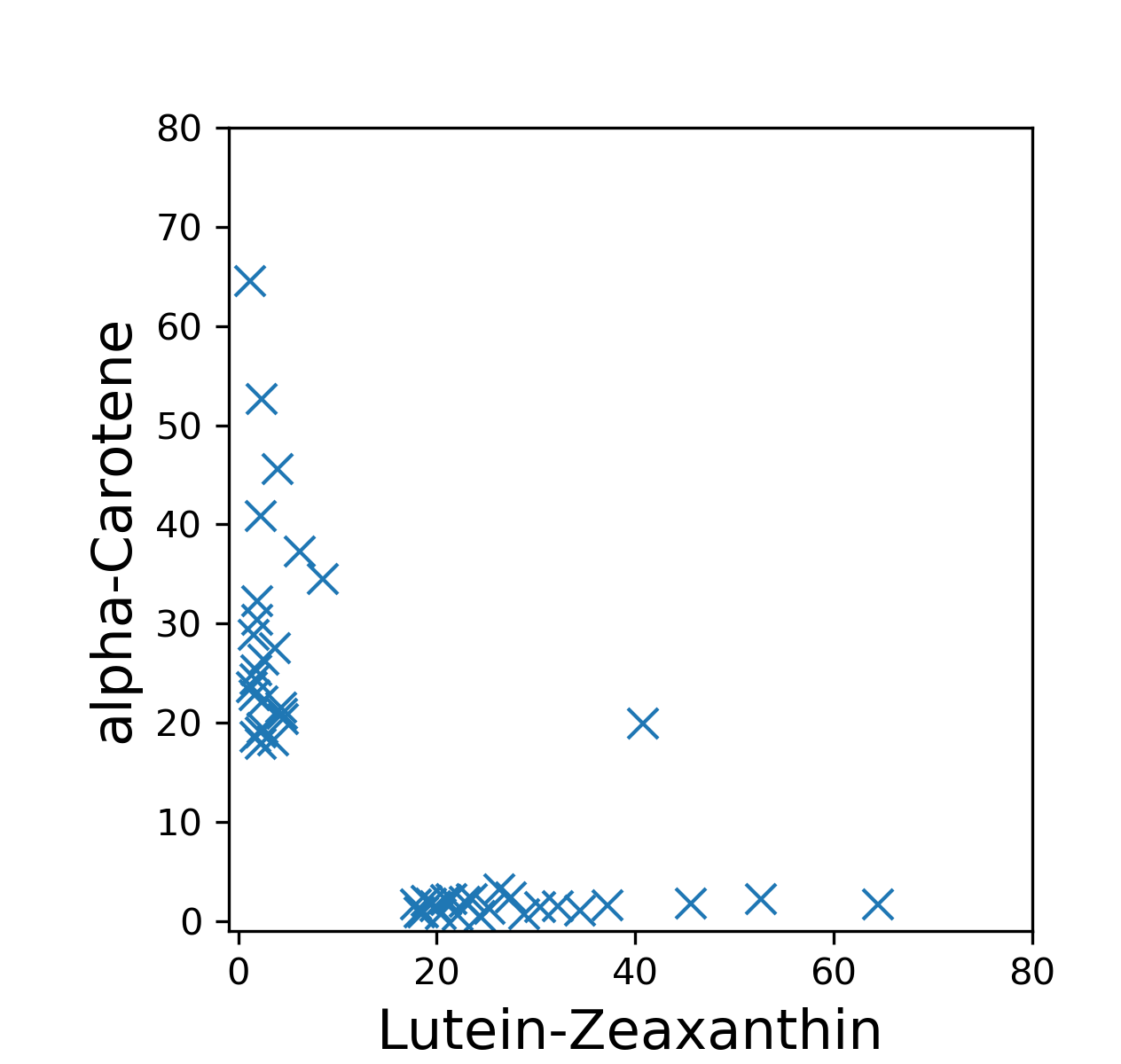}
    \end{subfigure}
    \caption{The empirical bivariate extremes (25 largest observations).} \label{bivariateplots}
\end{figure}

\subsection{Simulation study}

We want to illustrate the effect of observational noise in the ML model. We simulate recursive ML vectors with propagating noise, where the innovations $Z_1,\dots,Z_4$ are Fr\'echet(2) distributed and we use the estimated $\hat \bB$ from \eqref{estKlenneStar} from the data analysis above for the ML coefficient matrix $\bB$. Moreover, we simulate three different scenarios. In the first scenario, we assume the non-noisy model as given in \eqref{nonnoise}, while for the second scenario we choose the propagating noise model with a medium sized noise and in the third setting we choose a noise variable which is stochastically larger. The scenarios are given as follows:
\begin{enumerate}
    \item[(1)] No noise
    \item[(2)] $\ln(\varepsilon_i)\sim\text{Gamma}(\lambda=1, \alpha=2)$ for $i \in \{1,2,3,4\}$, which corresponds to $\mathbb E( \varepsilon_i) =2$
    \item[(3)] $\ln(\varepsilon_i)\sim\text{Gamma}(\lambda=2, \alpha=2)$ for $i \in \{1,2,3,4\}$, which corresponds to $\mathbb E (\varepsilon_i) =4$
\end{enumerate}
We assume to have no information on the underlying DAG and we only consider the quality of the estimator $\check b_{ji}$ given in  \eqref{minratios}. We choose the sample sizes $n \in \{50, 200, 500, 1000\}$ and 1000 simulation runs for each sample size. 
We first assess the success probabilities for Algorithm~\ref{alg1}. 
Table~\ref{table:1} shows that the topological order can be correctly estimated even for small sample sizes. Moreover, the number of correct runs increases for larger noise variables. This is expected since the noise variables are one-sided. Therefore, for a path $p$ from $j$ to $i$ the ratio  $U_i/U_j \geq d_{ji}(p)\prod_{k\in S_p}\eps_k$ increases, while the ratio $U_j/U_i \leq 1/(d_{ji}(p)\prod_{k\in S_p}\eps_k)$ decreases. Therefore, it is easier to identify the paths in $\mathcal D$ for larger noise. 

\begin{table}
\centering
\begin{tabular}{|c c c c|} 
 \hline
 Sample Size & Correct Runs(1) & Correct Runs(2) & Correct Runs(3)  \\ [0.5ex] 
 \hline
 50 & 555 & 946 & 998 \\ 
 200 & 995 & 1000 & 1000 \\
 500 & 1000 & 1000 & 1000 \\
 1000 & 1000 & 1000 & 1000 \\ [1ex] 
 \hline
\end{tabular}
\caption{Empirical success probability for estimated topolgical order being in the equivalence class of topological orders for (1) No noise, (2): Gamma(1,2), (3); Gamma(2,2).}
\label{table:1}
\end{table}

Next, we want to assess the quality of the estimated ML coefficient matrix $\check \bB$. To do so, for every pair $(j,i)$ with $b_{ji}>0$ and every simulation run $k \in \{1, \ldots,1000\}$, we denote the minimum ratio estimator given in \eqref{minratios} by $\check b_{ji}^k$.

We consider the empirical RMSE, standard deviation and bias for each $b_{ji}>0$ in each model (1)-(3). 
In what follows we compare the three classical quantities
\begin{align}
     {\text{bias}}(\check b_{ji})&:=\frac1{1000}\sum\limits_{k=1}^{1000} \check b^k_{ji}-b_{ji}, \label{bias}\\
      {\text{SD}}(\check b_{ji})& := \sqrt{\frac1{1000} \sum\limits_{k=1}^{1000} (\check b^k_{ji}-\ov {\check b}_{ji})^2}\quad\mbox{with}\quad \ov {\check b}_{ji} = \frac1{1000}\sum_{k=1}^{1000} \check b^k_{ji}, \label{SD}\\
    {\text{RMSE}}(\check b_{ji})& :=\sqrt{\frac1{1000} \sum\limits_{k=1}^{1000} (\check b^k_{ji}-b_{ji})^2}, \label{RMSE}
\end{align}

All three quantities are comparatively small even for small sample sizes and decrease whenever the sample size increases. 
Moreover, they are larger in the propagating noise model and larger noise terms also increase the three quantities.
This is in line with what we can expect from the model as noise terms increase the ratios $U_i/U_j$ and hence also increase the minimum ratio estimator. On the other hand, recall from above that with increasing noise the estimation of the DAG improves. 

\begin{table}[H]
\centering
\begin{tabular}{|c c c c c c|} 
 \hline
 Sample Size & Edge & Edge Weight & Bias(1) & Bias(2) & Bias(3)  \\ [0.5ex] 
 \hline
 50 & $AC \to BC$ & 0.177 & 0.012 & 0.020 & 0.046 \\ 
 50 & $AC \to VA$ & 0.146 & 0.028 & 0.029 & 0.063 \\
 50 & $LZ \to BC$ & 0.168 & 0.013 & 0.020 & 0.045 \\
 50 & $LZ \to VA$ & 0.132 & 0.027 & 0.029 & 0.064 \\
 50 & $BC \to VA$ & 0.321 & 0 & 0.014 & 0.053 \\
 200 & $AC \to BC$ & 0.177 & 0 & 0.005 & 0.019 \\ 
 200 & $AC \to VA$ & 0.146 & 0 & 0.007 & 0.026 \\
 200 & $LZ \to BC$ & 0.168 & 0 & 0.005 & 0.020 \\
 200 & $LZ \to VA$ & 0.132 & 0.001 & 0.007 & 0.026 \\
 200 & $BC \to VA$ & 0.321 & 0 & 0.004 & 0.023 \\
 500 & $AC \to BC$ & 0.177 & 0 & 0.002 & 0.012 \\ 
 500 & $AC \to VA$ & 0.146 & 0 & 0.003 & 0.016 \\
 500 & $LZ \to BC$ & 0.168 & 0 & 0.002 & 0.012 \\
 500 & $LZ \to VA$ & 0.132 & 0 & 0.003 & 0.016 \\
 500 & $BC \to VA$ & 0.321 & 0 & 0.001 & 0.015 \\
 1000 & $AC \to BC$ & 0.177 & 0 & 0.001 & 0.008 \\ 
 1000 & $AC \to VA$ & 0.146 & 0 & 0.001 & 0.011 \\
 1000 & $LZ \to BC$ & 0.168 & 0 & 0.001 & 0.008 \\
 1000 & $LZ \to VA$ & 0.132 & 0 & 0.001 & 0.010 \\
 1000 & $BC \to VA$ & 0.321 & 0 & 0.001 & 0.010 \\[1ex] 
 \hline
\end{tabular}
\caption{Empirical Bias \eqref{bias} for (1): No noise, (2): Gamma(1,2), (3): Gamma(2,2)}
\label{table:4}
\end{table}

\begin{table}[H]
\centering
\begin{tabular}{|c c c c c c|} 
 \hline
 Sample Size & Edge & Edge Weight & Std(1) & Std(2) & Std(3)  \\ [0.5ex] 
 \hline
 50 & $AC \to BC$ & 0.177 & 0.031 & 0.023 & 0.031 \\ 
 50 & $AC \to VA$ & 0.146 & 0.046 & 0.033 & 0.041 \\
 50 & $LZ \to BC$ & 0.168 & 0.031 & 0.022 & 0.031 \\
 50 & $LZ \to VA$ & 0.132 & 0.046 & 0.033 & 0.043 \\
 50 & $BC \to VA$ & 0.321 & 0.006 & 0.015 & 0.031 \\
 200 & $AC \to BC$ & 0.177 & 0.001 & 0.005 & 0.011 \\ 
 200 & $AC \to VA$ & 0.146 & 0.002 & 0.006 & 0.015 \\
 200 & $LZ \to BC$ & 0.168 & 0.002 & 0.005 & 0.012 \\
 200 & $LZ \to VA$ & 0.132 & 0.005 & 0.008 & 0.017 \\
 200 & $BC \to VA$ & 0.321 & 0 & 0.004 & 0.013 \\
 500 & $AC \to BC$ & 0.177 & 0 & 0.002 &  0.007 \\ 
 500 & $AC \to VA$ & 0.146 & 0 & 0.003 &  0.009 \\
 500 & $LZ \to BC$ & 0.168 & 0 & 0.002 &  0.007 \\
 500 & $LZ \to VA$ & 0.132 & 0.001 & 0.003 &  0.009 \\
 500 & $BC \to VA$ & 0.321 & 0 & 0.001 &  0.008 \\
 1000 & $AC \to BC$ & 0.177 & 0 & 0.001 & 0.004 \\ 
 1000 & $AC \to VA$ & 0.146 & 0 & 0.001 & 0.006 \\
 1000 & $LZ \to BC$ & 0.168 & 0 & 0.001 & 0.004 \\
 1000 & $LZ \to VA$ & 0.132 & 0 & 0.001 & 0.006 \\
 1000 & $BC \to VA$ & 0.321 & 0 & 0.001 & 0.005 \\[1ex] 
 \hline
\end{tabular}
\caption{Empirical Standard Deviation \eqref{SD} for (1): No noise, (2): Gamma(1,2), (3): Gamma(2,2)}
\label{table:3}
\end{table}

\begin{table}[H]
\centering
\begin{tabular}{|c c c c c c|} 
 \hline
 Sample Size & Edge & Edge Weight & RMSE(1) & RMSE(2) & RMSE(3)  \\ [0.5ex] 
 \hline
 50 & $AC \to BC$ & 0.177 & 0.033 & 0.030 & 0.056 \\ 
 50 & $AC \to VA$ & 0.146 & 0.054 & 0.044 & 0.075 \\
 50 & $LZ \to BC$ & 0.168 & 0.033 & 0.029 & 0.054 \\
 50 & $LZ \to VA$ & 0.132 & 0.053 & 0.044 & 0.077 \\
 50 & $BC \to VA$ & 0.321 & 0.006 & 0.021 & 0.061 \\
 200 & $AC \to BC$ & 0.177 & 0.001 & 0.007 & 0.023 \\ 
 200 & $AC \to VA$ & 0.146 & 0.002 & 0.009 & 0.030 \\
 200 & $LZ \to BC$ & 0.168 & 0.002 & 0.007 & 0.023 \\
 200 & $LZ \to VA$ & 0.132 & 0.006 & 0.011 & 0.031 \\
 200 & $BC \to VA$ & 0.321 & 0 & 0.005 & 0.027 \\
 500 & $AC \to BC$ & 0.177 & 0 & 0.003 & 0.014 \\ 
 500 & $AC \to VA$ & 0.146 & 0 & 0.004 & 0.018 \\
 500 & $LZ \to BC$ & 0.168 & 0 & 0.003 & 0.014 \\
 500 & $LZ \to VA$ & 0.132 & 0.001 & 0.004 & 0.018 \\
 500 & $BC \to VA$ & 0.321 & 0 & 0.002 & 0.017 \\
 1000 & $AC \to BC$ & 0.177 & 0 & 0.001 & 0.090 \\ 
 1000 & $AC \to VA$ & 0.146 & 0 & 0.002 & 0.120 \\
 1000 & $LZ \to BC$ & 0.168 & 0 & 0.001 & 0.090 \\
 1000 & $LZ \to VA$ & 0.132 & 0 & 0.002 & 0.012 \\
 1000 & $BC \to VA$ & 0.321 & 0 & 0.001 & 0.011 \\[1ex] 
 \hline
\end{tabular}
\caption{Empirical RMSE \eqref{RMSE} for (1): No noise, (2): Gamma(1,2), (3): Gamma(2,2)}
\label{table:2}
\end{table}


\bibliographystyle{plain}
\bibliography{main}

\appendix
\section{Proofs of Section~3}\label{A}

\textsc{Proof of Theorem~\ref{solution} } 
Rewrite \eqref{1stdef_1} in matrix form by means of the tropical matrix multiplication \eqref{ch2:odot}  as
$$\bU = \big(\bU\odot \bC \vee \bZ\big)\odot \bE_d.$$
The associative law implies 
      \begin{align}\label{noisematrixform}
         \bU  & =  (\bU  \odot \bC \odot \bE_d) \lor \bZ \odot \bE_d \quad
         \Leftrightarrow \quad  \bU =  \bU  \odot \bar{\bC} \lor \bar{\bZ},
    \end{align}
    with $\bar{\bC}=\bC\odot \bE_d$, which is identical to \eqref{randomweights}, and $\bar{\bZ}=\bZ \odot \bE_d$. 
    The right-most equation in \eqref{noisematrixform} is of the same form as the non-noisy model in \eqref{eq:XtroprepX}, so that analogously to its solution given in \eqref{eq:dotrepr}, we get the solution
    \begin{align*}
           \bB^* = (\bI_d\vee \bar{\bC})^{\odot (d-1)}, \qquad \bU= \bar{\bZ} \odot \bB^* = \bZ \odot \bE_d \odot \bB^*,
    \end{align*}
    where $ \bB^*$ is the Kleene star matrix of $\bar \bC$.
    Therefore, defining $\bar{\bB}=\bE_d \odot \bB^*$ yields the result. 
   \hspace*{5cm} \halmos 
    
\medskip

\textsc{Proof of Corollary~\ref{cor:3rddef} } 
From \eqref{noisezdef} and the continuity of $Z$ and $\eps$ we have 
\begin{align} \label{k_real_i}
  U_i= \bigvee_{j \in \An(i)}\bar{b}_{ji}Z_j=\bar{b}_{ki}Z_k  
\end{align}
for some unique $k \in \An(i)$.
We want to show that this implies $U_i=\bar{b}_{ki}\tilde U_k$, i.e. $\tilde U_k=Z_k$. Applying first \eqref{1stdef_2}, then \eqref{noisezdef} and finally \eqref{randomb}, we obtain
$$\tilde U_k=\frac{ U_k}{\varepsilon_k}=\frac{\bigvee_{l \in \An(k)}\bar{b}_{lk}Z_l}{\varepsilon_k} \geq \frac{\bar{b}_{kk}Z_k}{\varepsilon_k}=Z_k.$$
Now assume that $\tilde U_k>Z_k$. Then there exists an $l \in \an(k) \subset \An(i)$ with $\bar b_{lk}Z_l> \varepsilon_k Z_k$. Note also that the maximum random path weight from $l$ to $i$ must be greater or equal than the maximum random path weight from $l$ to $i$ passing through node $k$.
These two facts lead to 
$$U_i=\bigvee_{j \in \An(i)}\bar{b}_{ji}Z_j \geq \bar{b}_{li}Z_l \geq \frac{\bar{b}_{lk}\bar{b}_{ki}}{\varepsilon_k}Z_l>\bar{b}_{ki}Z_k,$$
The above inequality, however,  contradicts \eqref{k_real_i}.
   \hspace*{5cm} \halmos 

\medskip

    \textsc{Proof of Lemma~\ref{lemma1} }
  (a) \, We first assume that $j=i$. Since $\mathcal D$ is a DAG, $\de(i) \cap \pa(i)=\emptyset$ and $\bar{b}_{ik}\bar{b}_{ki}=0$ for all $k \neq i$.
Therefore, the equality holds and the inequality is equivalent to $\bar{b}_{ii} \geq 0$ which obviously holds. \\ 
 Next, assume $j \neq i$ and $j \not \in \an(i)$. Then by \eqref{randomb}  $\bar{b}_{ji}=0$ and there is no path from $j$ to $i$. Therefore,  $\de(j) \cap \pa(i) = \emptyset$. Hence, the right-hand side of the inequality equals zero. Moreover, the equality holds as well, otherwise $\bar{b}_{jk}>0$ and $\bar{b}_{ki}>0$ for some $k \in V$ and therefore, by \eqref{randomb} there would be a path from $j$ to $k$ and from $k$ to $i$ which contradicts $j \not \in \an(i)$. \\
 {For $j \in \pa(i)$ with $\de(j) \cap \pa(i) = \emptyset$, the critical path must be the edge $j \to i$ since it is the only path from $j$ to $i$. Furthermore, the equality $\bar b_{ji}=\frac{\bar b_{jk}\bar b_{ki}}{\bar b_{kk}}$ holds for $k=i$ and $k=j$ while for all $k \not \in \{i,j\}$ it must hold that $\frac{\bar b_{jk}\bar b_{ki}}{\bar b_{kk}}=0$.
 Therefore, the equality holds. Moreover, the right-hand side of the inequality again equals zero and we have strict inequality. }\\
Now assume $j \in \an(i)$ and $\de(j) \cap \pa(i) \neq \emptyset$.  Then for every path  $p=[j=k_0 \rightarrow k_1 \rightarrow \ldots \rightarrow k_n=i]$ with $n \geq 2$ from $j$ to $i$ and every $k_m \in \{k_1,\ldots,k_{n-1}\}$, by \eqref{randompathweight}, 
               \begin{align}\label{subpaths}
  \bar{d}_{ji}(p) = & \varepsilon_j \prod_{l=0}^{n-1}c_{k_lk_{l+1}}\varepsilon_{k_{l+1}}\notag\\
                  = & \frac{\varepsilon_{j} \prod_{l=0}^{m-1}c_{k_lk_{l+1}}\varepsilon_{k_{l+1}}\cdot \varepsilon_{k_m} \prod_{l=m}^{n-1}c_{k_lk_{l+1}}  \varepsilon_{k_{l+1}}}{\varepsilon_{k_m}}\notag \\
                  =  & \frac{\bar{d}_{jk_{m}}(p_1)\bar{d}_{k_{m}i}(p_2)}{\bar{b}_{k_mk_m}},
              \end{align}
 with $p_1=[j=k_0 \rightarrow k_1 \rightarrow \ldots \rightarrow k_m]$ and $p_2=[k_m \rightarrow \ldots \rightarrow k_n=i]$, where in the last step we have used that 
              $\varepsilon_{k_m}=\bar{b}_{k_mk_m}$. 
     Therefore, for the random critical path $p$ with  $\bar{b}_{ji}=\bar{d}_{ji}(p)$ 
              it holds that every sub-path of this path is itself critical, otherwise we could find a path of larger random path weight by replacing the sub-path by a path of larger random weight. It follows that
              \begin{align*}
                  \bar{b}_{ji} \geq \bigvee\limits_{k \in \de(j) \cap \an(i)}\frac{\bar{b}_{jk}\bar{b}_{ki}}{\bar{b}_{kk}}
              \end{align*}
              with equality whenever the critical path $p$ from $j$ to $i$ contains a node $k \in \de(j) \cap \an(i)$. Since for $k=i$ or $k=j$ we have $\bar{b}_{ji}=\frac{\bar{b}_{jk}\bar{b}_{ki}}{\bar{b}_{kk}}$ and for $k \in V \setminus \big((\an(i)\cap \de(j))\cup \{j,i\}\big)$ we have $\frac{\bar{b}_{jk}\bar{b}_{ki}}{\bar{b}_{kk}}=0$, the equality holds as well.\\[2mm]
(b) \,  First assume that there is a path $p:=[j \to \ldots \to k \to \ldots \to i]$ with $\bar d_{ji}(p)=\bar b_{ji}$. Then by \eqref{subpaths} we have $\bar{b}_{ji}=\frac{\bar{d}_{jk}(p_1)\bar{d}_{ki}(p_2)}{\bar{b}_{kk}}$. Now every sub-path of a random critical path must be itself critical, as explained in the proof of part a). Hence, $\bar{b}_{jk}=\bar{d}_{jk}(p_1)$ and $\bar{b}_{ki}=\bar{d}_{ki}(p_2)$ and for this reason $\bar{b}_{ji}=\frac{\bar{b}_{jk}\bar{b}_{ki}}{\bar{b}_{kk}}$. 

In contrast, let $\bar d_{ji}(p) < \bar b_{ji}$ for all $p \in P_{jki}$, where $P_{jki}$ denotes all paths from $j$ to $i$ that pass through $k$. Now choose $p_1=[j \to \ldots \to k]$ and $p_2=[k \to \ldots \to i]$ such that $\bar d_{jk}(p_1)=\bar b_{jk}$ and $\bar d_{ki}(p_2)=\bar b_{ki}$. Then, for the path $p \in P_{jki}$ that results from concatenation of $p_1$ and $p_2$ we have by \eqref{randomb}
$$\bar b_{ji}> \bar d_{ji}(p)=\frac{\bar{b}_{jk}\bar{b}_{ki}}{\bar{b}_{kk}},$$
which proves the reverse direction. \\[2mm]
(c) \, For $j=i$ the inequality obviously holds, since $b_{ii}=1$. 
        If $j \not \in \An(i)$, 
        then by definition $\bar{b}_{ji}=b_{ji}=0$ and $\bar{b}_{jj}=\varepsilon_{j} \geq 1$. Therefore, 
        the inequality is equivalent to $U_i/U_j \geq 0$, which is true. 
        Now let $j \in \an(i)$. Then by \eqref{randompathweight} and \eqref{randomb} the center ratio can be written as 
        \begin{align}\label{eq:A3}
                  \frac{\bar{b}_{ji}}{\bar{b}_{jj}}:=\bigvee\limits_{p \in P_{ji}} \frac{\bar{d}_{ji}(p)}{\bar{b}_{jj}}= \bigvee\limits_{p \in P_{ji}} d_{ji}(p)\prod_{l=0}^{n-1}\varepsilon_{k_{l+1}} \geq  \bigvee\limits_{p \in P_{ji}} d_{ji}(p) = b_{ji},
              \end{align}
              since $\bar{b}_{jj}=\varepsilon_j$ and $\varepsilon_i \geq 1$.
        Now we use \eqref{noise3rddef} and obtain by \eqref{eq:A3}
              \begin{align*}
                  \frac{U_i}{U_j}= \frac{\bigvee_{k \in \An(i)}\bar{b}_{ki}\tilde U_k}{U_j} \geq \frac{\bar{b}_{ji}\tilde U_j}{U_j} =\frac{\bar{b}_{ji} U_j}{\varepsilon_j U_j}=\frac{\bar{b}_{ji} }{\bar{b}_{jj} } \geq b_{ji}.
              \end{align*}
(d) \,  We first prove by contradiction that there is no lower bound for $U_i/U_j$ of larger value than the one given in part (c). 
Assume $j \in \an(i)$ and there is a lower bound $c>b_{ji}$.  Since $Z_1,\ldots,Z_d$ are iid, every innovation $Z_l$ can realize the maximum with positive probability, such that for every $l \in \An(j)$,
              \begin{align}\label{subset:omega}
                  \mathbb P \Big(\Big\{U_j = \bigvee_{k \in \An(j)}\bar{b}_{kj}Z_k=\bar{b}_{lj}Z_l\Big\} \cap \Big\{U_i =\bigvee_{k \in \An(i)}\bar{b}_{ki}Z_k=\bar{b}_{li}Z_l\Big\}\Big)>0.
              \end{align}
              Hence, without loss of generality we assume that this holds for $l = j$. 
              Denote the random critical path $p:=[j= k_0 \rightarrow \ldots \rightarrow  k_n=i]$ such that  $\bar{d}_{ji}(p)=\bar{b}_{ji}$. 
              Then, it follows on the event in \eqref{subset:omega} with $l=j$ from \eqref{randompathweight} 
              that
              \begin{align*}
                  \mathbb P \left(\frac{U_i}{U_j}<c\right) & =  \mathbb P \left(\frac{\bar{b}_{ji}}{\bar{b}_{jj}}<c\right)
                   = \mathbb P \left(\frac{\varepsilon_j d_{ji}(p) \prod_{l=0}^{n-1}\varepsilon_{ k_{l+1}}}{\varepsilon_j }<c\right) 
                 =\mathbb P \left( d_{ji}(p) \prod_{l=0}^{n-1}\varepsilon_{k_{l+1}}<c\right)>0,
              \end{align*}
              since  $d_{ji}(p)\leq b_{ji} < c$ and $\eps \geq 1$. Hence, $c$ is no lower bound and together with part c) this entails the support for $j\in\an(i)$. \\
             Now assume $j \not \in \An(i)$ such that $b_{ji}=0$. 
             Assume that $U_i/U_j$ is lower bounded by some $c>0$. Then by \eqref{noisezdef},
              \begin{align*}
                  \frac{U_i}{U_j} \geq c \quad\Leftrightarrow\quad \bigvee_{k \in \An(i)}\bar{b}_{ki}Z_k \geq c \bigvee_{k \in \An(j)}\bar{b}_{kj}Z_k,
              \end{align*}
              which is equivalent to 
              \begin{align*}
                  \bigvee_{k \in \An(i)}\bar{b}_{ki}Z_k   \geq c \, \Big(\bigvee_{k \in \An(i)\cap \An(j)}\bar{b}_{kj}Z_k \lor \bigvee_{k \in  \An(j)\setminus \An(i)}\bar{b}_{kj}Z_k\Big).
              \end{align*}
              Therefore, it holds in particular, that 
              \begin{align}\label{contradict}
                  \bigvee_{k \in \An(i)}\bar{b}_{ki}Z_k   \geq c \, \bar{b}_{lj}Z_l
              \end{align}
              for every $l \in \An(j)\setminus \An(i)$. This set is non-empty since $j \not \in \An(i)$, so it contains at least $j$. 
              However, since the innovation and the noise variables are all independent and unbounded above, we have for every $l \in \An(j)\setminus \An(i)$ 
              \begin{align*}
                   \mathbb P \Big(Z_l\geq \frac{\bigvee_{k \in \An(i)}\bar{b}_{ki}Z_k}{c \ \bar{b}_{lj}}\Big)>0,
              \end{align*}
              contradicting \eqref{contradict} and, hence, the assumption of a lower positive bound $c$ for ${U}_i/{U}_j$. 
              
              The upper interval limits of $U_i/U_j$ for $j \in \an(i)$ and and $j \not \in \An(i)$ follow from changing the roles of $i$ and $j$. For $j \neq i$, the ratio $U_i/U_j$ always contains $\varepsilon_i$ or $\varepsilon_j$ and both random variables are atom-free and independent of all innovations $Z_1,\ldots,Z_d$ and $\varepsilon_k$ for $k \neq i$ and $k \neq j$. Therefore, the ratio inherits the continuity of the noise variables and part d) follows. \\[2mm] 
(e) \, For $j=i$ we have $b_{ji}=1 \neq 0=\bigvee_{k \in \de(j) \cap \an(i)}\frac{b_{jk}b_{ki}}{b_{kk}}$. If $j \not \in \An(i)$ we have $b_{ji}=\bar b_{ji}=0$ by \eqref{randomb}.
 
Next assume that $j \in \an(i)$, and $b_{ji} = \bigvee_{k \in \de(j) \cap \an(i)}\frac{b_{jk}b_{ki}}{b_{kk}}\neq 0$. Then there is a path 
$p=[j=k_0 \rightarrow k_1 \rightarrow \ldots \rightarrow k_n=i]$ from $j$ to $i$ with non-random path weight $d_{ji}(p)=b_{ji}$, which is not the edge $j\to i$. 

For a contradiction, assume that $\bar{b}_{ji} > \bigvee_{k \in \de(j) \cap \an(i)}\frac{\bar{b}_{jk}\bar{b}_{ki}}{\bar{b}_{kk}}$. This is equivalent to the edge $j \to i$ being the random critical path. 
However, every path $p \in P_{ji}$ has random path weight, which depends on
 
              both noise variables $\varepsilon_i$ and $\varepsilon_j$, so in particular, the non-random critical path $p=[j=k_0 \rightarrow k_1 \rightarrow \ldots \rightarrow k_n=i]$ from $j$ to $i$ with path weight $d_{ji}(p)=b_{ji}$ is one of these paths. Therefore, by \eqref{randompathweight}  and since $b_{ji}>c_{ji}$, the random path weight of $p$ is
            \begin{align*}
                  \bar{d}_{ji}(p)=b_{ji} \varepsilon_j \prod_{l=0}^{n-1}\varepsilon_{ k_{l+1}}\ge  b_{ji}  \varepsilon_j \varepsilon_i > c_{ji}  \varepsilon_j \varepsilon_i = \bar{b}_{ji},
            \end{align*}
            where we have used that $\varepsilon\ge 1$. This is a contradiction and hence $\bar{b}_{ji} = \bigvee_{k \in \de(j) \cap \an(i)}\frac{\bar{b}_{jk}\bar{b}_{ki}}{\bar{b}_{kk}}$. \\[2mm]
(f) \,  The assumptions $b_{ji} > \bigvee_{k \in \de(j) \cap \an(i)}\frac{b_{jk}b_{ki}}{b_{kk}}$ and $\de(j) \cap \an(i) \neq\emptyset$
are equivalent to the edge  $p_{\max}= [j\to i]$ being the only non-random critical path.

Let $p'=[j=k_0 \rightarrow k_1 \rightarrow \ldots \rightarrow k_n=i]\neq p_{\max}$ be the path such that $\bigvee_{p \in P_{ji}\setminus \{p_{\max}\}} \bar{d}_{ji}(p)=\bar{d}_{ji}(p')$. Then 
        \begin{align*}
            \bigvee\limits_{p \in P_{ji}\setminus \{p_{\max}\}} \bar{d}_{ji}(p)=\bar{d}_{ji}(p')
            =\bigvee\limits_{k \in \de(j) \cap \an(i)}\frac{\bar{b}_{jk}\bar{b}_{ki}}{\bar{b}_{kk}},
        \end{align*}
        otherwise we can construct a path of larger random path weight from $j$ to $i$ passing through $k$ as explained in the proof of part a).
        First assume that $\bar b_{ji}=\bar d_{ji}(p_{\max})$. Then, $\bar{b}_{ji} > \bigvee_{k \in \de(j) \cap \an(i)}\frac{\bar{b}_{jk}\bar{b}_{ki}}{\bar{b}_{kk}}$ and $\de(j) \cap \an(i) \neq\emptyset$ is by \eqref{randompathweight} and \eqref{randomb} equivalent to 
        \begin{align}\label{crit.edge:pos.prob}
            \bar{b}_{ji} > \bar{d}_{ji}(p') = \varepsilon_i \varepsilon_j d_{ji}(p') \prod_{l=0}^{n-2}\varepsilon_{ k_{l+1}}
            \quad \iff \quad
            \frac{b_{ji}}{d_{ji}(p')} >  \prod_{l=0}^{n-2}\varepsilon_{ k_{l+1}}.
        \end{align}
        Since  $\varepsilon\ge 1$, also $b_{ji}/d_{ji}(p')>1$. Hence, the event given by \eqref{crit.edge:pos.prob} has positive probability which is however, strictly smaller than one, since the noise variables do not have an upper bound. 
        Therefore, since $\bar{b}_{ji} \geq \bigvee_{k \in \de(j) \cap \an(i)}\frac{\bar{b}_{jk}\bar{b}_{ki}}{\bar{b}_{kk}}$, by part a), the complementary event 
        $$\Big\{\bar{b}_{ji} = \bigvee\limits_{k \in \de(j) \cap \an(i)}\frac{\bar{b}_{jk}\bar{b}_{ki}}{\bar{b}_{kk}}\Big\}$$ 
        is also having positive probability.

\halmos

\medskip

    \textsc{Proof of Lemma~\ref{lemma:crit_rand_path} }
(a) Suppose there is an edge $k_l \to k_{l+1}$ in $p$ such that $c_{k_lk_{l+1}}\not \in \mathcal D^B$. Then, $\de(j) \cap \an(i) \neq \emptyset$ and by Lemma~\ref{lemma1}~e) $\mathbb P (\bar b_{k_lk_{l+1}}=c_{k_lk_{l+1}} \varepsilon_{k_{l}}\varepsilon_{k_{l+1}})=0$, so we can replace the edge $k_l\to k_{l+1}$ by some other path to get a new path from $j$ to $i$ of larger random path weight than $p$. Hence, $p$ is not a possible critical path realization. The same argument can be used for the reverse.  \\[2mm] 
(b) \, First consider $\neg (S_{p_1} \cap S_{p_2} = \emptyset$ or for every $r \in S_{p_1} \cap S_{p_2}$ the sub-path of $p_1$ from $j$ to $r$  is a sub-path of $p_2$ or the  sub-path  of $p_2$ from $l$ to $r$ is a sub-path of $p_1$). 
Then there exists some node $r \in S_{p_1}\cap S_{p_2}$ such that
$p_1=[j\to\ldots\to s \to r\to \ldots \to i]$ and $p_2=[l\to\ldots\to t \to r\to \ldots \to m]$ with $s\neq t$. 
Denote by $p_{11}:=[j\to\ldots\to s]$ the sub-path of $p_1$ from $j$ to $s$. 
We want to show by contradiction that the event \eqref{jointprobability} has probability zero. Therefore, we consider the subset of $\Omega$ 
such that \eqref{jointprobability} holds and show that it is a null-set. 
Since on this subset, $p_{1}$ is the random critical path and passes through $s$, by Lemma~\ref{lemma1}~b) we have  $\bar{b}_{ji} = \frac{\bar{b}_{js}\bar{b}_{si}}{\bar{b}_{ss}}$ and $U_s=U_j \bar{b}_{js} /\eps_j= U_j d_{js}(p_{11}) \prod_{k \in S_{p_{11}}} \eps_k$.
With the same argument it also holds that $\bar{b}_{ji} = \frac{\bar{b}_{jr}\bar{b}_{ri}}{\bar{b}_{rr}}$  and $U_r=U_j \bar{b}_{jr} /\eps_j= U_j d_{js}(p_{11}) \prod_{k \in S_{p_{11}}} \eps_k c_{sr} \eps_r$. 
Hence, it must holds that $U_r=U_s  c_{sr} \eps_r$.
By the same arguments, we also must have $U_r=U_t c_{tr} \eps_r$, which together leads to
\begin{align*}
    U_s c_{sr}=U_t c_{tr}.
\end{align*}
This is by  \eqref{randomb} and \eqref{noisezdef} equivalent to
\begin{align*}
    c_{sr}\bigvee_{l \in \An(s)}\eps_l\bigvee\limits_{p \in P_{ls}} d_{ls}(p) \prod_{k \in  S_{p}} \varepsilon_{k}Z_l=c_{tr}\bigvee_{l \in \An(t)}\eps_l\bigvee\limits_{p \in P_{lt}} d_{lt}(p) \prod_{k \in  S_{p}} \varepsilon_{k}Z_l.
\end{align*}
Now since $\mathcal D$ is acyclic, there cannot be a path from $s$ to $t$ and from $t$ to $s$; so without loss of generality we can assume that there is no path from $t$ to $s$. 
However, the right-hand side of the equation always contains $\eps_t$ which is not part of the left-hand side. Since $Z_1,\ldots,Z_d$ as well as $\varepsilon_1,\ldots,\varepsilon_i$ are atom-free and independent random variables, this can only happen on a null-set.

Next consider the reverse, i.e. $S_{p_1} \cap S_{p_2} = \emptyset$ or for every $r \in S_{p_1} \cap S_{p_2}$ the sub-path of $p_1$ from $j$ to $r$  is a sub-path of $p_2$ or the  sub-path  of $p_2$ from $l$ to $r$ is a sub-path of $p_1$.

If $S_{p_1}\cap S_{p_2}=\emptyset$, then the probability of \eqref{jointprobability} is obviously positive.
Without loss of generality we now assume that for every $r \in S_{p_1} \cap S_{p_2}$ the sub-path of $p_2$ from $l$ to $r$  is a sub-path of $p_1$. We now define $r$ to be the last common node of the two paths $p_1$ and $p_2$.
Then, $p_1$ and $p_2$ induce the paths $p'=[j \to \ldots \to l \to \ldots r]$, $p''=[r \to \ldots \to i]$ and $p'''=[ r \to \ldots \to m]$.
Then 
\begin{align*}
    &    \Big\{U_i= U_j d_{ji}(p_1)\prod_{k \in S_{p_1}}\varepsilon_{k}, U_m=U_l d_{lm}(p_2)\prod_{k \in S_{p_2}}\varepsilon_{k} \Big\}\\
&=
         \Big\{U_r= U_j d_{jr}(p')\prod_{k \in S_{p'}}\varepsilon_{k}, U_i= U_r d_{ri}(p'')\prod_{k \in S_{p''}}\varepsilon_{k}, U_m=U_r d_{rm}(p''')\prod_{k \in S_{p''}}\varepsilon_{k}\Big\},
        \end{align*}
which has positive probability, since $S_{p'}\cap S_{p''}\cap S_{p'''}=\emptyset$.
\halmos 

\medskip

\textsc{Proof of Theorem~\ref{epscount} } 
By the law of total probability we have for $x\ge 1$,
\begin{align*}
    I(x) &:= \mathbb P\Big(\frac{U_i}{U_j} \leq  b_{ji}x\Big) = \mathbb P\Big(\frac{U_i}{U_j} \leq  b_{ji}x,  U_i =\tilde U_j \bar{b}_{ji}\Big) + \mathbb P\Big(\frac{U_i}{U_j} \leq  b_{ji}x , U_i \neq \tilde U_j \bar{b}_{ji}\Big)\\
    &=:I_1(x)  + I_2(x) 
\end{align*}
We denote all paths from $j$ to $i$ by $P_{ji}=\{p_1, \ldots , p_r, p_{\max}\}$. There are two situations, either $r=0$ (where we interpret the above set of paths as $\{p_{\max}\}$), or $r\ge 1$. We first give a proof for $r\ge 1$.
We start with $I_1(x)$.
Since $p_{\max}$ is generic, every path $p \neq p_{\max}$ from $j$ to $i$ has non-random edge weight $d_{ji}(p)<b_{ji}$. 
Therefore, with \eqref{1stdef_2} in the first line, \eqref{randomb} in the third
 and \eqref{randompathweight} in the last,
 we have for $x>1$,
    \begin{align}
       I_1(x) &= \mathbb P(U_i/(\tilde U_j\eps_j) \leq  b_{ji}x , U_i =\tilde U_j \bar{b}_{ji}) \notag \\ 
        & = \mathbb P(\bar{b}_{ji}/\varepsilon_j \leq  b_{ji}x, U_i =\tilde U_j \bar{b}_{ji} )\notag \\
       & =\mathbb P\Big(\bigvee\limits_{p \in P_{ji}} \bar{d}_{ji}(p) /\varepsilon_j \leq  b_{ji}x , U_i =\tilde U_j \bar{b}_{ji} \Big) \notag \\ 
         & = \mathbb P\Big(\bigvee\limits_{p \in P_{ji}} d_{ji}(p)\prod_{k\in S_p}\varepsilon_{k}  \leq  b_{ji}x, U_i =\tilde U_j \bar{b}_{ji} \Big). \label{I(x)}
    \end{align}
By definition of $\bar{b}_{ji}$ in \eqref{randomb}, there is a path $p \in \{p_{\max}, p_1,\ldots,p_r\}$ such that  $\bar{d}_{ji}(p)=\bar{b}_{ji}$ and by continuity of $\eps$ the probability that multiple paths satisfy the equation is equal to 0. Therefore, again applying the law of total probability, we find
    \begin{align}\label{i1}
       I_1(x) =& \ \mathbb P\Big(\bigvee\limits_{p \in P_{ji}} d_{ji}(p)\prod_{k \in S_p}\varepsilon_{k}  \leq  b_{ji}x , \bar{d}_{ji}(p_{\max})=\bar{b}_{ji} , U_i =\tilde U_j \bar{b}_{ji} \Big) \\
       +& \mathbb P\Big(\bigvee\limits_{p \in P_{ji}} d_{ji}(p)\prod_{k \in S_p}\varepsilon_{k}  \leq  b_{ji}x , \bigvee\limits_{p \in \{p_1,\ldots,p_r\}} \bar{d}_{ji}(p)=\bar{b}_{ji}, U_i =\tilde U_j \bar{b}_{ji} \Big) \notag \\
        =&\ \mathbb P\Big(\bigvee\limits_{p \in P_{ji}} d_{ji}(p)\prod_{k \in S_p}\varepsilon_{k}  \leq  b_{ji}x , \bar{d}_{ji}(p_{\max})=\bar{b}_{ji} , U_i =\tilde U_j \bar{b}_{ji} \Big) \notag \\
       + &\sum\limits_{s=1}^r\mathbb P\Big(\bigvee\limits_{p \in P_{ji}} d_{ji}(p)\prod_{k \in S_p}\varepsilon_{k}  \leq  b_{ji}x , \bar{d}_{ji}(p_s)=\bar{b}_{ji}, U_i =\tilde U_j \bar{b}_{ji} \Big)=:I_{11}(x)+I_{12}(x).\notag
    \end{align}
We first find upper and lower bounds for $I_{11}(x)$. We denote by $P_{kji}$ all paths from $k$ to $i$ which pass through $j$.
Using the simple identity 
\begin{align}\label{simple}
\{z_1\vee z_2\le a, z_1\vee z_2=z_1\}=\{z_1\le a, z_2\leq z_1\},
\end{align}
\eqref{randomb} and \eqref{noisezdef} imply 
\begin{align}\label{I11_ext} 
   &\Big\{\bigvee\limits_{p \in P_{ji}} d_{ji}(p)\prod_{k \in S_p}\varepsilon_{k}  \leq  b_{ji}x \Big\} \bigcap \Big\{ \bar{d}_{ji}(p_{\max})=\bar{b}_{ji} \Big\} \bigcap \Big\{ U_i =\tilde U_j \bar{b}_{ji} \Big\} \\
    =  &\Big\{\bigvee\limits_{p \in P_{ji}} d_{ji}(p)\prod_{k \in S_p}\varepsilon_{k}  \leq \bar{d}_{ji}(p_{\max})  \Big\} \bigcap \Big\{ \bar{d}_{ji}(p_{\max})\leq b_{ji}x \Big\} \bigcap \Big\{ U_i =\tilde U_j \bar{b}_{ji} \Big\} \notag \\
    =  &\Big\{\bigvee\limits_{p \in P_{ji}} d_{ji}(p) \prod_{k \in S_p}\varepsilon_k  \leq {b_{ji}}\prod_{k \in  S_{p_{\max}}} \varepsilon_{k}  \Big\} \bigcap \Big\{ \prod_{k \in  S_{p_{\max}}} \varepsilon_{k} \leq x \Big\}
    \bigcap \Big\{ U_i =\tilde U_j \bar{b}_{ji} \Big\} \notag \\
    =  &\bigcap\limits_{p \in P_{ji}\setminus \{p_{\max}\}} \Big\{ \prod_{k \in S_p}\varepsilon_k  \leq \frac{b_{ji}}{d_{ji}(p)}\prod_{k \in  S_{p_{\max}}} \varepsilon_{k}  \Big\} \bigcap \Big\{ \prod_{k \in  S_{p_{\max}}} \varepsilon_{k} \leq x \Big\} \bigcap \notag \\
    &  \bigcap_{l \in \An(i)}\Big\{\bigcap\limits_{p \in P_{li}\setminus P_{lji}}\Big\{ d_{li}(p) \prod_{k \in  S_{p} \cup \{l\}} \varepsilon_{k}Z_l  \leq  b_{ji} \prod_{k \in  S_{p_{\max}}} \varepsilon_{k} \bigvee_{l \in \An(j)} \bar{b}_{lj} Z_l  \Big\}\Big\} \label{I11_ext2}.
\end{align}
Cancelling all noise variables possible, and since $\eps>1$, we find a lower bound
    \begin{align} \label{I11}
       I_{11}(x) 
       =& \, \mathbb P\Big(\bigcap\limits_{p \in P_{ji}\setminus \{ p_{\max} \}} \Big\{ \prod_{k \in S_p\setminus S_{p_{\max}}}\varepsilon_{k}  \leq  \frac{b_{ji} }{d_{ji}(p)} \prod_{k \in S_{p_{\max}}\setminus S_{p}}\varepsilon_{k} \Big\} \bigcap \Big\{\prod_{k \in S_{p_{\max}}}\varepsilon_{k} \leq x \Big\} \bigcap  \notag  \\
       &  \bigcap_{l \in \An(i)}\Big\{\bigcap\limits_{p \in P_{li}\setminus P_{lji}}\Big\{ d_{li}(p) \prod_{k \in  (S_{p}\cup\{l\})\setminus S_{p_{\max}}} \varepsilon_{k}Z_l  \leq  b_{ji} \prod_{k \in  S_{p_{\max}}\setminus (S_{p}\cup\{l\})} \varepsilon_{k} \bigvee_{l \in \An(j)} \bar{b}_{lj} Z_l  \Big\}\Big\} \Big)  \\
        \geq & \, \mathbb P\Big(\bigcap\limits_{p \in P_{ji}\setminus \{ p_{\max} \}} \Big\{ \prod_{k \in S_p\setminus S_{p_{\max}}}\varepsilon_{k}  \leq  \frac{b_{ji} }{d_{ji}(p)} \Big\} \bigcap \Big\{ \prod_{k \in S_{p_{\max}}}\varepsilon_{k} \leq x \Big\} \bigcap \notag  \\
       &  \bigcap_{l \in \An(i)}\Big\{\bigcap\limits_{p \in P_{li}\setminus P_{lji}}\Big\{ d_{li}(p) \prod_{k \in  (S_{p}\cup\{l\})\setminus S_{p_{\max}}} \varepsilon_{k}Z_l  \leq  b_{ji}  \bigvee_{l \in \An(j)} \bar{b}_{lj} Z_l  \Big\}\Big\} \Big) \notag \\
       =& \, \mathbb P\Big( \bigcap_{l \in \An(i)}\Big\{\bigcap\limits_{p \in P_{li}\setminus P_{lji}}\Big\{ d_{li}(p) \prod_{k \in  (S_{p}\cup\{l\})\setminus S_{p_{\max}}} \varepsilon_{k}Z_l  \leq  b_{ji}  \bigvee_{l \in \An(j)} \bar{b}_{lj} Z_l  \Big\}\Big\} \bigcap  \notag \\
       &  \bigcap\limits_{p \in P_{ji}\setminus \{ p_{\max} \}} \Big\{ \prod_{k \in S_p\setminus S_{p_{\max}}}\varepsilon_{k}  \leq  \frac{b_{ji} }{d_{ji}(p)} \Big\} \Big)  \, \mathbb P \Big( \prod_{k \in S_{p_{\max}}}\varepsilon_{k} \leq x\Big) \notag \\
       =: & \, c_1 \ \mathbb P\Big( \prod_{k \in S_{p_{\max}}}\varepsilon_{k} \leq x \Big) \notag,
    \end{align}
for some constant $c_1 \in [0,1]$ by independence of the noise variables. 

We show that $c_1>0$. 
To do so, recall 
that $b_{ji}/d_{ji}(p)>1$ for every $p \neq p_{\max}$. 
Therefore, since $\{p \in P_{ji}\setminus \{ p_{\max} \}\}\neq \emptyset$ and $\eps>1$,
\begin{align}\label{const:1}
    \mathbb P\Big(\bigcap\limits_{p \in P_{ji}\setminus \{ p_{\max} \}} \Big\{ \prod_{k \in S_p\setminus S_{p_{\max}}}\varepsilon_{k}  \leq  \frac{b_{ji} }{d_{ji}(p)} \Big\} \Big)>0.
\end{align}
Next, we want to show that also 
\begin{align} \label{const:2}
    \mathbb P\Big( \bigcap_{l \in \An(i)}\Big\{\bigcap\limits_{p \in P_{li}\setminus P_{lji}}\Big\{ d_{li}(p) \prod_{k \in  (S_{p}\cup\{l\})\setminus S_{p_{\max}}} \varepsilon_{k}Z_l  \leq  b_{ji}  \bigvee_{l \in \An(j)} \bar{b}_{lj} Z_l  \Big\}\Big\}  \Big)>0.
\end{align}
For this, observe that the left-hand side of the inequality in \eqref{const:2} does not contain $Z_j$, since all paths from $j$ to $i$ pass through $j$.
Since $\bar{b}_{lj}$ and the left-hand side of the inequality in  \eqref{const:2} is independent of $Z_j$ for all $l \in \{1,\ldots,d\}$ and $Z_j$ has unbounded support, $Z_j$ can become arbitrarily large with positive probability such that \eqref{const:2} holds. 

The intersection of the two events has also positive probability since \eqref{const:1} is independent of $Z_j$.
This implies that $c_1>0$ and a positive lower bound for $I_{11}(x)$. 

To get an upper bound, observe that $\varepsilon\ge 1$ and, hence, for every set $S_p$  we have 
\begin{align*}
   \Big\{\varepsilon_k: k \in S_{p_{\max}} \text{ and } \prod_{k \in S_{p_{\max}}}\varepsilon_{k} \leq x \Big\} \subseteq  \Big\{\varepsilon_k: k \in S_{p_{\max}} \text{ and } \prod_{k \in S_{p_{\max}}\setminus S_p}\varepsilon_{k} \leq x \Big\}.
\end{align*}
Therefore, starting with $\eqref{I11}$ we find the upper bound
\begin{align}
     I_{11}(x) \leq & \, \mathbb P\Big(\bigcap\limits_{p \in P_{ji}\setminus \{ p_{\max} \}} \Big\{ \prod_{k \in S_p\setminus S_{p_{\max}}}\varepsilon_{k}  \leq  \frac{b_{ji} }{d_{ji}(p)} x \Big\} \bigcap \Big\{ \prod_{k \in S_{p_{\max}}}\varepsilon_{k} \leq x \Big\} \bigcap \label{upperboundI11}  \\
       &  \bigcap_{l \in \An(i)}\Big\{\bigcap\limits_{p \in P_{li}\setminus P_{lji}}\Big\{ d_{li}(p) \prod_{k \in  (S_{p}\cup\{l\})\setminus S_{p_{\max}}} \varepsilon_{k}Z_l  \leq  b_{ji} x  \bigvee_{l \in \An(j)} \bar{b}_{lj} Z_l  \Big\}\Big\} \Big) \notag  \\
       = & \, \mathbb P\Big( \bigcap_{l \in \An(i)}\Big\{\bigcap\limits_{p \in P_{li}\setminus P_{lji}}\Big\{ d_{li}(p) \prod_{k \in  (S_{p}\cup\{l\})\setminus S_{p_{\max}}} \varepsilon_{k}Z_l  \leq  b_{ji} x \bigvee_{l \in \An(j)} \bar{b}_{lj} Z_l  \Big\}\Big\} \bigcap \notag   \\
       &  \bigcap\limits_{p \in P_{ji}\setminus \{ p_{\max} \}} \Big\{ \prod_{k \in S_p\setminus S_{p_{\max}}}\varepsilon_{k}  \leq  \frac{b_{ji} }{d_{ji}(p)} x \Big\} \Big)  \, \mathbb P \Big( \prod_{k \in S_{p_{\max}}}\varepsilon_{k} \leq x\Big) \notag \\
       = & \, c_2(x) \, \mathbb P\Big( \prod_{k \in S_{p_{\max}}}\varepsilon_{k} \leq x \Big). \notag
\end{align}
Since the innovations and the noise variables are atom-free, it follows that $\lim_{x \downarrow 1} c_2(x)= c_1$ and, therefore,
\begin{align}\label{III_asymp}
       I_{11}(x) \sim c_1 \ \mathbb P\Big( \prod_{k \in S_{p_{\max}}}\varepsilon_{k} \leq x \Big),\quad x\downarrow 1.
\end{align}
We next show that $I_{12}(x)=o(I_{11}(x))$ as $x\downarrow 1$. 
We have for each summand $m \in \{1,\ldots,r\}$, using the simple identity \eqref{simple} to obtain the third line,
\begin{align} \label{I12}
       \mathbb P &\Big(\bigvee\limits_{p \in P_{ji}} d_{ji}(p)\prod_{k \in S_p}\varepsilon_{k}  \leq  b_{ji}x,  \bar{d}_{ji}(p_m)=\bar{b}_{ji}, U_i =\tilde U_j \bar{b}_{ji} \Big) \notag  \\
        \leq \mathbb P &\Big(\bigvee\limits_{p \in P_{ji}} d_{ji}(p)\prod_{k \in S_p}\varepsilon_{k}  \leq  b_{ji}x,  \bar{d}_{ji}(p_m)=\bar{b}_{ji} \Big) \notag \\
        = \mathbb P &\Big(\bigvee\limits_{p \in P_{ji}} d_{ji}(p)\prod_{k \in S_p}\varepsilon_{k}  \leq  \bar{d}_{ji}(p_m),  \bar{d}_{ji}(p_m)\leq b_{ji}x \Big) \notag \\
       = \mathbb P & \Big(\bigcap\limits_{p \in P_{ji}\setminus \{p_m\}} \Big\{ \prod_{k \in S_p}\varepsilon_{k}  \leq  \frac{d_{ji}(p_m) }{d_{ji}(p)} \prod_{k \in S_{p_m}}\varepsilon_{k} \Big\} \bigcap \Big\{ \prod_{k \in S_{p_m}}\varepsilon_{k} \leq \frac{b_{ji}x} {d_{ji}(p_m)} \Big\} \Big) \notag \\
        \leq \mathbb P &\Big( \Big\{ \prod_{k \in S_{p_{\max}}}\varepsilon_{k} \leq \frac{d_{ji}(p_m) }{b_{ji}} \prod_{k \in S_{p_m}}\varepsilon_{k} \Big\} \bigcap \Big\{ \prod_{k \in S_{p_m}}\varepsilon_{k} \leq \frac{b_{ji}x} {d_{ji}(p_m)} \Big\} \Big) 
        \end{align}
        Now  the first event rewrites as
$\{\frac{b_{ji}}{d_{ji}(p_m) } \prod_{k \in S_{p_{\max}}}\varepsilon_{k}  \leq  \prod_{k \in S_{p_m}}\varepsilon_{k}\} \subseteq
    \{\frac{b_{ji}}{d_{ji}(p_m) }  \leq  \prod_{k \in S_{p_m}}\varepsilon_{k} \}$, since $\eps>1$.
    Moreover, 
    $$\Big\{ \prod_{k \in S_{p_{\max}}}\varepsilon_{k} \leq \frac{d_{ji}(p_m) }{b_{ji}} \prod_{k \in S_{p_m}}\varepsilon_{k}\Big\} \bigcap \Big\{ \prod_{k \in S_{p_m}}\varepsilon_{k} \leq \frac{b_{ji}x} {d_{ji}(p_m)}\Big\}\subseteq \Big\{ \prod_{k \in S_{p_{\max}}}\varepsilon_{k} \leq x\Big\}.$$ 
    Hence, 
        \begin{align*}
        \eqref{I12}\leq \mathbb P &\Big( \Big\{ \prod_{k \in S_{p_{\max}}}\varepsilon_{k} \leq x \Big\} \bigcap \Big\{ \prod_{k \in S_{p_m}}\varepsilon_{k} \in \left[\frac{b_{ji}} {d_{ji}(p_m)}, \frac{b_{ji}x} {d_{ji}(p_m)}\right] \Big\} \Big),
    \end{align*}
    Moreover, since $\varepsilon\ge 1$, we have for every subset $S \subseteq S_{p_{\max}}$ that $1 \leq \prod_{k \in S}\varepsilon_{k} \leq x $, whenever $1\leq \prod_{k \in S_{p_{\max}}}\varepsilon_{k} \leq x$.
Therefore, for another node set $\tilde{S}$ with $S \cap \tilde S = \emptyset$ we have
\begin{align}\label{interval:est}
    \prod_{k \in S}\varepsilon_{k}\prod_{k \in \tilde S}\varepsilon_{k} \in [a,b] \quad\Rightarrow\quad \prod_{k \in \tilde S}\varepsilon_{k} \in [a/x,b].
\end{align}
Finally, since $\bar{d}_{ji}(p_m)=\bar{b}_{ji}$ and $d_{ji}(p_m)<d_{ji}(p_{\max})$ we have $S_{p_m}\setminus S_{p_{\max}} \neq \emptyset$. In total, we obtain
\begin{align*}
  \eqref{I12} & \leq \mathbb P\Big( \Big\{ \prod_{k \in S_{p_{\max}}}\varepsilon_{k} \leq x \Big\} \bigcap \Big\{ \prod_{k \in S_{p_m}}\varepsilon_{k} \in \left[\frac{b_{ji}} {d_{ji}(p_m)}, \frac{b_{ji}x} {d_{ji}(p_m)}\right] \Big\} \Big) \\
       & \leq  \mathbb P\Big(\prod_{k \in S_{p_{\max}}}\varepsilon_{k} \leq x \Big) \mathbb P\Big( \prod_{k \in S_{p_m}\setminus S_{p_{\max}}}\varepsilon_{k} \in \left[\frac{b_{ji}} { x d_{ji}(p_m)}, \frac{b_{ji}x} {d_{ji}(p_m)}\right] \Big)\\
       & = \mathbb P\Big(\prod_{k \in S_{p_{\max}}}\varepsilon_{k} \leq x \Big)  o(1), \quad x \downarrow 1,
    \end{align*}
    as the interval in the second probability gets arbitrarily small and the distribution of $\eps$ is atom-free.
Comparing this upper bound with \eqref{III_asymp} we can see that every summand of $I_{12}(x)$
is negligible with respect to $I_{11}(x)$ as $x \downarrow 1$. 
Since there are only finitely many nodes and hence finitely many paths from $j$ to $i$, we have proved that $I_{12}(x)=o(I_{11}(x))$ as $x\downarrow 1$. 
Hence,
\begin{align}\label{III_asymp2}
       I_1(x) \sim c_1 \ \mathbb P\Big( \prod_{k \in S_{p_{\max}}}\varepsilon_{k} \leq x \Big),\quad x\downarrow 1.
\end{align}
Next, we assume that $r=0$, i.e. that there is only one path $p_{\max}$ from $j$ to $i$. 
Then from \eqref{i1} we find that $I_1(x)=I_{11}(x)$ and simplifies \eqref{I11} to 
\begin{align*}
        I_1(x) &=\mathbb P\Big( \Big\{\prod_{k \in S_{p_{\max}}}\varepsilon_{k} \leq x \Big\} \bigcap   \bigcap_{l \in \An(i)}\Big\{\bigcap\limits_{p \in P_{li}\setminus P_{lji}}\Big\{ d_{li}(p) \prod_{k \in  S_{p}\cup\{l\}} \varepsilon_{k}Z_l  \leq  b_{ji} \prod_{k \in  S_{p_{\max}}} \varepsilon_{k} \bigvee_{l \in \An(j)} \bar{b}_{lj} Z_l  \Big\}\Big\} \Big)  \\ 
        &\geq \mathbb P\Big( \Big\{ \prod_{k \in S_{p_{\max}}}\varepsilon_{k} \leq x \Big\} \bigcap  \bigcap_{l \in \An(i)}\Big\{\bigcap\limits_{p \in P_{li}\setminus P_{lji}}\Big\{ d_{li}(p) \prod_{k \in  (S_{p}\cup\{l\})\setminus S_{p_{\max}}} \varepsilon_{k}Z_l  \leq  b_{ji}  \bigvee_{l \in \An(j)} \bar{b}_{lj} Z_l  \Big\}\Big\} \Big)  \\ 
       &= \mathbb P\Big(\bigcap_{l \in \An(i)}\Big\{\bigcap\limits_{p \in P_{li}\setminus P_{lji}}\Big\{ d_{li}(p) \prod_{k \in  (S_{p}\cup\{l\})\setminus S_{p_{\max}}} \varepsilon_{k}Z_l  \leq  b_{ji}  \bigvee_{l \in \An(j)} \bar{b}_{lj} Z_l  \Big\}\Big\} \Big) \, \mathbb P \Big( \prod_{k \in S_{p_{\max}}}\varepsilon_{k} \leq x\Big)  \\
       &= c_1 \ \mathbb P\Big( \prod_{k \in S_{p_{\max}}}\varepsilon_{k} \leq x \Big) 
\end{align*}
for $c_1>0$. On the other hand,  
\begin{align*}
        I_1(x) &\leq \mathbb P\Big(\bigcap_{l \in \An(i)}\Big\{\bigcap\limits_{p \in P_{li}\setminus P_{lji}}\Big\{ d_{li}(p) \prod_{k \in  (S_{p}\cup\{l\})\setminus S_{p_{\max}}} \varepsilon_{k}Z_l  \leq  b_{ji} x  \bigvee_{l \in \An(j)} \bar{b}_{lj} Z_l  \Big\}\Big\} \Big) \, \mathbb P \Big( \prod_{k \in S_{p_{\max}}}\varepsilon_{k} \leq x\Big)  \\
       &= c_2(x) \ \mathbb P\Big( \prod_{k \in S_{p_{\max}}}\varepsilon_{k} \leq x \Big),
\end{align*}
and, since $Z$ and $\eps$ are atom-free it again follows that $\lim_{x \downarrow 1} c_2(x)= c_1$ and therefore \eqref{III_asymp2} holds also for $r=0$. \\

We next show that $I_2(x)=o(I_1(x))$ as $x\downarrow 1$. 
Since $I_{12}(x)=o(I_{11}(x)$ as $x\downarrow 1$, we can and do assume that 
\begin{align}\label{pmaxcrit}
    \bar{b}_{ji}=b_{ji}\varepsilon_j \prod\limits_{k \in S_{p_{\max}}}\varepsilon_k.
\end{align}
Moreover, since for all paths $p\in P_{lji}$ we have $l\in\An(i)$ if and only if $l\in\An(j)$,
\begin{align*}
    U_i &=\bigvee_{l \in \An(i)}\bigvee\limits_{p \in P_{li}\setminus P_{lji}} d_{li}(p) \prod_{k \in  S_{p}\cup\{l\}} \varepsilon_{k}Z_l \vee \bigvee_{l \in \An(j)}\bigvee\limits_{p \in P_{lji}} d_{li}(p) \prod_{k \in  S_{p}\cup\{l\}} \varepsilon_{k}Z_l\\
    &=\bigvee_{l \in \An(i)}\bigvee\limits_{p \in P_{li}\setminus P_{lji}} d_{li}(p) \prod_{k \in  S_{p}\cup\{l\}} \varepsilon_{k}Z_l \vee b_{ji}  \prod_{k \in  S_{p_{\max}}} \varepsilon_{k} U_j
\end{align*}
by \eqref{pmaxcrit}.
Therefore, if
    $\bigvee_{l \in \An(i)}\bigvee\limits_{p \in P_{li}\setminus P_{lji}} d_{li}(p) \prod_{k \in  S_{p}\cup\{l\}} \varepsilon_{k}Z_l > b_{ji}  \prod_{k \in  S_{p_{\max}}} \varepsilon_{k} U_j,$
then
\begin{align}\label{Uirep}
    U_i=\bigvee_{l \in \An(i)}\bigvee\limits_{p \in P_{li}\setminus P_{lji}} d_{li}(p) \prod_{k \in  S_{p}\cup\{l\}} \varepsilon_{k}Z_l.
\end{align}
Moreover, it holds 
by \eqref{noise3rddef}, \eqref{randomb} and \eqref{randompathweight}
\begin{align}\label{eq:Ueps}
  U_i= \bigvee_{k \in \An(i)}\bar{b}_{ki}\tilde U_k\geq \bar{b}_{ji}\tilde U_j\geq b_{ji} \prod_{k \in S_{p_{\max}}}\varepsilon_k U_j.
\end{align}
Hence,  $U_i/U_j\leq b_{ji}x$ implies that $\prod_{k \in S_{p_{\max}}}\eps_k \leq x$. Therefore, 
using \eqref{1stdef_2}, \eqref{Uirep} and \eqref{pmaxcrit} we get 
\begin{align*}
    I_2(x)&=\mathbb P\Big(\frac{U_i}{U_j} \leq  b_{ji}x , U_i > \tilde U_j \bar{b}_{ji}\Big)=\mathbb P\Big(U_i \in (\tilde U_j \bar{b}_{ji}, \tilde U_j \eps_j b_{ji}x ]\Big) \\
    &\leq \mathbb P\Big(  \Big\{ \prod_{k \in S_{p_{\max}}}\varepsilon_{k} \leq x \Big\} \bigcap \Big\{  \bigvee_{l \in \An(i)}\bigvee\limits_{p \in P_{li}\setminus P_{lji}} d_{li}(p) \prod_{k \in  S_{p}\cup\{l\}} \varepsilon_{k}Z_l \in \Big(b_{ji}  \tilde U_j \eps_j \prod_{k \in S_{p_{\max}}}\eps_k , b_{ji} \tilde U_j \eps_j x\Big] \Big\}  \Big) \\
    &\leq \mathbb P\Big(  \Big\{ \prod_{k \in S_{p_{\max}}}\varepsilon_{k} \leq x \Big\} \bigcap \Big\{  \bigvee_{l \in \An(i)}\bigvee\limits_{p \in P_{li}\setminus P_{lji}} d_{li}(p) \prod_{k \in  S_{p}\cup\{l\}} \varepsilon_{k}Z_l \in \left(b_{ji} \tilde U_j \eps_j  , b_{ji} \tilde U_j \eps_j x\right] \Big\}  \Big),
\end{align*}
since $\eps \geq 1$.
Using that $U_j=\tilde U_j \varepsilon_j$ and $j \not \in S_{p_{\max}}$ and the same argument as in \eqref{interval:est}, we get 
\begin{align*}
     &I_2(x) \leq \mathbb P\Big(  \Big\{ \prod_{k \in S_{p_{\max}}}\varepsilon_{k} \leq x \Big\} \bigcap \Big\{  \bigvee_{l \in \An(i)}\bigvee\limits_{p \in P_{li}\setminus P_{lji}} d_{li}(p) \prod_{k \in  (S_{p}\cup\{l\})\setminus S_{p_{\max}}} \varepsilon_{k}Z_l \in \left(\frac{b_{ji}\tilde  U_j \eps_j}{x}   , b_{ji}\tilde U_j \eps_j x\right] \Big\}  \Big) \\
     &= \mathbb P\Big(  \Big\{ \prod_{k \in S_{p_{\max}}}\varepsilon_{k} \leq x \Big\} \Big) \mathbb P \Big( \Big\{  \bigvee_{l \in \An(i)}\bigvee\limits_{p \in P_{li}\setminus P_{lji}} \frac{d_{li}(p)}{\varepsilon_j} \prod_{k \in  (S_{p}\cup\{l\})\setminus S_{p_{\max}}} \varepsilon_{k}Z_l \in \Big(\frac{b_{ji} \tilde  U_j}{x}   , b_{ji}\tilde U_j x\Big] \Big\}  \Big),
\end{align*}
since $\mathcal D$ being acyclic implies that $\tilde U_j$ and $\eps_j$ are independent of $\varepsilon_k$ for every $k \in S_{p_{\max}}$. For $x \downarrow 1$ the interval in the second probability gets arbitrarily small. Since the distribution of the noise-variables is atom-free and the left-hand side contains $\varepsilon_j$  that is not included in $\tilde U_j$, this probability tends to zero as $x \downarrow 1$. Comparing this upper bound with \eqref{III_asymp2} we can see that $I_2(x)=o(I_1(x))$ as $x\downarrow 1$. Since $I_{12}(x)=o(I_{11}(x))$, we have
\begin{align} \label{I_to_II11}
       I(x) \sim I_1(x) \sim I_{11}(x) \sim 
       c_1 \ \mathbb P\Big( \prod_{k \in S_{p_{\max}}}\varepsilon_{k} \leq x \Big),
       \quad x\downarrow 1,
\end{align}
holds, where the last asymptotic equivalence follows from \eqref{III_asymp2}. Moreover, we have  by \eqref{i1}, using \eqref{randompathweight},\eqref{randomb} and \eqref{1stdef_2},
\begin{align}
     I(x) \sim I_{11}(x)& = \mathbb P\Big(\bigvee\limits_{p \in P_{ji}} d_{ji}(p)\prod_{k \in S_p}\varepsilon_{k}  \leq  b_{ji}x , \bar{d}_{ji}(p_{\max})=\bar{b}_{ji} , U_i =\tilde U_j \bar{b}_{ji} \Big) \notag \\
    & = \mathbb P\Big(\bar{d}_{ji}(p_{\max})/\eps_j \leq  b_{ji}x , U_i =\tilde  U_j \bar{d}_{ji}(p_{\max}) \Big) \notag \\
    & = \mathbb P\Big(\prod_{k \in S_{p_{\max}}}\varepsilon_{k}  \leq  x, U_i = U_j b_{ji}\prod_{k \in S_{p_{\max}}}\varepsilon_{k} \Big) ,\quad x\downarrow 1, \label{I11:alt}
\end{align}
which, together with  \eqref{I_to_II11}, proves the result. 
\halmos

\medskip

\textsc{Proof of Corollary~\ref{epscount_prop}. }
    We first show the result for $P_{ji}\setminus\{p_1,\ldots,p_n\}\neq \emptyset$, i.e. there exists a path $p$ from $j$ to $i$ with $d_{ji}<b_{ji}$. We start as in the proof of Theorem~\ref{epscount} for $x \geq 1$
\begin{align*}
    I(x) = I_1(x)  + I_2(x) 
\end{align*}
and similarly to \eqref{I(x)}, we again apply the law of total probability to $I_1(x)$ 
    \begin{align*}
       I_1(x) =& \mathbb P\Big(\bigvee\limits_{p \in P_{ji}} d_{ji}(p)\prod_{k\in S_p}\varepsilon_{k}  \leq  b_{ji}x, U_i =\tilde U_j \bar{b}_{ji} \Big) \\
        =&\mathbb P\Big(\bigvee\limits_{p \in P_{ji}} d_{ji}(p)\prod_{k \in S_p}\varepsilon_{k}  \leq  b_{ji}x ,  \bigvee\limits_{p \in \{p_1,\ldots,p_n\}}\bar{d}_{ji}(p)=\bar{b}_{ji} , U_i =\tilde U_j \bar{b}_{ji} \Big) \\
       + & \ \mathbb P\Big(\bigvee\limits_{p \in P_{ji}} d_{ji}(p)\prod_{k \in S_p}\varepsilon_{k}  \leq  b_{ji}x , \bigvee\limits_{p \in P_{ji}\setminus\{p_1,\ldots,p_n\}}\bar{d}_{ji}(p) = \bar{b}_{ji} , U_i =\tilde U_j \bar{b}_{ji} \Big) \\
        =&\mathbb P\Big(\bigvee\limits_{p \in P_{ji}} d_{ji}(p)\prod_{k \in S_p}\varepsilon_{k}  \leq  b_{ji}x , \bigvee\limits_{p \in P_{ji}\setminus\{p_1,\ldots,p_n\}}\bar{d}_{ji}(p)\leq  \bigvee\limits_{p \in \{p_1,\ldots,p_n\}}\bar{d}_{ji}(p) , U_i =\tilde U_j \bar{b}_{ji} \Big) \\
       + & \ \mathbb P\Big(\bigvee\limits_{p \in P_{ji}} d_{ji}(p)\prod_{k \in S_p}\varepsilon_{k}  \leq  b_{ji}x , \bigvee\limits_{p \in P_{ji}\setminus\{p_1,\ldots,p_n\}}\bar{d}_{ji}(p)>  \bigvee\limits_{p \in \{p_1,\ldots,p_n\}}\bar{d}_{ji}(p) , U_i =\tilde U_j \bar{b}_{ji} \Big) \\
       =&: \tilde I_{11}(x)+\tilde I_{12}(x).\notag
    \end{align*}
    With the same arguments as in the proof of Theorem~\ref{epscount} we find upper and lower bounds for $\tilde I_{11}(x)$. Analogously to \eqref{I11_ext2} and \eqref{I11} we find
    \begin{align*}
        \tilde I_{11}(x)= &\mathbb P \Big( \bigcap\limits_{p \in P_{ji}\setminus \{p_1,\ldots,p_n\}} \Big\{ \prod_{k \in S_p}\varepsilon_k  \leq \frac{b_{ji}}{d_{ji}(p)} \bigvee\limits_{\tilde p \in \{p_1,\ldots,p_n\}} \prod_{k \in  S_{\tilde p}} \varepsilon_{k}  \Big\} \bigcap  \bigcap\limits_{p \in \{p_1,\ldots, p_n\}}\bigg\{\prod\limits_{k \in  S_{p} } \varepsilon_{k} \leq x \bigg\}   \\
    &   \bigcap \bigcap_{l \in \An(i)}\Big\{\bigcap\limits_{p \in P_{li}\setminus P_{lji}}\Big\{ d_{li}(p) \prod_{k \in  S_{p} \cup \{l\}} \varepsilon_{k}Z_l  \leq  b_{ji} \bigvee\limits_{\tilde p \in \{p_1,\ldots,p_n\}} \prod_{k \in  S_{\tilde p}} \varepsilon_{k} \bigvee_{l \in \An(j)} \bar{b}_{lj} Z_l  \Big\}\Big\} \Big) \\
    \geq &\mathbb P \Big( \bigcap\limits_{p \in P_{ji}\setminus \{p_1,\ldots,p_n\}} \Big\{ \prod_{k \in S_p\setminus(\cup_{i=1}^nS_{p_i})}\varepsilon_k  \leq \frac{b_{ji}}{d_{ji}(p)}   \Big\} \bigcap  \bigcap\limits_{p \in \{p_1,\ldots, p_n\}}\bigg\{\prod\limits_{k \in  S_{p} } \varepsilon_{k} \leq x \bigg\}   \\
    &   \bigcap \bigcap_{l \in \An(i)}\Big\{\bigcap\limits_{p \in P_{li}\setminus P_{lji}}\Big\{ d_{li}(p) \prod_{k \in  S_{p} \cup \{l\}\setminus(\cup_{i=1}^nS_{p_i})} \varepsilon_{k}Z_l  \leq  b_{ji}  \bigvee_{l \in \An(j)} \bar{b}_{lj} Z_l  \Big\}\Big\} \Big) \\
    = & c_1 \ \mathbb P \Big(  \bigcap\limits_{p \in \{p_1,\ldots, p_n\}}\Big\{\prod\limits_{k \in  S_{p} } \varepsilon_{k} \leq x \Big\}  \Big) 
    \end{align*}
    and analogously to \eqref{upperboundI11} we have 
        \begin{align*}
        \tilde I_{11}(x) \leq  &\mathbb P \Big( \bigcap\limits_{p \in P_{ji}\setminus \{p_1,\ldots,p_n\}} \Big\{ \prod_{k \in S_p\setminus(\cup_{i=1}^nS_{p_i})}\varepsilon_k  \leq \frac{b_{ji}}{d_{ji}(p)}x   \Big\} \bigcap  \bigcap\limits_{p \in \{p_1,\ldots, p_n\}}\bigg\{\prod\limits_{k \in  S_{p} } \varepsilon_{k} \leq x \bigg\}   \\
    &   \bigcap \bigcap_{l \in \An(i)}\Big\{\bigcap\limits_{p \in P_{li}\setminus P_{lji}}\Big\{ d_{li}(p) \prod_{k \in  S_{p} \cup \{l\}\setminus(\cup_{i=1}^nS_{p_i})} \varepsilon_{k}Z_l  \leq  b_{ji}x  \bigvee_{l \in \An(j)} \bar{b}_{lj} Z_l  \Big\}\Big\} \Big) \\
    =& c_2(x) \ \mathbb P \Big(  \bigcap\limits_{p \in \{p_1,\ldots, p_n\}}\Big\{\prod\limits_{k \in  S_{p} } \varepsilon_{k} \leq x \Big\}  \Big)
    \end{align*}
    With the same arguments as in the previous proof, we can show that $c_1 \in (0,1)$ and $c_2(x) \to c_1$ for $x \downarrow 1$ and $\tilde I_{12}(x)=o(\tilde I_{11}(x))$ and $I_{2}(x)=o(I_{1}(x))$. Hence, the result follows. If \mbox{$P_{ji}\setminus\{p_1,\ldots,p_n\}= \emptyset$} the result follows analogously. 
\halmos 

\medskip

\textsc{Proof of Corollary~\ref{epscount_cor} }
From Theorem \ref{epscount} we have as $x\downarrow 1$,
\begin{align*}
       \mathbb P\Big(\bigwedge\limits_{k=0}^n \frac{U^k_i}{U^k_j} \leq  b_{ji}x\Big)
       &=1-\Big(1-\mathbb P\Big( \frac{U_i}{U_j} \leq  b_{ji}x\Big)\Big)^n 
       = 1-\Big(1-c (1+ o(1)) \, \mathbb P\Big( \prod_{k \in S_{p}}\varepsilon_{k} \leq x\Big)\Big)^n \\
       &=1-\sum\limits_{k=0}^n\binom{n}{k}\Big(-c (1+ o(1))\, \mathbb P\Big( \prod_{k \in S_{p}}\varepsilon_{k} \leq x\Big)\Big)^k\sim c \, n \, \mathbb P\bigg(\prod\limits_{i=1}^{n} \varepsilon_{k_i} \leq x \bigg),
    \end{align*}
    where we have used 
    the binomial theorem and the fact that the summands for $k \geq 2$ are negligible when $n$ is fixed. 
\halmos 

\medskip

\textsc{Proof of Theorem~\ref{epscountmult}. }
    We give a proof for  $P_{ji}\setminus\{p_1\}\neq \emptyset$ and $P_{lm}\setminus\{p_2\}\neq \emptyset$, i.e. $p_1$ and $p_2$ are not the only paths from $j$ to $i$ and from $l$ to $m$, respectively. All other cases follow analogously.  By the law of total probability, we have
    \begin{align*}
        &\mathbb P\Big(\frac{U_i}{U_j} \leq  b_{ji}x_1, \frac{U_m}{U_l} \leq  b_{lm}x_2\Big) \\
        = &\mathbb P\Big(\prod_{k \in S_{p_1}} \varepsilon_{k} \leq x_1,\prod_{k \in S_{p_2}} \varepsilon_{k} \leq x_2,U_i=U_j b_{ji}\prod_{k \in S_{p_1}} \varepsilon_{k}, U_m=U_l b_{lm}\prod_{k \in S_{p_2}} \varepsilon_{k}\Big)\\
        + &\mathbb P\Big(\prod_{k \in S_{p_1}} \varepsilon_{k} \leq x_1,\frac{U_m}{U_l} \leq  b_{lm} x_2,U_i=U_j b_{ji}\prod_{k \in S_{p_1}} \varepsilon_{k}, U_m \neq U_l b_{lm}\prod_{k \in S_{p_2}} \varepsilon_{k}\Big)\\
        + &\mathbb P\Big(\frac{U_i}{U_j} \leq  b_{ji}x_1, \prod_{k \in S_{p_2}} \varepsilon_{k} \leq x_2,U_i \neq U_j b_{ji}\prod_{k \in S_{p_1}} \varepsilon_{k}, U_m=U_l b_{lm}\prod_{k \in S_{p_2}} \varepsilon_{k}\Big)\\
        + &\mathbb P\Big(\frac{U_i}{U_j} \leq  b_{ji}x_1,\frac{U_m}{U_l} \leq  b_{lm} x_2,U_i \neq U_j b_{ji}\prod_{k \in S_{p_1}} \varepsilon_{k}, U_m \neq U_l b_{lm}\prod_{k \in S_{p_2}} \varepsilon_{k}\Big) \\
        =: &I_1(x_1,x_2)+I_2(x_1,x_2)+I_3(x_1,x_2)+I_4(x_1,x_2)
    \end{align*}
    {We first consider $I_1(x_1,x_2)$. Observe that for $I_{11}(x)$ defined in \eqref{i1} we have by \eqref{I11:alt}
    \begin{align*}
       I_{11}(x)=\mathbb P\Big( \prod_{k \in S_{p_{\max}}}\varepsilon_{k}  \leq  x , U_i = U_j b_{ji} \prod_{k \in S_{p_{\max}}}\eps_k \Big) 
    \end{align*}
    and hence, $I_1(x_1,x_2)$ is the bivariate extension to $I_{11}(x)$. For this reason, we can follow the proof of Theorem~\ref{epscount} at \eqref{I11_ext}, we again find upper and lower bounds based on the decomposition}
    \begin{small}
    \begin{align*}
   & \Big\{\bigvee\limits_{p \in P_{ji}\setminus \{p_1\}} \prod_{k \in S_p}\varepsilon_k  \leq \frac{b_{ji}}{d_{ji}(p)}\prod_{k \in  S_{p_1}} \varepsilon_{k}  \Big\} \bigcap \Big\{\bigvee\limits_{p \in P_{lm}\setminus \{p_2\}} \prod_{k \in S_p}\varepsilon_k  \leq \frac{b_{lm}}{d_{lm}(p)}\prod_{k \in  S_{p_2}} \varepsilon_{k}  \Big\} \bigcap  \\
        &\Big\{ \prod_{k \in  S_{p_1}} \varepsilon_{k} \leq x_1 \Big\} \bigcap  \bigcap_{n \in \An(i)}\Big\{\bigcap\limits_{p \in P_{ni}\setminus P_{nji}}\Big\{ d_{ni}(p) \prod_{k \in  S_{p}\cup \{n\}} \varepsilon_{k}Z_n  \leq  b_{ji} \prod_{k \in  S_{p_1}} \varepsilon_{k} \bigvee_{n \in \An(j)} \bar{b}_{nj} Z_n  \Big\}\Big\} \bigcap \\
        &   \Big\{ \prod_{k \in  S_{p_2}} \varepsilon_{k} \leq x_2 \Big\}   \bigcap  \bigcap_{n \in \An(m)}\Big\{\bigcap\limits_{p \in P_{nm}\setminus P_{nlm}}\Big\{ d_{nm}(p) \prod_{k \in  S_{p}\cup \{n\}} \varepsilon_{k}Z_n  \leq  b_{lm} \prod_{k \in  S_{p_2}} \varepsilon_{k} \bigvee_{n \in \An(l)} \bar{b}_{nl} Z_n  \Big\}\Big\}.
    \end{align*}
    \end{small}
Now for three paths $p$, $p_1$ and $p_2$ and a node $i$ we denote 
$$S_{p + i \setminus p_1+p_2}:=(S_{ p}\cup\{i\})\setminus(S_{p_1}\cup S_{p_2})\quad\mbox{and}\quad 
S_{p \setminus p_1+p_2}:=S_{ p}\setminus(S_{p_1}\cup S_{p_2}).$$ 
On the set $\{ \prod_{k \in  S_{p_1}} \varepsilon_{k} \leq x_1\} \cap  \{ \prod_{k \in  S_{p_2}} \varepsilon_{k} \leq x_2 \}$ we have for $x_1,x_2>1$, since $\eps>1$,
\begin{align*}
    &\Big\{\bigvee\limits_{p \in P_{ji}\setminus \{p_1\}} \prod_{k \in S_p}\varepsilon_k  \leq \frac{b_{ji}}{d_{ji}(p)}\prod_{k \in  S_{p_1}} \varepsilon_{k}  \Big\} = \Big\{\bigvee\limits_{p \in P_{ji}\setminus \{p_1\}} \prod_{k \in S_p \setminus S_{p_1}}\varepsilon_k  \leq \frac{b_{ji}}{d_{ji}(p)}\prod_{k \in  S_{p_1} \setminus S_p} \varepsilon_{k}  \Big\} \\ 
    \supseteq  &\Big\{\bigvee\limits_{p \in P_{ji}\setminus \{p_1\}} \prod_{k \in S_p \setminus S_{p_1}}\varepsilon_k  \leq \frac{b_{ji}}{d_{ji}(p)}  \Big\} \supseteq  \Big\{\bigvee\limits_{p \in P_{ji}\setminus \{p_1\}} \prod_{k \in S_{ p \setminus p_1+p_2}}\varepsilon_k  \leq \frac{b_{ji}}{d_{ji}(p)x_2}  \Big\}
\end{align*}
as well as 
\begin{align*}
    &\bigcap_{n \in \An(i)}\Big\{\bigcap\limits_{p \in P_{ni}\setminus P_{nji}}\Big\{ d_{ni}(p) \prod_{k \in  S_{p}\cup \{n\}} \varepsilon_{k}Z_n  \leq  b_{ji} \prod_{k \in  S_{p_1}} \varepsilon_{k} \bigvee_{n \in \An(j)} \bar{b}_{nj} Z_n  \Big\}\Big\} \\
    =&\bigcap_{n \in \An(i)}\Big\{\bigcap\limits_{p \in P_{ni}\setminus P_{nji}}\Big\{ d_{ni}(p) \prod_{k \in  (S_{p}\cup \{n\})\setminus S_{p_1}} \varepsilon_{k}Z_n  \leq  b_{ji} \prod_{k \in  S_{p_1}\setminus S_p} \varepsilon_{k} \bigvee_{n \in \An(j)}\bar{b}_{nj} Z_n \Big\}\Big\} \\
    \supseteq&\bigcap_{n \in \An(i)}\Big\{\bigcap\limits_{p \in P_{ni}\setminus P_{nji}}\Big\{ d_{ni}(p) \prod_{k \in  (S_{p}\cup \{n\})\setminus S_{p_1}} \varepsilon_{k}Z_n  \leq  b_{ji}  \bigvee_{n \in \An(j)}\bigvee_{\tilde p \in P_{nj}} d_{nj}(\tilde p) \prod_{k \in  S_{\tilde p + n \setminus p_1+p_2}}\varepsilon_k  Z_n  \Big\}\Big\} \\
    \supseteq&\bigcap_{n \in \An(i)}\Big\{\bigcap\limits_{p \in P_{ni}\setminus P_{nji}}\Big\{ d_{ni}(p) \prod_{k \in   S_{ p + n \setminus p_1+p_2}} \varepsilon_{k}Z_n  \leq  \frac{b_{ji}}{x_2}  \bigvee_{n \in \An(j)}\bigvee_{\tilde p \in P_{nj}} d_{nj}(\tilde p) \prod_{k \in  S_{\tilde p + n \setminus p_1+p_2}}\varepsilon_k  Z_n  \Big\}\Big\}.
\end{align*}    
Therefore,
\begin{align*}
    &I_1(x_1,x_2) \geq \mathbb P \Big(\Big\{\bigvee\limits_{p \in P_{ji}\setminus \{p_1\}} \prod_{k \in S_{p \setminus p_1+p_2}}\varepsilon_k  \leq \frac{b_{ji}}{d_{ji}(p)x_2}  \Big\} \bigcap \Big\{\bigvee\limits_{p \in P_{lm}\setminus \{p_1\}} \prod_{k \in S_{p \setminus p_1+p_2}}\varepsilon_k  \leq \frac{b_{lm}}{d_{lm}(p)x_1}  \Big\}  \\
    &\bigcap \bigcap_{n \in \An(i)}\Big\{\bigcap\limits_{p \in P_{ni}\setminus P_{nji}}\Big\{ d_{ni}(p) \prod_{k \in  S_{p + n \setminus p_1+p_2}} \varepsilon_{k}Z_n  \leq  \frac{b_{ji}}{x_2}  \bigvee_{n \in \An(j)}\bigvee_{\tilde p \in P_{nj}} d_{nj}(\tilde p) \prod_{k \in  S_{\tilde p + n \setminus p_1+p_2}}\varepsilon_k  Z_n  \Big\}\Big\}  \\
    &\bigcap \bigcap_{n \in \An(m)}\Big\{\bigcap\limits_{p \in P_{nm}\setminus P_{nlm}}\Big\{ d_{nm}(p) \prod_{k \in  S_{p + n \setminus p_1+p_2}} \varepsilon_{k}Z_n  \leq  \frac{b_{lm}}{x_1}  \bigvee_{n \in \An(l)}\bigvee_{\tilde p \in P_{nl}} d_{nl}(\tilde p) \prod_{k \in  S_{\tilde p + n \setminus p_1+p_2}}\varepsilon_k  Z_n  \Big\}\Big\} \\
    & \bigcap \Big\{ \prod_{k \in  S_{p_1}} \varepsilon_{k} \leq x_1 \Big\} \bigcap \Big\{ \prod_{k \in  S_{p_2}} \varepsilon_{k} \leq x_2 \Big\} \Big) \\
    &=: c_3(x_1,x_2) \ \mathbb P \Big( \Big\{ \prod_{k \in  S_{p_1}} \varepsilon_{k} \leq x_1 \Big\} \bigcap \Big\{ \prod_{k \in  S_{p_2}} \varepsilon_{k} \leq x_2 \Big\} \Big).
\end{align*}
For an upper bound, observe that on $\{ \prod_{k \in  S_{p_1}} \varepsilon_{k} \leq x_1\} \cap  \{ \prod_{k \in  S_{p_2}} \varepsilon_{k} \leq x_2 \}$ we have
\begin{align*}
    &\Big\{\bigvee\limits_{p \in P_{ji}\setminus \{p_1\}} \prod_{k \in S_p}\varepsilon_k  \leq \frac{b_{ji}}{d_{ji}(p)}\prod_{k \in  S_{p_1}} \varepsilon_{k}  \Big\} = \Big\{\bigvee\limits_{p \in P_{ji}\setminus \{p_1\}} \prod_{k \in S_p \setminus S_{p_1}}\varepsilon_k  \leq \frac{b_{ji}}{d_{ji}(p)}\prod_{k \in  S_{p_1} \setminus S_p} \varepsilon_{k}  \Big\} \\ 
    \subseteq  &\Big\{\bigvee\limits_{p \in P_{ji}\setminus \{p_1\}} \prod_{k \in S_p \setminus S_{p_1}}\varepsilon_k  \leq \frac{b_{ji}x_1}{d_{ji}(p)}  \Big\} \subseteq  \Big\{\bigvee\limits_{p \in P_{ji}\setminus \{p_1\}} \prod_{k \in S_{p \setminus p_1+p_2}}\varepsilon_k  \leq \frac{b_{ji}x_1}{d_{ji}(p)}  \Big\}
\end{align*}
as well as 
\begin{align*}
    &\bigcap_{n \in \An(i)}\Big\{\bigcap\limits_{p \in P_{ni}\setminus P_{nji}}\Big\{ d_{ni}(p) \prod_{k \in  S_{p}\cup \{n\}} \varepsilon_{k}Z_n  \leq  b_{ji} \prod_{k \in  S_{p_1}} \varepsilon_{k} \bigvee_{n \in \An(j)} \bar{b}_{nj} Z_n  \Big\}\Big\} \\
    \subseteq&\bigcap_{n \in \An(i)}\Big\{\bigcap\limits_{p \in P_{ni}\setminus P_{nji}}\Big\{ d_{ni}(p) \prod_{k \in  (S_{p}\cup \{n\})\setminus S_{p_1}} \varepsilon_{k}Z_n  \leq  b_{ji} x_1 x_2 \bigvee_{n \in \An(j)}\bigvee_{\tilde p \in P_{nj}} d_{nj}(\tilde p) \prod_{k \in  S_{\tilde p + n \setminus p_1+p_2}}\varepsilon_k  Z_n  \Big\}\Big\} \\
    \subseteq&\bigcap_{n \in \An(i)}\Big\{\bigcap\limits_{p \in P_{ni}\setminus P_{nji}}\Big\{ d_{ni}(p) \prod_{k \in  S_{p + n \setminus p_1+p_2}} \varepsilon_{k}Z_n  \leq  b_{ji} x_1 x_2 \bigvee_{n \in \An(j)}\bigvee_{\tilde p \in P_{nj}} d_{nj}(\tilde p) \prod_{k \in  S_{\tilde p + n \setminus p_1+p_2}}\varepsilon_k  Z_n  \Big\}\Big\}.
\end{align*}       
For this reason, 
\begin{align*}
    &I_1(x_1,x_2) \leq c_4(x_1,x_2) \ \mathbb P \Big( \Big\{ \prod_{k \in  S_{p_1}} \varepsilon_{k} \leq x_1 \Big\} \bigcap \Big\{ \prod_{k \in  S_{p_2}} \varepsilon_{k} \leq x_2 \Big\} \Big),
\end{align*}
with 
\begin{align*}
    &c_4(x_1,x_2):= \mathbb P \Big(\Big\{\bigvee\limits_{p \in P_{ji}\setminus \{p_1\}} \prod_{k \in S_{p \setminus p_1+p_2}}\varepsilon_k  \leq \frac{b_{ji}x_1}{d_{ji}(p)}  \Big\} \bigcap \Big\{\bigvee\limits_{p \in P_{lm}\setminus \{p_1\}} \prod_{k \in S_{p \setminus p_1+p_2}}\varepsilon_k  \leq \frac{b_{lm}x_2}{d_{lm}(p)}  \Big\} \bigcap \\
    &\bigcap_{n \in \An(i)}\Big\{\bigcap\limits_{p \in P_{ni}\setminus P_{nji}}\Big\{ d_{ni}(p) \prod_{k \in  S_{p + n \setminus p_1+p_2}} \varepsilon_{k}Z_n  \leq  b_{ji} x_1 x_2  \bigvee_{n \in \An(j)}\bigvee_{\tilde p \in P_{nj}} d_{nj}(\tilde p) \prod_{k \in  S_{\tilde p + n \setminus p_1+p_2}}\varepsilon_k  Z_n  \Big\}\Big\} \bigcap \\
    &\bigcap_{n \in \An(m)}\Big\{\bigcap\limits_{p \in P_{nm}\setminus P_{nlm}}\Big\{ d_{nm}(p) \prod_{k \in  S_{p + n \setminus p_1+p_2}} \varepsilon_{k}Z_n  \leq  b_{lm} x_1 x_2  \bigvee_{n \in \An(l)}\bigvee_{\tilde p \in P_{nl}} d_{nl}(\tilde p) \prod_{k \in  S_{\tilde p + n \setminus p_1+p_2}}\varepsilon_k  Z_n  \Big\}\Big\} \Big).
\end{align*}
Since all random variables are continuous, $c_4(x_1,x_2)$ tends to $c_3(x_1,x_2)$ for $x_1,x_2\downarrow 1$, and 
\begin{align*}
    c_3(x_1,x_2) \leq c \leq c_4(x_1,x_2)
\end{align*}
with
\begin{align*}
    &c:= \mathbb P \Big(\Big\{\bigvee\limits_{p \in P_{ji}\setminus \{p_1\}} \prod_{k \in S_{p \setminus p_1+p_2}}\varepsilon_k  \leq \frac{b_{ji}}{d_{ji}(p)}  \Big\} \bigcap \Big\{\bigvee\limits_{p \in P_{lm}\setminus \{p_1\}} \prod_{k \in S_{p  \setminus p_1+p_2}}\varepsilon_k  \leq \frac{b_{lm}}{d_{lm}(p)}  \Big\} \bigcap \\
    &\bigcap_{n \in \An(i)}\Big\{\bigcap\limits_{p \in P_{ni}\setminus P_{nji}}\Big\{ d_{ni}(p) \prod_{k \in  S_{p + n \setminus p_1+p_2}} \varepsilon_{k}Z_n  \leq  b_{ji}   \bigvee_{n \in \An(j)}\bigvee_{\tilde p \in P_{nj}} d_{nj}(\tilde p) \prod_{k \in  S_{\tilde p + n \setminus p_1+p_2}}\varepsilon_k  Z_n  \Big\}\Big\} \bigcap \\
    &\bigcap_{n \in \An(m)}\Big\{\bigcap\limits_{p \in P_{nm}\setminus P_{nlm}}\Big\{ d_{nm}(p) \prod_{k \in  S_{p + n \setminus p_1+p_2}} \varepsilon_{k}Z_n  \leq  b_{lm}   \bigvee_{n \in \An(l)}\bigvee_{\tilde p \in P_{nl}} d_{nl}(\tilde p) \prod_{k \in  S_{\tilde p + n \setminus p_1+p_2}}\varepsilon_k  Z_n  \Big\}\Big\} \Big).
\end{align*}
Since $i \neq m$ we can use the same arguments as for $c_1$ in the proof of Theorem~\ref{epscount} to show that $c>0$.  
Therefore, we only need to show that $I_i(x_1,x_2)=o(I_1(x_1,x_2))$ for $i \in \{2,3,4\}$. 
It is obvious that $I_2(x_1,x_2)=o(I_1(x_1,x_2))$ implies the other two cases. Using the same arguments from the proof of Theorem~\ref{epscount} regarding $I_2(x)=o(I_1(x))$ and $I_{12}(x)=o(I_{11})(x)$, the result follows.
\halmos 

\section{Proofs of Section~4}\label{B}

\textsc{Proof of Proposition~\ref{a.s.convergence} }
     We first consider the convergence of the simple minimum ratio $\bigwedge_{k=1}^n (U_i^k/U_j^k)$. 
     For $j=i$ the result is obvious. Moreover, for $j \in \an(i)$ we have with Lemma~\ref{lemma1} d), for $x>0$,
    \begin{align*}
        \lim\limits_{n \to \infty} \mathbb{P} \Big(\Big \vert\bigwedge\limits_{k=1}^n \frac{U_i^k}{U_j^k} -b_{ji} \Big \vert > x\Big) = \lim\limits_{n \to \infty} \Big(\mathbb{P} \Big( \frac{U_i}{U_j} -b_{ji}  > x\Big)\Big)^n = 0,
    \end{align*}
    showing that the minimum ratio converges for $n \to \infty$ in probability to $b_{ji}$.
    Since $\bigwedge_{k=1}^n U_i^k/U_j^k$ is non-increasing, it converges almost surely. For $j \not\in \an(i)$ we have $b_{ji}=0$ and the same result holds. Therefore, the estimators \eqref{estimate2a}, \eqref{estimate2} and \eqref{minratios} converge almost surely. 
    Considering the inequality \eqref{ineq:estimator}
    for the estimator \eqref{estimate1}, this estimator as well converges almost surely. Finally using \eqref{maxel} and \eqref{maxel2}, the same also holds for the estimator \eqref{estimate3}.
\halmos 

\medskip

\textsc{Proof of Lemma~\ref{alg2_lemma} } 
    Assume first that the output $\hat{ \bB}$ is not idempotent. Then there exists an entry $\hat{ b}_{ji}$ in $\hat{ \bB}$ such that $\hat{ b}_{ji}<\hat{ b}_{jk}\hat{ b}_{ki}$. Therefore, since the input matrix $\check{\bB}$ is idempotent, we have $\hat{ b}_{ji}=0$ while $\hat{ b}_{jk}>0$ and $\hat{ b}_{ki}>0$. This is a contradiction to the if-condition on line~\ref{If-condition} in the algorithm. \\
    Now assume that there is an idempotent matrix $\bB'$ that preserves all values that are larger than $\delta_1$ but contains more zero entries. Then there is an entry $b'_{ji}$ such that $b'_{ji}=0$ while $\hat{ b}_{ji}>0$. Since $\hat{ b}_{ji}>0$ there must be some $k \in \{j+1,\ldots,i-1\}$ such that $\hat{ b}_{jk} \not \in S$ and $\hat{ b}_{ki} \not \in S$, otherwise we would have set $\hat{ b}_{ji}$ equal to zero. Because we sort pairs $(j,i)$ by distance and $(j,k)$ and $(k,i)$ both have smaller distance it must also hold that both, $\hat{ b}_{jk}$ and $\hat{ b}_{ki}$ are strictly greater than zero. In comparison, since $\bB'$ is idempotent, either $b'_{jk}$ or $b'_{ki}$ is equal to zero. \\
    Therefore, we have $b'_{jk}=0$ while $\hat{ b}_{jk}>0$ or $b'_{ki}=0$ while $\hat{ b}_{ki}>0$. In both cases, the distance compared to the pair $(j,i)$ is decreased. Repeating this argument we can assume that $(j-i)=1$. This, however, leads to a contradiction since $\hat{ b}_{ji}$ is set to zero for all pairs $(j,i)$ of distance one if $\check{b}_{ji}<\delta_1$. 
\halmos 

\section{Proofs of Section~5}\label{C}

We first provide some auxiliary results. Observe that for a random variable $Y \geq 0$ we have $Y\in RV_\al^0$, if and only if $1/Y\in RV_\al^\infty$.

\begin{proposition}[Karamata's Tauberian Theorem~1.7.1', \cite{BGT}] \label{lemma3.11}
    Let $U$ be a non-decreasing function on $\mathbb R$ with $U(x)=0$ for all $x<0$, and Laplace-Stieltjes transform $\hat U(s)=\int_{[0,\infty)}e^{-sx}\\ dU(x)< \infty$ for all large $s$.
    For $l \in RV_0^\infty$ and $c \geq 0, \rho\geq 0$, the following are equivalent
    \begin{alignat}{2}\label{LST}
        U(x)      & \sim cx^\rho l(1/x)/\Gamma(1+p), \qquad &   & x\downarrow 0 \notag\\
        \hat U(s) & \sim cs^{-\rho}l(s), \qquad             &   & s\to \infty.
    \end{alignat}
\end{proposition}

From this we obtain the following result. 

\begin{lemma} \label{r.v.addup}
   \begin{enumerate}
   \item
   Let  $X \in RV_{\alpha_1}^0, Y \in RV_{\alpha_2}^0$ be independent, then
       $ X+Y\in RV_{\alpha_1+\alpha_2}^0$,
   \item
   Let $X,Y\geq1$ be independent and such that $\tilde X=\ln(X)\in RV_{\alpha_1}^0,\tilde Y = \ln(Y)\in RV_{\alpha_2}^0$. Then
   $( X  Y -1) \in RV_{\alpha_1+\alpha_2}^0$.
   \end{enumerate}
\end{lemma}

\begin{proof}
(a) \, We use Proposition~\ref{lemma3.11} a) for $U$ being $F_X$ or $F_Y$, the distribution function of $X$ or $Y$, respectively. Since the Laplace-Stieltjes transforms 
   $\hat{F}_X(s)=\int_{[0,\infty)}e^{-sx}dF_X(x)\leq1$ and $\hat{F}_Y(s)=\int_{[0,\infty)}e^{-sx}dF_Y(x)\leq1$ for all $s\ge 0$, by Proposition~\ref{lemma3.11},
   they are both regularly varying at $\infty$ in the sense of \eqref{LST};
   i.e., $\hat{F}_X\in RV_{\alpha_1}^\infty,\hat{F}_Y\in RV_{\alpha_2}^\infty$.
   By independence, the convolution theorem for Laplace-Stieltjes transforms gives
       $\hat{F}_{X+Y}(s)=
        \hat{F}_{X}(s)\hat{F}_{Y}(s)$
    and, therefore, $\hat{F}_{X+Y}\in RV_{\alpha_1+\alpha_2}^\infty$.
    Applying again Proposition~\ref{lemma3.11} we find that $X_1+X_2\in RV_{\alpha_1+\alpha_2}^0$.\\
    (b) \, This follows from a Taylor expansion. 
\end{proof}


\textsc{Proof of Lemma~\ref{ratiorv} }
By Theorem~\ref{epscount} we get for $x \downarrow 1$,
\begin{align*}
    \lim_{t \downarrow 0}\frac{\mathbb P (\ln(U_i/U_j)-\ln(b_{ji}) \leq t x)}{\mathbb P (\ln(U_i/U_j)-\ln(b_{ji}) \leq t)} &= \lim_{t \downarrow 0}\frac{\mathbb P (U_i/U_j \leq b_{ji}\exp(tx))}{\mathbb P (U_i/U_j \leq b_{ji}\exp(t))} \\
     = \lim_{t \downarrow 0}\frac{ c \ \mathbb P  \Big(\prod_{k \in S_p}\varepsilon_k \leq \exp(tx)\Big)}{ c \ \mathbb P  \Big(\prod_{k \in S_p}\varepsilon_k \leq \exp(t)\Big)} 
    & =  \lim_{t \downarrow 0}\frac{ \ \mathbb P  \Big(\sum_{k \in S_p} \ln( \varepsilon_k) \leq tx\Big)}{ \ \mathbb P  \Big(\sum_{k \in S_p}\ln(\varepsilon_k) \leq t\Big)} =x^{\zeta(p)\alpha}
\end{align*}
for $\zeta(p)=|S_p|$ by Lemma~\ref{r.v.addup}~a) and the fact that $\ln(\varepsilon_k)\in RV_{\alpha}^0$.
\halmos 

\medskip

For the proof of Theorem~\ref{3.11} we need the following distribution family.

\begin{definition}
    A positive random variable $Y$ is \emph{Fr\'echet distributed with shape $\alpha>0$ and scale $s>0$} and we write $Y \sim$ Fr\'echet$(\alpha,s)$ if the distribution function of $Y$ is given by
    \begin{align*}
        \Phi_{\alpha,s}(x)=\exp\left(-\left(\frac{x}{s}\right)^{- \alpha}\right), \quad x > 0.
    \end{align*}
\end{definition}    

The proof of Theorem~\ref{3.11} is divided into a proof of the one-dimensional marginal limit distributions, followed by the proof of the multidimensional result. 
We start with the one-dimensional limits.

\begin{proposition} \label{3.9} 
    Let $\bU$ be a recursive ML vector with propagating noise on a DAG $\mathcal D$ as defined in \eqref{1stdef_1} and assume that the path $p:=[j\to \dots \to i]$ from $j$ to $i$ is generic. Assume further that $\tilde\eps=\ln(\varepsilon) \in RV_\alpha^0$.
    For the node set $S_{p}$ choose $a_n\sim F_{\sum_{k \in S_p}\tilde \varepsilon_k}^\leftarrow(1/n)$ as $\nto$. Let $\bU^1,\ldots,\bU^n $ be an iid sample from $\bU$. Then
    \begin{align*}
        \lim_{n\to\infty} \mathbb P \left( \frac1{ a_n b_{ji}}\Big(
        \bigwedge\limits_{k=1}^n U_i^k/ U_j^k-b_{ji}\Big)\leq x  \right) =  \Psi_{(\zeta(p)\alpha,c^{1/(\zeta(p)\alpha)})}(x), \quad x>0,
    \end{align*}
    for the same constant $c$ as in Theorem~\ref{epscount}, and $\Psi_{\alpha,s}$ denotes the Weibull distribution from Definition~\ref{defn:5.1}. 
\end{proposition} 

\begin{proof}
    Define $X:=\ln(U_i/U_j)-\ln(b_{ji})$ with distribution function $F_X$. 
    Then by Lemma~\ref{ratiorv}, $X \in RV^0_{\zeta(p)\alpha}$, which implies that 
    $1/X \in RV^{\infty}_{\zeta(p)\alpha}$. 
    Using e.g. Theorem~3.3.7 of \cite{EKM},
    for 
    \begin{align}\label{DAF2}
    a_{1/X}(n) &\sim F^\leftarrow_{1/X}(1-1/n)\sim 1/F^\leftarrow_{X}(1/n)\to\infty,\quad n\to\infty,
    \end{align}
     we get 
    \begin{align*}
        \lim_{n\to\infty} \mathbb P \Big(\bigvee_{k=1}^n \Big(\frac{1}{X}\Big)^k \leq a_{1/X}(n)x \Big)
        &= \lim_{n\to\infty} \mathbb P \Big( \bigwedge_{k=1}^nX^k \geq \frac{1}{a_{1/X}(n)x}  \Big)
        = \Phi_{\zeta(p)\alpha,1}(x), \quad x>0,
    \end{align*}
    which implies by the continuity of $X$,
    \begin{align}\label{DAF1}
       \lim_{n\to\infty} \mathbb P \Big( \bigwedge_{k=1}^nX^k \leq \frac{x}{a_{1/X}(n)}  \Big) = 1- \Phi_{\zeta(p)\alpha,1}(1/x),\quad x>0.
    \end{align}
   Choose now
    \begin{align}\label{28}
        \tilde a_n \sim F_X^\leftarrow(1/n) \sim 1/a_{1/X}(n)\downarrow 0, \quad n\to\infty.
    \end{align}
    Hence, we have with \eqref{DAF1} and \eqref{28} by Lemma~\ref{ratiorv},
    \begin{align}\label{mda}
        \lim_{n\to\infty} \mathbb P \Big(\frac1{ \tilde a_n}\Big(
        \bigwedge\limits_{k=1}^n\ln(U_i^k/ U_j^k)-\ln(b_{ji})\Big) \leq x \Big) = 1- \Phi_{\zeta(p)\alpha,1}(1/x),\quad x>0.
    \end{align}
    Recall from Theorem~\ref{epscount} and the regular variation of $ \tilde\eps$, that for the same $c$ as defined in Theorem~\ref{epscount} we have
    $$F_X(x)\sim c \ \mathbb P\Big(\sum_{k \in S_p}\wt\eps_k\le x\Big)\sim \mathbb P\Big(\sum_{k \in S_p}\wt\eps_k\le x  c^{1/(\zeta(p)\alpha)}\Big)= F_{\sum_{k \in S_p}\wt\eps_k}(x  c^{1/(\zeta(p)\alpha)}),\quad x\downarrow 0,$$ 
    which implies that $1/n\sim F_X(\tilde a_n)\sim F_{\sum_{k \in S_p}\wt\eps_k}(\tilde a_n c^{1/(\zeta(p)\alpha)}) $.
    For the generalized inverses this implies that 
    $$\tilde a_n\sim F_X^{\leftarrow}(1/n)\sim c^{-1/(\zeta(p)\alpha)} F^{\leftarrow}_{{\sum_{k \in S_p}\wt\eps_k}}(1/n) \sim c^{-1/(\zeta(p)\alpha)} a_n, \quad n\to\infty.$$
    From this we find
    \begin{align*}
         \mathbb P\Big( \frac1{ \tilde a_n}\Big(
        \bigwedge\limits_{k=1}^n\ln(U_i^k/ U_j^k)-\ln(b_{ji})\Big)\leq x \Big)
        = \mathbb P \Big(\bigwedge\limits_{k=1}^n U_i^k/ U_j^k\leq \exp( \tilde  a_nx)b_{ji} \Big),\quad x>0.
    \end{align*}
    A Taylor expansion around 0 yields $\exp( \tilde a_n x)=1+ \tilde a_n x(1+o(1))$ as $n\to\infty$, because $\tilde a_n\downarrow 0$.
    Since for $x>0$,
    \begin{align*}
        \lim_{n\to\infty}  \mathbb P\Big( \frac1{ \tilde a_nb_{ji}}\Big(
        \bigwedge\limits_{k=1}^nU_i^k/ U_j^k-b_{ji}\Big)\leq x \Big)
        = \lim_{n\to\infty}  \mathbb P \Big( \frac1{ a_n b_{ji}}\Big(
        \bigwedge\limits_{k=1}^nU_i^k/ U_j^k-b_{ji}\Big)\leq c^{1/(\zeta(p)\alpha)}x \Big),
    \end{align*}
    we obtain with \eqref{mda}  
    \begin{align*}
       &\lim_{n\to\infty} \mathbb P \Big( \frac1{ a_n b_{ji}}\Big(
        \bigwedge\limits_{k=1}^nU_i^k/ U_j^k-b_{ji}\Big)\leq x  \Big) = 1- \Phi_{\zeta(p)\alpha,c^{-1/(\zeta(p)\alpha)}}(1/x) 
        =\Psi_{\zeta(p)\alpha,c^{1/(\zeta(p)\alpha)}}(x),
    \end{align*}
    which proves the assertion.
\end{proof}

\medskip

Now we can prove Theorem~\ref{3.11}.

\medskip

\textsc{Proof of Theorem~\ref{3.11} }
As we shall find asymptotic independence of estimates between different node pairs, it suffices to prove the bivariate result.
 
We first simplify notation as follows.
Assume pairs of nodes $(j,i)\neq (l,m)$ and denote the generic paths $p_1=p_{ji}$ with node set $S_{p_1}$ and $p_2=p_{lm}$ with node set $S_{p_2}$, respectively. 
Further denote $a^{1}_n=a^{(ji)}_n$ and $a^{2}_n=a^{(lm)}_n$.

By Proposition~\ref{3.9} we have
\begin{align*}
        \lim_{n\to\infty} \mathbb P \Big( \frac1{ a_n^1 b_{ji}}\Big(
        \bigwedge\limits_{k=1}^n U_i^k/ U_j^k-b_{ji}\Big)\geq x_1  \Big) =  \Phi_{\zeta(p_1)\alpha,c_1^{-1/(\zeta(p_1)\alpha)}}(1/x_1)=\exp\Big(-\frac{x_1^{\zeta(p_1)\alpha}}{c_1}\Big), 
    \end{align*}
for $x_1>0$. Since 
$    \lim_{n \to \infty}(1-\frac{a}{n})^n=\exp(-a)$ 
for $a \in \mathbb R$, we get by independence of the ratios $U_i^k/U_j^k$ for $k=1,\dots,n$ as $n\to\infty$,
\begin{align} \label{asym_prob_1}
    \mathbb P \Big( \frac1{ a_n^1 b_{ji}}
         \Big(U_i/ U_j-b_{ji}\Big)\leq x_1  \Big)=\frac{x_1^{\zeta(p_1)\alpha}}{c_1n}(1 + o(1)), \quad x_1>0.
\end{align}
With the same argument, we have 
\begin{align} \label{asym_prob_2}
    \mathbb P \Big( \frac1{ a_n^2 b_{lm}}
         \Big(U_m/ U_l-b_{lm}\Big)\leq x_2  \Big)=\frac{x_2^{\zeta(p_2)\alpha}}{c_2n}(1 + o(1)), \quad x_2>0.
\end{align}
Therefore, we have on the one hand
\begin{align} \label{comp_1}
    &\lim_{n\to\infty} \mathbb P \Big( \frac1{ a_n^1 b_{ji}}
        \bigwedge\limits_{k=1}^n \Big(U_i^k/ U_j^k-b_{ji}\Big)\geq x_1  \Big)\mathbb P \Big( \frac1{ a_n^2 b_{lm}}
        \bigwedge\limits_{k=1}^n \Big(U_m^k/ U_l^k-b_{lm}\Big)\geq x_2  \Big) \notag \\ 
        =& \exp\Big(-\frac{x_1^{\zeta(p_1)\alpha}}{c_1}\Big)\exp\Big(-\frac{x_2^{\zeta(p_2)\alpha}}{c_2}\Big)=\exp\Big(-\frac{x_1^{\zeta(p_1)\alpha}}{c_1}-\frac{x_2^{\zeta(p_2)\alpha}}{c_2}\Big), \quad x_1,x_2>0,
\end{align}
whereas, on the other hand, we have by independence of the bivariate ratios $(U_i^k/U_j^k, U_m^k/U_l^k)$ for $k=1,\dots,n$,
\begin{align}\label{rev_prob}
    & \lim\limits_{n \to \infty}\mathbb P \Big( \frac1{ a_n^1 b_{ji}}
        \bigwedge\limits_{k=1}^n \Big(U_i^k/ U_j^k-b_{ji}\Big)\geq x_1, \frac1{ a_n^2 b_{lm}}
        \bigwedge\limits_{k=1}^n \Big(U_m^k/ U_l^k-b_{lm}\Big)\geq x_2  \Big) \notag  \\
        =& \lim\limits_{n \to \infty} \Big\{ \mathbb P \Big( \frac1{ a_n^1 b_{ji}}
        \Big(U_i^k/ U_j^k-b_{ji}\Big)\geq x_1, \frac1{ a_n^2 b_{lm}}
         \Big(U_m^k/ U_l^k-b_{lm}\Big)\geq x_2 \Big)\Big\}^n \notag  \\
         =&\lim\limits_{n \to \infty} \Big\{1-\mathbb P \Big( \frac1{ a_n^1 b_{ji}}
         \Big(U_i/ U_j-b_{ji}\Big)\leq x_1  \Big)-\mathbb P \Big( \frac1{ a_n^2 b_{lm}}
         \Big(U_m/ U_l-b_{lm}\Big)\leq x_2  \Big)+  \notag \\
         & \quad \mathbb P \Big( \frac1{ a_n^1 b_{ji}}
         \Big(U_i/ U_j-b_{ji}\Big)\leq x_1, \frac1{ a_n^2 b_{lm}}
         \Big(U_m/ U_l-b_{lm}\Big)\leq x_2  \Big)\Big\}^n.
\end{align}
By 
\eqref{eq:Ueps} we have 
$U_i \geq b_{ji} \prod_{k \in S_{p_1}}\varepsilon_k U_j$ 
and $U_m\geq b_{lm} \prod_{k \in S_{p_2}}\varepsilon_k U_l$,
which implies 
\begin{align}\label{comm_prob}
    &\mathbb P \Big( \frac1{ a_n^1 b_{ji}}
         \Big(U_i/ U_j-b_{ji}\Big)\leq x_1, \frac1{ a_n^2 b_{lm}}
         \Big(U_m/ U_l-b_{lm}\Big)\leq x_2  \Big) \\ 
         \leq  &\mathbb P \Big( \frac1{ a_n^1 }
         \Big( \prod_{k \in S_{p_1}}\varepsilon_k-1\Big)\leq x_1, \frac1{ a_n^2 }
         \Big( \prod_{k \in S_{p_2}}\varepsilon_k-1\Big)\leq x_2  \Big). \notag
\end{align}
Since $(j,i)\neq (l,m)$, either $S_{p_1}\setminus S_{p_2}\neq \emptyset$ or $S_{p_2}\setminus S_{p_1}\neq \emptyset$
and without loss of generality, we assume $S_{p_2}\setminus S_{p_1}\neq \emptyset$. 
Since $\varepsilon \geq 1$ and all $\varepsilon_k$ are independent, we get 
\begin{align}\label{comm_prob2}
    \eqref{comm_prob}
         \leq &\mathbb P \Big( \frac1{ a_n^1 }
         \Big( \prod_{k \in S_{p_1}} \varepsilon_k-1\Big)\leq x_1,   \frac1{ a_n^2 }
         \Big(  \prod_{k \in S_{p_2}\setminus S_{p_1}}\varepsilon_k-1\Big)\leq x_2 \Big)\notag \\
         = &\mathbb P \Big( \frac1{ a_n^1 }
         \Big( \prod_{k \in S_{p_1}} \varepsilon_k-1\Big)\leq x_1\Big) \mathbb P \Big( \frac1{ a_n^2 }
         \Big(  \prod_{k \in S_{p_2}\setminus S_{p_1}}\varepsilon_k-1\Big)\leq x_2  \Big). 
\end{align}
By Lemma~\ref{r.v.addup}~b) we know that $(\prod_{k \in S_{p_1}}\varepsilon_k-1)\in RV^0_{\zeta(p_1)\alpha}$. Moreover, observe that by a Taylor expansion we have $a_n^1 \sim F^{\leftarrow}_{\sum_{k\in S_p}\tilde\eps_k}(1/n) \sim F^{\leftarrow}_{\prod_{k\in S_p}\eps_k-1}(1/n)$ as $n\to\infty$.
Therefore, by Theorem~3.3.7 of \cite{EKM} as in the proof of Proposition~\ref{3.9}, similarly to \eqref{mda}, it holds that 
\begin{align*}
   \lim\limits_{n \to \infty} \mathbb P \Big( \frac1{ a_n^1 }
         \Big( \bigwedge\limits_{t=1}^n \prod_{k \in S_{p_1}} \varepsilon^t_k-1\Big)\geq x_1\Big)=\Phi_{\zeta(p_1)\alpha,1}(1/x_1)=\exp\Big(-x_1^{\zeta(p_1)\alpha}\Big),\quad x_1 >0.
\end{align*}
We proceed as in \eqref{asym_prob_1} and \eqref{asym_prob_2} to obtain as $n\to\infty$
\begin{align}\label{asym_prob_3}
    \mathbb P \Big( \frac1{ a_n^1 }
         \Big(  \prod_{k \in S_{p_1}} \varepsilon_k-1\Big)\leq x_1\Big)=\Big(\frac{x_1^{\zeta(p_1)\alpha}}{n}\Big)(1+o(1)), \quad x_1>0.
\end{align}
Moreover, {since $S_{p_2}\setminus S_{p_1}\neq \emptyset$, $\varepsilon$ is atom-free and $a_n^2 \to 0$ as $n \to \infty$,}
we have
\begin{align}\label{asymp_prob4}
    \mathbb P \Big( \frac1{ a_n^2 }
         \Big( \prod_{k \in S_{p_2}\setminus S_{p_1}}\varepsilon_k-1\Big)\leq x_2  \Big)=o(1), \quad x_2>0.
\end{align}
Therefore, we have by \eqref{comm_prob2}, \eqref{asym_prob_3} and \eqref{asymp_prob4}
\begin{align*}
    \eqref{comm_prob} = \Big(\frac{x_1^{\zeta(p_1)\alpha}}{n}\Big)(1+o(1))o(1), \quad x_1, x_2>0.
\end{align*}
Comparing this with \eqref{asym_prob_1} and \eqref{asym_prob_2}, we find that the last term in \eqref{rev_prob} is negligible.
Hence, we obtain 
\begin{align*}
&\lim\limits_{n \to \infty}\mathbb P \Big( \frac1{ a_n^1 b_{ji}}
        \bigwedge\limits_{k=1}^n \Big(U_i^k/ U_j^k-b_{ji}\Big)\geq x_1, \frac1{ a_n^2 b_{lm}}
        \bigwedge\limits_{k=1}^n \Big(U_m^k/ U_l^k-b_{lm}\Big)\geq x_2  \Big) \\
    =&\lim\limits_{n \to \infty} \Big(1-\frac{x_1^{\zeta(p_1)\alpha}}{c_1n}(1 + o(1))-\frac{x_2^{\zeta(p_2)\alpha}}{c_2n}(1 + o(1))\Big)^n
    =\exp\Big(-\frac{x_1^{\zeta(p_1)\alpha}}{c_1}-\frac{x_2^{\zeta(p_2)\alpha}}{c_2}\Big).
\end{align*}
Comparing this to \eqref{comp_1} yields the result.

\halmos

\end{document}